\documentclass{article}

\usepackage{amssymb,amsmath,amsthm}
\usepackage{color}
\usepackage{graphicx}
\usepackage{enumitem}
\usepackage{caption}
\usepackage[mathscr]{eucal}
\usepackage{multicol}
\usepackage{bm}
\usepackage{marginnote}
\usepackage{mathtools}
\usepackage[top=2cm,bottom=2cm]{geometry}
\usepackage{subfigure}
\usepackage{xcolor}
\usepackage{hyperref}
\hypersetup{
    colorlinks,
    linkcolor={blue!50!black},
    citecolor={green!50!black},
    urlcolor={red!80!black}
}

\usepackage[normalem]{ulem}

\usepackage[compress]{cite}

\providecommand{\E}[1]{{\ensuremath{\mathbb{E}}\mspace{-2mu}\left[#1\right]}}%
\providecommand{\work}[1]{{\ensuremath{\mathrm{Work}}\mspace{-2mu}\left[#1\right]}}%

\providecommand{\rset}{\mathbb{R}}

\providecommand{\nset}{\mathbb{N}}
\providecommand{\Order}[1]{ {\ensuremath{ \mathcal O\left( #1 \right)}} }

\renewcommand{\vec}[1]{{\ensuremath{{\boldsymbol #1}}}}
\providecommand{\valpha}{{\ensuremath{\vec{\alpha}}}}
\providecommand{\vbeta}{{\ensuremath{\vec{\beta}}}}

\renewcommand{\div}{{\ensuremath{ \text{div} }}}

\newcommand{\Leb}{\textnormal{Leb}}
\newcommand{\LebD}{\widetilde{\textnormal{Leb}}}
\newcommand{\maxLebD}{\mathbb{L}}
\newcommand{\qoi}{F}
\newcommand{\descqoi}[2]{\qoi_{#1,#2}}
\newcommand{\bb}{\vec b}
\newcommand{\ee}{\vec e}
\newcommand{\ii}{\vec i}
\newcommand{\jj}{\vec j}
\newcommand{\kk}{\vec k}
\renewcommand{\ll}{\vec l}
\newcommand{\pp}{\vec p}
\newcommand{\qq}{\vec q}

\newcommand{\vv}{\vec v}
\newcommand{\ww}{\vec w}
\newcommand{\xx}{\vec x}
\newcommand{\yy}{\vec y}
\newcommand{\zz}{\vec z}

\newcommand{\oone}{\bm{1}}
\newcommand{\eell}{\boldsymbol{\ell}}
\newcommand{\di}[1]{~\text{d}#1}
\newcommand{\ellrad}{\zeta}	%
\newcommand{\mcE}{\mathcal{E}}
\newcommand{\real}[1]{\mathfrak{Re}\left[ #1\right]\!}
\newcommand{\im}[1]{\mathfrak{Im}\left[ #1 \right]}
\newcommand{\norm}[2]{\left\| #1  \right\|_{#2} }
\newcommand{\pdf}{\varrho} 		%
\renewcommand{\div}{\textnormal{div}}
\newcommand{\finiteset}{\mathcal{G}}

\newcommand{\finitesuppset}{\mathfrak{L}}

\newcommand{\mesh}{\mathbb{T}}
\newcommand{\exmpld}{d}
\newcommand{\summab}{\eta}
\newcommand{\quadpointset}{\mathscr{P}}

\newcommand{\forcing}{\varsigma}

\newcommand{\bigtimes}{\times}

\allowdisplaybreaks[1] \title{Multi-index Stochastic Collocation convergence rates for random PDEs with parametric regularity}

\author{
  {\footnote{\texttt{(abdullateef.hajiali,
        raul.tempone)@kaust.edu.sa}, \texttt{fabio.nobile@epfl.ch},
      \texttt{tamellini@imati.cnr.it}}}
  Abdul-Lateef~Haji-Ali\footnote{CEMSE, King Abdullah University of
    Science and Technology (KAUST), Thuwal 23955-6900, Saudi Arabia}%
  \,, Fabio~Nobile\footnote{MATHICSE-CSQI, Ecole Polytechnique
    F\'ed\'erale de Lausanne, Station 8, CH 1015, Lausanne,
    Switzerland}%
  \,, Lorenzo Tamellini\footnote{Dipartimento di Matematica
    ``F. Casorati'', Universit\`a di Pavia, Via Ferrata 5, 27100
    Pavia, Italy}\, \footnote{CNR-IMATI, Via Ferrata 1, 27100, Pavia,
    Italy} \,, Ra\'{u}l~Tempone\footnotemark[3]%
}

\date{\today}

\theoremstyle{plain}

\newtheorem{theorem}{Theorem}
\newtheorem{lemma}[theorem]{Lemma}
\newtheorem{corollary}[theorem]{Corollary}
\newtheorem{remark}{Remark}
\newtheorem{example}{Example}

\newtheorem{assumption}{Assumption}

\InputIfFileExists{src/local}{}{}

\begin{document}
\maketitle

\pagestyle{myheadings}
\markboth{MISC convergence rates for random PDEs}
{MISC convergence rates for random PDEs}

\begin{abstract}
  We analyze the recent Multi-index Stochastic Collocation (MISC)
  method for computing statistics of the solution of a partial
  differential equation (PDEs) with random data, where the random
  coefficient is parametrized by means of a countable sequence of
  terms in a suitable expansion. MISC is a combination technique based
  on mixed differences of spatial approximations and quadratures over
  the space of random data and, naturally, the error analysis uses the
  joint regularity of the solution with respect to both the variables
  in the physical domain and parametric variables. In MISC, the number
  of problem solutions performed at each discretization level is not
  determined by balancing the spatial and stochastic components of the
  error, but rather by suitably extending the knapsack-problem
  approach employed in the construction of the quasi-optimal
  sparse-grids and Multi-index Monte Carlo methods.  We use a greedy
  optimization procedure to select the most effective mixed
  differences to include in the MISC estimator. We apply our
  theoretical estimates to a linear elliptic PDEs in which the
  log-diffusion coefficient is modeled as a random field,
  with a covariance similar to a Mat\'ern model, whose
  realizations have spatial regularity determined by a scalar
  parameter.
  We conduct a complexity analysis based on a summability argument
  showing algebraic rates of convergence with respect to the overall
  computational work. The rate of convergence depends on the
  smoothness parameter, the physical dimensionality and the efficiency of
  the linear solver. Numerical experiments show the effectiveness of
  MISC in this infinite-dimensional setting compared with the
  Multi-index Monte Carlo method and compare the convergence rate
  against the rates predicted in our theoretical analysis.


 \textbf{AMS class:}
   41A10 (approx by polynomials),
   65C20 (models, numerical methods),
   65N30 (Finite elements)
   65N05 (Finite differences)
\end{abstract}
\section{Introduction}
In this work, we analyze and apply the recent MISC method
  \cite{hajiali.eal:MISC1} to the approximation of quantities of
interest (outputs) from the solutions of linear elliptic partial
differential equations (PDEs) with random coefficients.  Such
equations arise in many applications in which the coefficients of the
PDE are described in terms of random variables/fields due either to a
lack of knowledge of the system or to its inherent non-predictability.
 We focus on the weak approximation of the
solution of the following linear elliptic $\yy$-parametric problem:
\begin{equation}\label{eq:PDE1}
  \begin{cases}
    - \div (a(\xx,\yy)\,\nabla u(\xx,\yy)) = \forcing(\xx) & \mbox{in}\quad \mathscr{B}  \\
    u(\xx,\yy) \,=\, 0 								& \mbox{on}\quad \partial \mathscr{B}.
  \end{cases}
\end{equation}
Here, $\mathscr{B} \subset \rset^d$ with $d \in \nset$ denotes the
  ``physical domain'', and the operators $\div$ and $\nabla$ act with
  respect to the physical variable, $\xx \in \mathscr{B}$, only.  We
  assume that $\mathscr{B}$ has a tensor structure, i.e.,
  $\mathscr{B} = \mathscr{B}_1 \times \mathscr{B}_2 \times \ldots
  \times \mathscr{B}_D$, with $D \in \nset$,
  $\mathscr{B}_i \subset \rset^{d_i}$ and $\sum_{i=1}^D d_i = d$, see,
  e.g., Figures~\ref{subf-Bsq} and~\ref{subf-cyl}.  This
  assumption simplifies the analysis detailed in the following,
  although MISC can be applied to more general domains, such as
  \begin{itemize}
  \item domains obtained by mapping from a reference tensor domain as
    in Figure~\ref{subf-lean-cyl}, by suitably extending the
    approaches in \cite{hughes:IGA,gordon.hall:transfinite.interp};
  \item non-tensor domains that can be immersed in a tensor bounding box,
    $\mathscr{B} \subset \hat{\mathscr{B}} = \hat{\mathscr{B}}_1 \times \hat{\mathscr{B}}_2 \times \ldots  \hat{\mathscr{B}}_D$,
    as in Figure~\ref{subf-nont}, whose mesh is obtained as a tensor product of meshes on each
    component, $\hat{\mathscr{B}}_i$, of the bounding box;
  \item domains that admit a structured mesh,
    i.e., with a regular connectivity, whose level of refinement
    in each ``direction'' can be set independently, as in Figure~\ref{subf-struct};
  \item domains that can be decomposed in patches satisfying any of
    the conditions above (observe that the meshes on each patch need
    not be conforming).
  \end{itemize}
\begin{figure}[tbp]
  \centering
  \subfigure[$\mathscr{B}_1, \mathscr{B}_2$ segments \label{subf-Bsq}]{
    \includegraphics[width=0.24\linewidth]{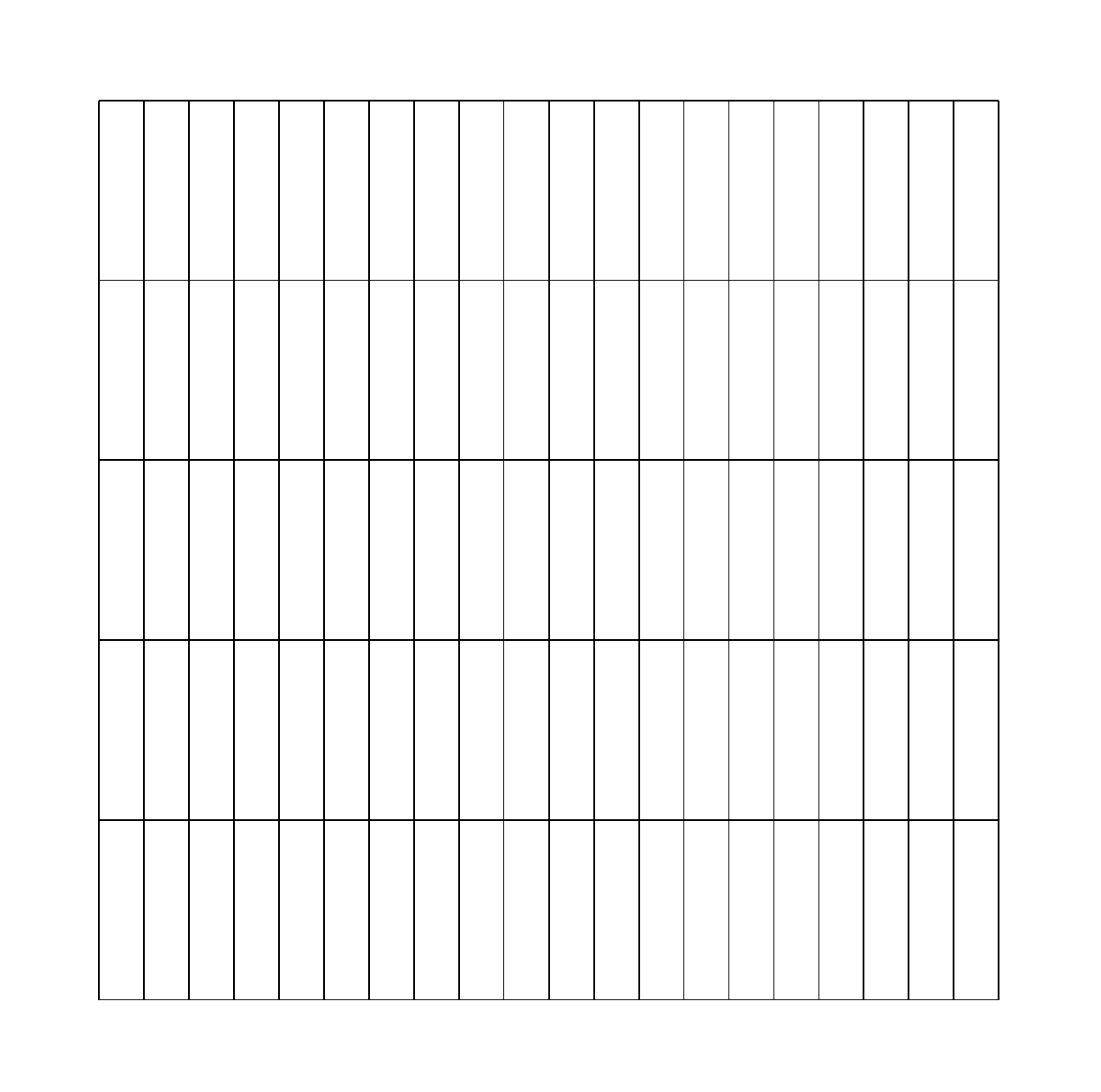}}
  \subfigure[$\mathscr{B}_1$ circle, $\mathscr{B}_2$ segment \label{subf-cyl}]{
    \includegraphics[width=0.24\linewidth]{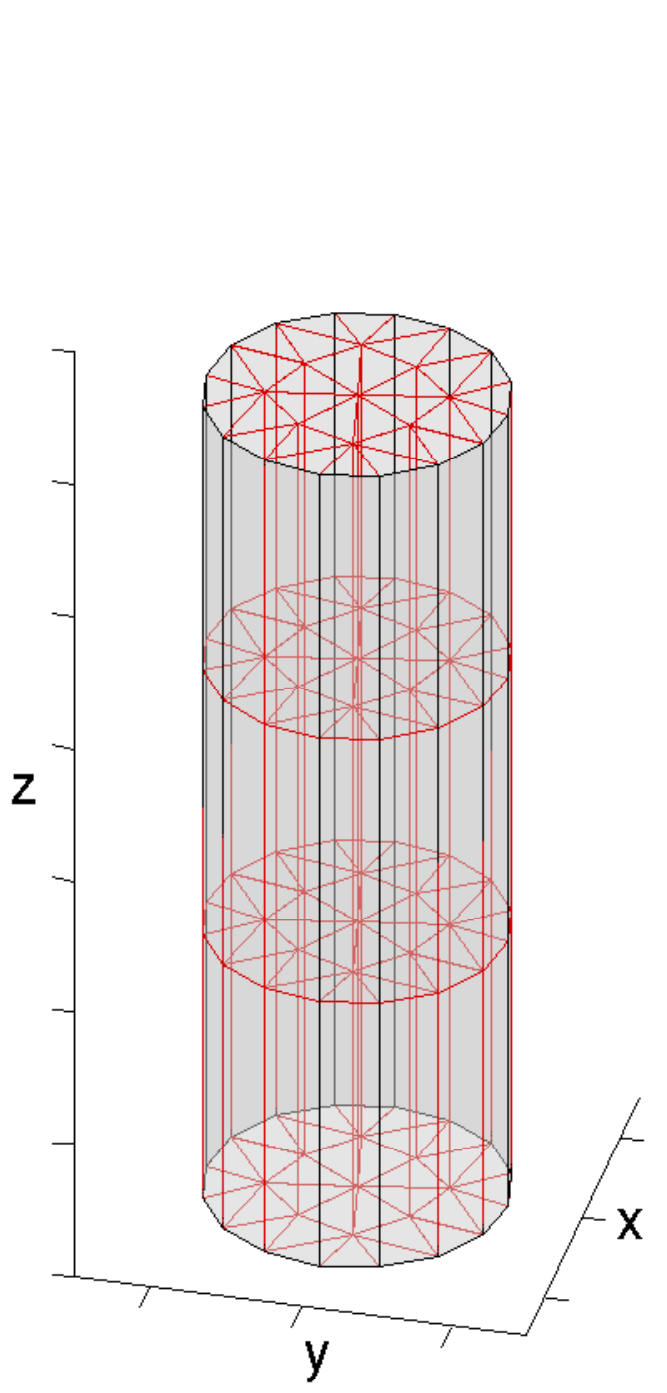}}
  \subfigure[$\mathscr{B}_1$ circle, $\mathscr{B}_2$ segment \label{subf-lean-cyl}]{
    \includegraphics[width=0.24\linewidth]{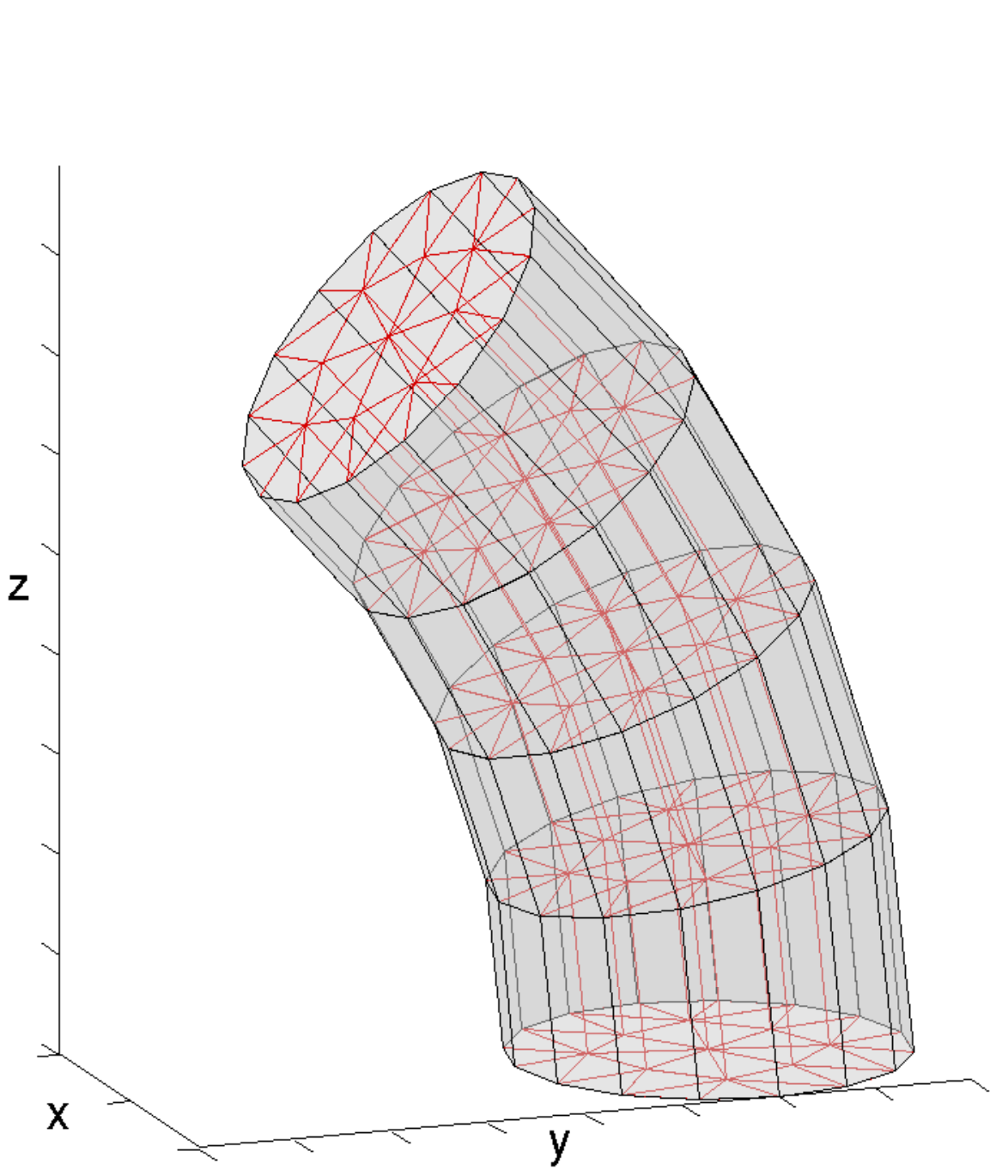}} \\
  \subfigure[non-tensor $\mathscr{B}$ that can be immersed in a tensor bounding box / divided into tensor subdomains \label{subf-nont}]{
    \includegraphics[width=0.24\linewidth]{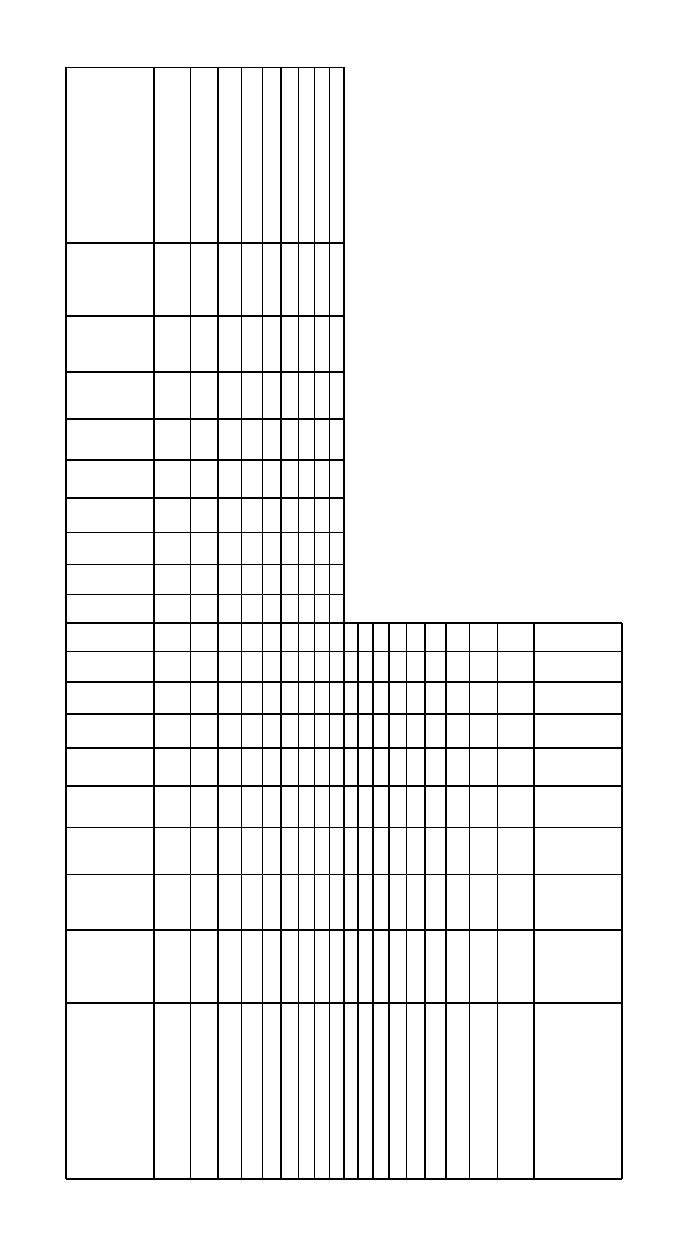}} \quad \quad
  \subfigure[non-tensor $\mathscr{B}$ with a structured mesh \label{subf-struct}]{
    \includegraphics[width=0.24\linewidth]{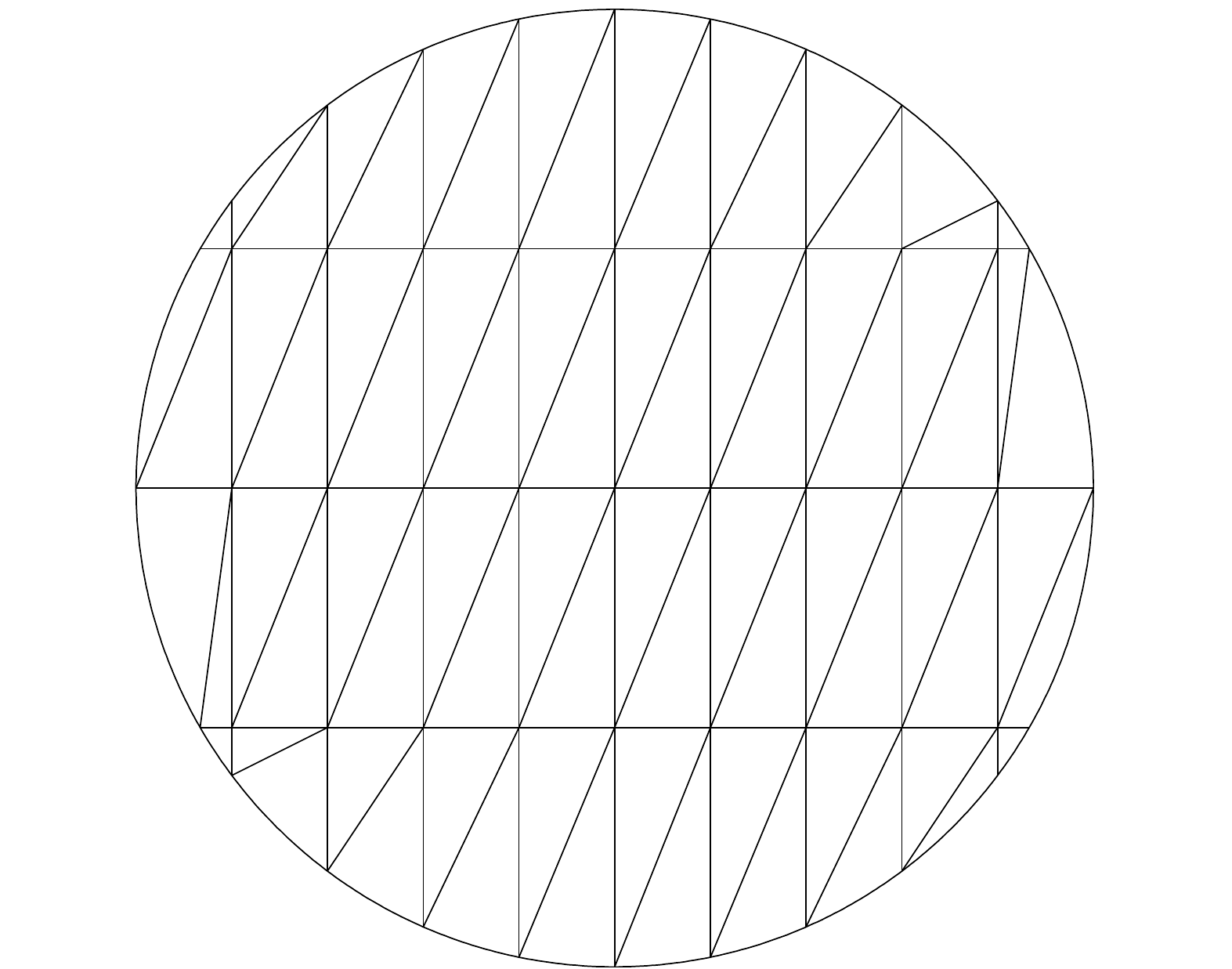}}
  \caption{Examples of physical domains on which MISC can be
      applied: (a) and (b) are within the framework
      of this work, while treating (c) requires the introduction of a mapping from~(b).
      MISC can also be formulated in non-tensor domains as in~(d) and~(e),
      but extending the analysis of the present work to this case is less straightforward and out of the scope
      of this work.}
  \label{fig:domains}
\end{figure}
The parameter $\yy=\{y_j\}_{j\ge 1}$ in \eqref{eq:PDE1} is a random
sequence whose components are independent and uniformly distributed
random variables.  More precisely, each $y_j$ has support in $[-1,1]$
with measure $\frac{\di{\lambda}}{2}$, where $\di{\lambda}$ is the
standard Lebesgue measure.  We further define
$\Gamma = \bigtimes_{ j\ge 1} [-1,1]$ (hereafter referred to as the
``stochastic domain'' or the ``parameter space''), with the
  cylindrical probability measure
$\di{\mu} = \mathop{\bigtimes}_{j \geq 1} \frac{\di{\lambda}}{2}$, (see,
  e.g., \cite[Chapter 3, Section 5]{bogachev2007:measure}).

The right-hand side of \eqref{eq:PDE1}, namely the deterministic
function $\forcing$, does not play a central role in this work and it
is assumed to be a smooth function of class
$C_0^{\infty}(\overline{\mathscr{B}})$, where
  $\overline{\mathscr{B}}$ denotes the closure of $\mathscr{B}$.
This regularity requirement can be relaxed, but we keep it to ease the
presentation, since our main goal in this work is to track the
  effect of the regularity of the coefficient $a$ in \eqref{eq:PDE1}
  on the MISC convergence rate.  Here, we focus on the
following family of diffusion coefficients:
\begin{equation}\label{eq:kappa_and_a}
  a(\xx,\yy) = e^{\kappa(\xx,\yy)}, \text{ with } \kappa(\xx,\yy) = \sum_{j \geq 1} \psi_j(\xx) y_j,
\end{equation}
where $\{\psi_j\}_{j \geq 1}$ is a sequence of functions
$\psi_j \in C^t(\overline{\mathscr{B}})$ for $t\geq 0$ such that
$\norm{\psi_j}{L^\infty(\mathscr{B})} \to 0$ as $j \to
\infty$. Hereafter, without loss of generality, we assume that the
sequence $ \{ \norm{\psi_j}{L^\infty(\mathscr{B})} \}_{j \geq 1}$ is
ordered in decreasing order. Thanks to a straightforward
application of the Lax-Milgram lemma, the well-posedness of
\eqref{eq:PDE1} in the classical Sobolev space,
$V=H^1_0(\mathscr{B})$, is guaranteed almost surely (a.s.) in $\Gamma$
if two functions, $a_{\min}, a_{\max}:\Gamma\to \rset$, exist such
that
\begin{equation}\label{eq:bounding_a}
0< a_{\min}(\yy) \leq a(\xx,\yy) \leq a_{\max}(\yy) <\infty,
\quad \forall \xx \in \mathscr{B}, \,\, \mbox{a.s. in }  \Gamma.
\end{equation}
Moreover, the equation is well posed in the Bochner space,
$L^q(\Gamma;V)$,  for some $q \geq 1$,%
\footnote{Recall that, given $q \geq 1$,
$L^q(\Gamma;V) = \left\{v : \Gamma \to V \mbox{ strongly measurable,
    such that }  \int_\Gamma \norm{u}{V}^q\di{\mu} < \infty \right\}$.}
(see \cite{babuska.nobile.eal:stochastic2,charrier:error_estimates} and
the following discussion), provided that sufficiently high moments
of the functions $1/a_{\min}$ and $a_{\max}$ are bounded.
The goal of our computation is the approximation of an expected value,
\[
\E{\qoi} = \E{\Theta(u)}\in \rset,
\]
where $\Theta$ is a deterministic bounded and linear functional,
and $\qoi(\yy) = \Theta(u(\cdot,\yy))$ is a real-valued random variable,
$F:\Gamma \rightarrow \rset$.  To this end, we utilize the
Multi-index Stochastic Collocation method (MISC), which we have
introduced in a general setting in a previous work
\cite{hajiali.eal:MISC1}.

In MISC, we consider a decomposition in terms of tensorized univariate
details (i.e., a tensorized hierarchical decomposition), for both the
discrete space in which \eqref{eq:PDE1} is solved for a fixed value of
$\yy \in \Gamma$ and for the quadrature operator used to approximate
the expected value of $\qoi$, relying on the well-established theory
of sparse-grid approximation of PDEs on the one hand
\cite{zenger91sparse,b.griebel:acta,Bungartz.Griebel.Roschke.ea:pointwise.conv,%
  Griebel.schneider.zenger:combination,Hegland:combination} and of
sparse-grid quadrature on the other hand
\cite{b.griebel:acta,babuska.nobile.eal:stochastic2,xiu.hesthaven:high,
nobile.tempone.eal:sparse,nobile.tempone.eal:aniso,gana.zabaras:sparse}.
We use tensor products of such univariate details, obtaining combined
deterministic-stochastic, first-order mixed differences to build the
MISC estimator of $\E{\qoi}$ by selecting the most effective mixed
differences with an optimization approach inspired by the literature
on the knapsack problem (see, e.g., \cite{martello:knapsack}). The
knapsack approach also was used in \cite{nobile.eal:optimal-sparse-grids}
to obtain the so-called quasi-optimal sparse grids for PDEs with
stochastic coefficients and in
\cite{b.griebel:acta,griebel.knapek:optimized} in the context of
sparse-grid resolution of high-dimensional PDEs.

The resulting method can be seen as an extension of the sparse-grid
combination technique for PDEs with stochastic coefficients, as well
as a fully sparse, non-randomized version of the Multilevel Monte
Carlo method
\cite{Heinrich:MLMC,giles:MLMC,scheichl.charrier:MLMC,bsz11}.  In
particular, MISC differs from other works in the literature that
attempt to optimally combine spatial and stochastic resolution levels
\cite{teckentrup.etal:MLSC,van-wyk:MLSC,hps13,kss12,bieri:sparse.tensor.coll} %
in two aspects. First, MISC uses combined deterministic-stochastic,
first-order differences, which allows us to exploit not only the
regularity of the solution with respect to the spatial variables and
the stochastic parameters, but also the mixed
deterministic-stochastic regularity whenever available.  Second, the
MISC estimator is built upon an optimization procedure, whereas the
above-mentioned works try to balance the error contributions arising
from the deterministic and stochastic components of the method without
taking into account the corresponding costs.  Finally, MISC can also be
seen as a sparse-grid quadrature version of the Multi-index Monte Carlo
method that was proposed and analyzed in \cite{abdullatif.etal:MultiIndexMC}.

In \cite{hajiali.eal:MISC1}, MISC was introduced in a general setting
and we restricted the analysis to the case of
problems of type \eqref{eq:PDE1} depending on a \emph{finite} number
of random variables, $\yy \in \Gamma \subset \rset^N$, with
  $N < \infty$. Here, we provide a complexity analysis of MISC in the
more challenging case in which the diffusion coefficient $a$
  depends on a countable sequence of random variables,
  $\{y_j\}_{j\geq 1}$. Furthermore, we aim at tracking the dependence
  of the MISC converge rate on the smoothness of the realizations of
  $a$.  This new framework requires that the tools used to prove the
complexity of the method be changed: whereas in
\cite{hajiali.eal:MISC1} we used a ``direct counting'' argument, i.e.,
we derived a complexity estimate by explicitly summing the work and
the error contributions associated with each mixed difference included
in the MISC estimator, here we base our proof on a summability
argument and on suitable interpolation estimates in mixed regularity
spaces. We mention that in \cite{griebel:infdim} an infinite
dimensional analysis based on a direct counting argument was recently
carried out in the case of hyperbolic cross-type index sets that might
arise when quasi-optimizing the work contribution of sparse grid
stochastic collocation without spatial discretization.

\bigskip
The rest of this work is organized as
follows. Section~\ref{s:functional_setting} introduces suitable
  assumptions and a class of random diffusion coefficients that we
consider throughout the work; functional analysis
results that are needed for the subsequent analysis of the MISC
method are also provided. The MISC method is reviewed in Section~\ref{s:MISC-method}.
A complexity analysis of MISC with an infinite number of random variables is
carried out in Section~\ref{s:err}, where we provide a general
convergence theorem. In Section~\ref{s:example_analysis}, we
discuss the application of MISC to the specific class of diffusion
coefficients that we consider here and track the
dependence of the convergence rate on the regularity of the diffusion
coefficient.  Section~\ref{s:numerics} presents some numerical tests
to verify the convergence analysis conducted in the previous
section. Finally, Section~\ref{s:conclusions} provides some conclusions
and final remarks.  In the Appendix, we include some
technical results on the summability and regularity properties of
certain random fields written in terms of their series expansion.

In the following, $\nset$ denotes the set of integer numbers including
zero, while $\nset_+$ denotes the set of positive integer numbers
excluding zero.  We refer to sequences in
$\nset^{\nset_+}$ and $\nset_+^{\nset_+}$ as ``multi-indices''.
Moreover, we often use a vector notation for sequences, i.e., we
formally treat sequences as vectors in $\nset^{\nset_+}$ (or
$\rset^{\nset_+}$) and mark them with bold type.  We employ the following notation, with the understanding that
$N < \infty$ for actual vectors and $N=\infty$ for sequences:
\begin{itemize}
  \item $\oone$ denotes a vector in $\nset^N$ whose components are all equal to one;
  \item $\vec e^N_\ell$ denotes the $\ell$-th canonical vector in
    $\rset^N$, i.e., $\left(\vec e^N_\ell\right)_i = 1$ if $\ell=i$
    and zero otherwise; however, for the sake of clarity, we often omit
    the superscript $N$ whenever it is obvious from context. For
    instance, if $\vv \in \rset^N$, we write $\vv - \ee_1$ instead of
    $\vv - \ee_1^N$;
  \item given $\vv \in \rset^N$, $|\vv| = \sum_{i=1}^N v_i$, $|\vv|_0$ denotes the number
    of non-zero components of $\vv$,
    $\max(\vv) = \max_{i=1,\ldots N}v_i$ and $\min(\vv) = \min_{i=1,\ldots N}v_i$;
  \item $\finitesuppset_+$ denotes the set of sequences with positive components
    with only finitely many elements larger than 1, i.e.,
    $\finitesuppset_+ = \{ \pp \in \nset_+^{\nset_+} : |\pp-\oone|_0 <\infty \}$;
  \item given $\vv \in \rset^N$ and $f:\rset \to \rset$,
    $f(\vv)$ denotes the vector obtained by applying $f$ to each component of $\vv$,
    $f(\vv) = [f(v_1),f(v_2),\cdots,f(v_N)] \in \rset^N$;
  \item given $\vv, \ww \in \rset^N$, the inequality $\vv > \ww$ holds true if and only if $v_i > w_i$ $\forall i=1,\ldots,N$;
  \item given $\vv \in \rset^D$ and $\ww \in \rset^N$, we denote their concatenation
    by $[\vv, \ww] = (v_1,\ldots,v_D,w_1,\ldots, w_N) \in
    \rset^{D+N}$;
  \item given a set with finite cardinality, $\finiteset \subset \nset_+$,
      we define the set $\nset^\finiteset = \{\zz\in\mathbb{\nset}^{\nset_+}: z_j=0 \,,\,\forall j\notin \finiteset\}$.
    We similarly define $\rset^\finiteset$ and $\mathbb{C}^\finiteset$.
\end{itemize}

\section{Functional setting}\label{s:functional_setting}

Even though condition \eqref{eq:bounding_a}  ensures well-posedness of \eqref{eq:PDE1} in $V$,
we need to make sure that realizations of $u$ a.s. belong to
more regular spaces to prove a convergence rate result for MISC. More
specifically, due to the classic spatial sparse-grid approximation
theory, we need certain conditions on the mixed derivatives of
$u$ with respect to the physical coordinates. To
this end, we introduce suitable functional spaces (tensor products of
fractional Sobolev spaces, see also \cite{griebel.harbrecht:tensor})
and then a ``shift'' regularity assumption (see
Assumption~\ref{assump:shift} below), i.e., we assume that the
regularity of the realizations of $u$ is induced ``in a natural
way'' by the regularity of $a$, of the forcing, $\forcing$, and of the
smoothness of the physical domain, $\mathscr{B}$. In other words, we
rule out ``pathological''/``ad hoc'' examples in which $u$ is very
regular despite the data is not, e.g., when the forcing is chosen
such that $u \in C^q$ for $q>2$ even in the presence of a domain with
corners.
First, recall the definition of a fractional Sobolev space for
$l_i \in \rset_+ \setminus \nset_+$ and $\mathscr{B}_i \subset \nset^{d_i}$:
\[
  H^{l_i}(\mathscr{B}_i) = \bigg\{ u \in H^{\lfloor l_i\rfloor}(\mathscr{B}_i) :
  \sup_{\valpha \in \nset^{d_i}, |\valpha| = \lfloor l_i \rfloor}
  \int_{\mathscr{B}_i} \int_{\mathscr{B}_i}
  \frac{|D^{\valpha} u(\xx)- D^{\valpha}u(\xx')|^{2}}{|\xx-\xx'|^{d_i+ 2(l_i -\lfloor l_i \rfloor)}}d\xx d\xx' < \infty \bigg\},
\]
extending the definition of a standard Sobolev space $H^{l_i}(\mathscr{B}_i)$ for $l_i$ integer.
The tensorized fractional Sobolev space can then be defined as
\[
  H^{\ll}(\mathscr{B}) = H^{l_1}(\mathscr{B}_1) \otimes \ldots \otimes
  H^{l_D}(\mathscr{B}_D)
\]
for $\ll = \left(l_i \right)_{i=1}^D \in \rset^{D}_+$.
\footnote{We recall that $H^{\ll}(\mathscr{B})$ is the completion
of formal sums $v=\sum_{k=1}^{K} v_{1,k}v_{2,k}\cdots v_{D,k}$ with
$v_{i,k} \in H^{l_i}(\mathscr{B}_i)$ with respect to the norm induced
by the inner product
\[
(v,w)_{H^{\ll}(\mathscr{B})} = \sum_{k,i}
(v_{1,k},w_{1,i})_{H^{l_1}(\mathscr{B}_1)}(v_{2,k},w_{2,i})_{H^{l_2}(\mathscr{B}_2)}\cdots
(v_{D,k},w_{D,i})_{H^{l_D}(\mathscr{B}_D)}.
\]}
Finally, the mixed fractional Sobolev spaces, that we will need for
our analysis, can be defined for each $\qq \in \rset_+^D$ as
\[
\mathcal{H}^{\oone+\qq}(\mathscr{B}) = \bigcap_{j=1}^D H^{\ee_j +\qq}(\mathscr{B}).
\]
Observe that, while mixed fractional spaces,
$\mathcal{H}^{\oone+\qq}(\mathscr{B})$, are the proper setting for the
forthcoming analysis, we will, for ease of presentation, not look for the most general mixed
space in which the solution lives.
Instead, we will be content with deducing mixed regularity from
inclusions in standard
Sobolev spaces, to the point that Assumption~\ref{assump:shift}
will be written in terms of standard Sobolev spaces.
For this, we observe that $\mathcal{H}^{\oone}(\mathscr{B}) = H^1(\mathscr{B})$ holds and in general we
have the following inclusion result between standard and mixed fractional Sobolev spaces:
\begin{equation}\label{eq:mixed_to_standard_H_inclusions}
  u\in H^{1+r}(\mathscr{B}) \Rightarrow u \in \mathcal{H}^{\oone+r\qq}(\mathscr{B})
  \quad \mbox{for } r \in (0,\infty) \mbox{ and }  0<|\qq| \leq 1.
\end{equation}

Before stating precisely the shift-regularity assumption on $u$,
we need some more notation and setup. First, observe that this
assumption needs to be stated in the complex domain, for reasons that
will be made clear later. We therefore extend the diffusion
coefficient from $a(\cdot,\yy)$ with $\yy \in \Gamma$ to $a(\cdot,\zz)$
with $\zz \in \mathbb{C}^{\nset_+}$, so that the corresponding solution of
\eqref{eq:PDE1}, $u(\cdot,\zz)$, becomes a $H^1_0(\mathscr{B})$
function taking values in $\mathbb{C}$, i.e.,
$u(\cdot,\zz)\in H^1_0(\mathscr{B},\mathbb{C})$.
Since the complex-valued version of problem \eqref{eq:PDE1} is well posed as
long as there exists $\delta$ such that $\real{a(\xx,\zz)} > \delta > 0$ for almost every (a.e.)
$\xx \in \mathscr{B}$,
and since our approximation method will cover the countable set
of parameters $\zz \in \mathbb{C}^{\mathbb{N}_+}$ by multiple subsets of finite cardinality,
we define the following region in
$\mathbb{C}^\finiteset$ for a set of finite cardinality, $\finiteset\subset\nset_+$:
\begin{equation}\label{eq:Sigma_def}
  \Sigma_{\finiteset,\delta} = \{  \zz \in \mathbb{C}^\finiteset : \real{a(\xx,\zz)}\geq \delta>0 \text{ for a.e. } \xx\in\mathscr{B} \}.
\end{equation}
We are now ready to state the assumption on the link between the
  regularity of the coefficient, $a$, and the regularity of solution, $u$.
\begin{assumption}[Shift assumption]\label{assump:shift}
  For a given  $\mathscr{B}$, let $\psi_j \in C^t(\overline{\mathscr B})$ (cf. eq. \eqref{eq:kappa_and_a})
  and $\forcing \in C^\infty_0(\overline{\mathscr{B}})$. We assume
  that there exist $r$ such that $1 < r < t$ and that, for any finite
  set $\finiteset \subset \nset_+$ %
  and any $\zz\in \Sigma_{\finiteset,\delta}$, the three following conditions
  hold:
\begin{enumerate}
\item $u(\cdot,\zz) \in H^{1+r}(\mathscr{B},\mathbb{C}) \cap H^1_0(\mathscr{B},\mathbb{C})$;%
\item $\frac{D u(\cdot,z_j)}{ D z_j} \in  H^{1 + r}(\mathscr{B},\mathbb{C}), \;\; \forall j\in\finiteset$,
  where $\frac{D u(\cdot,z_j)}{ D z_j}$ denotes the partial complex derivative of $u$;
\item $\norm{u(\cdot,\zz)}{H^{1+s}(\mathscr{B},\mathbb{C})} \leq
  C(\delta,s,\forcing,\mathscr{B})\norm{a(\cdot,\zz)}{C^s(\overline{\mathscr{B}},\mathbb{C})}$,
  with $C(\delta,s,\forcing,\mathscr{B})\to \infty$ for $\delta \to 0$, for every $s=1,\ldots,\lfloor r \rfloor$.%
\end{enumerate}
\end{assumption}

In the following, we will need to ensure that
$\norm{u(\cdot, \vec z)}{{H^{1+s}(\mathscr{B,\mathbb{C}})}}$, for $s > 0$, is uniformely
bounded for all $\vec z$ in certain subregions of the complex plane.  Note that this is
a stronger condition than what is stated in the previous assumption,
where we only assumed pointwise control on the norms of $u$ (i.e., we
gave a bound that depends on $\zz$).  In particular, we show at
the end of this section that the possibility of having such a uniform
bound depends on certain summability properties of the diffusion
coefficient.  Toward this end, we state the following assumption,
which also guarantees the well posedness of Problem
\eqref{eq:PDE1}.

\begin{assumption}[Summability of the diffusion coefficient]\label{assump:b_and_p}
For every $s = 0,1,\ldots,s_{\max} \leq r$, define the sequences
$\bb_{s}=\{b_{s,j}\}_{j \geq 1}$ where
\begin{align}
& b_{s,j} \, =\, \max_{\vec{s} \in \nset^d : |\vec s| \leq s}
  \norm{D^{\vec s} \psi_j}{L^\infty(\mathscr{B})} \;,\qquad j \geq 1.
 \label{eq:bbarj_def}
\end{align}
We assume that an increasing sequence $\{p_s\}_{s=0}^{s_{\max}}$
exists such that $0 < p_0 \leq \ldots \leq p_{s_{\max}}<\frac{1}{2}$
and $\bb_{s} \in \ell^{p_s}$, i.e.,
\begin{equation}\label{eq:p_barp_def}
  \norm{\bb_s}{\ell^{p_s}}^{p_s} = \sum_{j \geq 1} b_{s,j}^{p_s} < \infty.
\end{equation}
\end{assumption}
We observe that with the above assumption, $b_{s,j} \to 0^+$ as
$j \to \infty$ and $0\le b_{0,j}\le b_{s,j}$ for every $s=0,1,\ldots,s_{\max}$. Moreover, given
Assumption~\ref{assump:b_and_p}, we have that $\bb_0 \in \ell^1$,
which, together with the fact that $y_j \in [-1,1]$ for $j \geq 1$, guarantees that
condition \eqref{eq:bounding_a} holds and therefore that \eqref{eq:PDE1} is well posed in $V$ a.s. in
$\Gamma$. Incidentally, we observe that the conditions in Assumption~\ref{assump:b_and_p}
are sufficient but not necessary for condition \eqref{eq:bounding_a}
to hold: indeed, one would only need $\vec b_{s} \in \ell^2$ for some integer
$s \geq 1$, see Lemma~\ref{lem:Wa_inf_bound} (and
Corollary~\ref{cor:well_being} for a specific example) in the Appendix.

As suggested above, the fact that, for a fixed $s$, the sequence $\bb_s$
is $p_s$-summable plays a central role in this work. Indeed, if
$\bb_s$ is $p_s$-summable, we show that
$\norm{u(\cdot,\zz)}{H^{1+s}(\mathscr{B},\mathbb{C})}$ is
uniformely bounded with respect to $\zz$ in a region of the
complex plane whose size is proportional to
$\norm{\bb_s}{\ell^{p_s}}^{p_s}$. We use this fact to show convergence
of the MISC method, with the convergence rate dictated by both $p_0$
and $p_s$.  In Theorem~\ref{theo:MISC_conv}, we detail how to
  optimally choose the value of $s$ in the range
  $0,1,\ldots,s_{\max}$, which is the main result of this work.
Restricting the range of values of $s$ by
$p_s < \frac{1}{2}$ is not crucial; we could relax this to
$p_{s} < 1$. However, we follow this more stringent assumption because
it considerably simplifies some technical steps in the following
discussion without affecting the main part of the proof, as we make
clear below (see Remark~\ref{rem:keep_it_simple}).  What is
  important is that $s_{\max}$ might be
  strictly smaller than $\lfloor r \rfloor$ (i.e., it could happen
  that $\bb_r$ is $p_r$-summable but with $p_r > \frac{1}{2}$, or not
  summable at all); in this case, the line of proof we propose does
  not fully exploit the regularity of the solution, $u$.

\begin{example}\label{ex:logunif}
In the numerical section of this work, we consider either $\mathscr{B} = [0,1]$, i.e., $d=D=d_1=1$,
or $\mathscr{B} = [0,1]^3$, i.e., $d=D=3$, $d_i=1$ and $\mathscr{B}_i=[0,1]$ for $i=1,2,3$.
In both cases, we consider the following form for $\kappa(\xx,\yy)$:
\begin{equation}\label{eq:our_matern}
\kappa(\xx,\yy) =  \sum_{\vec k \in \nset^d} A_{\vec k}
\sum_{\vec \ell \in \{0,1\}^d} y_{\vec k,\vec \ell } \, \prod_{i=1}^d
\left(\cos\left({\pi }  k_i  x_i \right)\right)^{\ell_i}
\left(\sin\left({\pi }  k_i  x_i \right)\right)^{1-\ell_i}.
\end{equation}
Observe that it is possible to write $\kappa$ in the form \eqref{eq:PDE1} using a bijective mapping from
$\{y_{\vec k,\vec \ell}\}_{\vec k \in \nset^d, \vec \ell \in\{0,1\}^d}$ to $\{y_j\}_{j \geq 1}$.
We also choose the following values for the $A_{\vec k}$ coefficients:
\begin{equation}\label{eq:ansatzcoeff}
A_{\vec k} {= \left(\sqrt{3}\right)} 2^{\frac{|\vec k|_0}{2}}(1 + |\vec k|^2)^{-\frac{\nu+d/2}{2}},
\end{equation}
for some $\nu>0$. We observe that $\nu$ is a parameter dictating the
$\xx$-regularity of the realizations of $\kappa$, hence of $a$.
Moreover, the parameters $\nu$ and $d$ govern the $p_s$-summability of
the sequence $\bb_s$ for any $s$ and, as a consequence, the
overall convergence of the MISC method, as discussed earlier.
Section~\ref{s:example_analysis} analyzes the summability
properties of the series \eqref{eq:our_matern}.
\end{example}
We conclude this preliminary section by making the shape
of the above-mentioned regions more precise in the complex-plane and showing how
their sizes depend on the summability properties of $a$.
In %
particular, we will exploit the fact that for any finite set
$\finiteset \subset \nset_+$, for every $s=0,1,\ldots,s_{\max}$ and for
any $\zz\in\mathbb{C}^\finiteset$ we have
$\kappa(\cdot,\zz)\in C^{s}(\overline{\mathscr{B}},\mathbb{C})$,
$\norm{\kappa(\cdot,\zz)}{C^{s}(\overline{\mathscr{B}},\mathbb{C})}
\le \sum_{j\in\finiteset}|z_j| b_{s,j}$ and infer, from the
multivariate Fa\`a di Bruno formula (see
Appendix~\ref{sec:app-summability} and
\cite{constantine.savits:multivariate}), that
$a(\cdot,\zz)\in C^{s}(\overline{\mathscr{B}},\mathbb{C})$ as well,
with the estimate
\begin{equation}\label{eq:C_s_norm_of_a}
  \norm{a(\cdot,\zz)}{C^s(\overline{\mathscr{B}},\mathbb{C})} \leq
  \frac{s!}{(\log 2)^s}\norm{a(\cdot,\zz)}{C^0(\overline{\mathscr{B}},\mathbb{C})}(1+\norm{\kappa(\cdot,\zz)}{C^s(\overline{\mathscr{B}})})^s,
  \quad \forall \zz \in \mathbb{C}^\finiteset.
\end{equation}
Next, for a given  $\ellrad>1$, let $\mcE_{\ellrad}$ denote the polyellipse in the complex plane
\begin{align*}%
  \mcE_{\ellrad} = \bigg\{ z \in \mathbb{C} : \,\,
  \real{z} \leq \frac{\ellrad + \ellrad^{-1}}{2}\cos{\vartheta}, \,\,\,
  \im{z} \leq \frac{\ellrad - \ellrad^{-1}}{2}\sin{\vartheta},
  \,\, \vartheta \in [0, 2\pi) \bigg\}.
\end{align*}
For any sequence $\vec{\ellrad} = \{\ellrad_{j}\}_{j \geq 1}$ with $\ellrad_{j}>1$ for every $j \geq 1$
and for any finite set, $\finiteset \subset \nset_+$, %
we introduce the Bernstein polyellipse:
\begin{equation}\label{eq:bernstein_polyellipse}
\mcE_{\vec{\ellrad}}^\finiteset = \{ \zz \in \mathbb{C}^{\finiteset} :
z_j \in \mcE_{\ellrad_j} \mbox{ for all } j \in \finiteset\}.
\end{equation}
\begin{lemma}[Holomorphic complex continuation of $u$ in $H^1_0(\mathscr{B};\mathbb{C})$
  in a Bernstein polyellipse]\label{lemma:polyell_analyticity}
  Consider the sequence $\bb_0$ defined in
  \eqref{eq:bbarj_def}. For any $\delta>0$, let $E_\delta>2$ be such that
  \[
  \frac{\pi}{E_\delta} = -\norm{\bb_0}{\ell^1} -\log \delta + \log \cos \left( \frac{\pi}{E_\delta} \right),
  \]
  and consider the sequence $\vec{\ellrad}_0 = \{\ellrad_{0,j}\}_{j \geq 1}$, with %
  \begin{align}
    \ellrad_{0,j} & = \tau_{0,j} + \sqrt{\tau_{0,j}^2+1} > 1 \label{eq:ellrad_dep_on_b} \\
    \tau_{0,j}    & = \frac{\pi}{E_\delta} \frac{(b_{0,j})^{p_0-1}}{\norm{\bb_0}{\ell^{p_0}}^{p_0}}\,\,, \label{eq:taun_def_H1_norm}
  \end{align}
  with $p_0$ as in \eqref{eq:p_barp_def}.
  Then, for any finite set $\finiteset \subset \nset_+$,
  the solution, $u$, admits a holomorphic complex continuation,
  $u:\mathbb{C}^\finiteset \rightarrow H^{1}_0(\mathscr{B},\mathbb{C})$,
  in the Bernstein polyellipse, $\mcE_{\vec{\ellrad}_0}^\finiteset
  \subset \Sigma_{\finiteset,\delta}$,
  with \[\sup_{\zz \in \mcE_{\vec{\ellrad}_0}^\finiteset}
  \norm{u(\cdot,\zz)}{H^{1}(\mathscr{B})} \leq C_{0,u} = \frac{\norm{\forcing}{H^{-1}(\mathscr{B})}}{\delta} <\infty,\]
  with $C_{0,u}$ independent of $\finiteset$.
\end{lemma}
\begin{proof}
It is well known in the literature that $u:\mathbb{C}^\finiteset \rightarrow H^{1}_0(\mathscr{B},\mathbb{C})$
is holomorphic in the region $\Sigma_{\finiteset,\delta}$ defined in \eqref{eq:Sigma_def}
(see, e.g., \cite{babuska.nobile.eal:stochastic2}).
To compute the parameters $\{ \ellrad_j\}_{j \in \finiteset}$ of %
a Bernstein polyellipse contained in $\Sigma_{\finiteset,\delta}$, we rewrite $a(\xx,\zz)$ as
\begin{align*}
a(\xx,\zz)
& = \exp \left( \sum_{j \in \finiteset} z_j \psi_j(\xx) \right)
  = \exp \left( \sum_{j \in \finiteset} \real{z_j} \psi_j(\xx) \right) \exp \left( \sum_{j \in \finiteset} i \im{z_j}  \psi_j(\xx) \right) \\
& = \exp \left( \sum_{j \in \finiteset} \real{z_j}  \psi_j(\xx) \right)
	\left[ \cos \left(\sum_{j \in \finiteset} \im{z_j}  \psi_j(\xx) \right) + i \sin \left( \sum_{j \in \finiteset} \im{z_j}  \psi_j(\xx) \right) \right],
\end{align*}
so that $\Sigma_{\finiteset,\delta}$ can be rewritten as
\[
\Sigma_{\finiteset,\delta} = \left\{\zz \in \mathbb{C}^\finiteset :
\exp \left( \sum_{j \in \finiteset} \real{z_j}  \psi_j(\xx) \right) \cos \left(\sum_{j \in \finiteset} \im{z_j}  \psi_j(\xx) \right) \geq \delta \,\,
\mbox{ for a.e. } \xx \in \mathscr{B}  \right\}.
\]
Now, for some $E>2$ that we choose in the following, the two following conditions on $\zz$ imply that $\zz \in \Sigma_{\finiteset,\delta}$:
\[
  \begin{cases}
    \cos \Big( \sum_{j \in \finiteset} |\im{z_j}| \,b_{0,j} \Big) \geq \displaystyle \cos\left( \frac{\pi}{E} \right)\\[10pt]
    \exp \Big( -\sum_{j \in \finiteset} |\real{z_j}| \,b_{0,j} \Big) \geq \displaystyle  \frac{\delta}{\cos \left( \frac{\pi}{E} \right)};
  \end{cases}
\]
equivalently, we write
\[
  \begin{cases}
    \sum_{j \in \finiteset} |\im{z_j}\!| \,b_{0,j} \leq  \displaystyle \frac{\pi}{E}  \\[8pt]
    \sum_{j \in \finiteset} |\real{z_j}| \,b_{0,j} \leq \displaystyle  -\log \delta + \log \cos \left( \frac{\pi}{E} \right).
  \end{cases}
\]
For a fixed value of
$E$, the equations above define a second region,
$\Omega_{\finiteset,\delta}$, included in
$\Sigma_{\finiteset,\delta}$.  In turn, the previous conditions are
verified if the following conditions, which define a hyper-rectangular
region, $\mathcal{R}_\delta \subset
  \Omega_{\finiteset,\delta}$, are verified:
\[
\begin{cases}
  \displaystyle |\im{z_j}\!| \leq \tau_{0,j} = \frac{\pi (b_{0,j})^{p_0-1}}{E \norm{\bb_0}{\ell^{p_0}}^{p_0}}, \\
  \displaystyle |\real{z_j}| \leq 1+ w_{0,j},
  \quad \mbox{ with }
  w_{0,j} = \frac{(b_{0,j})^{p_0-1}}{ \norm{\bb_0}{\ell^{p_0}}^{p_0}} \left( -\norm{\bb_0}{\ell^1} -\log \delta + \log \cos \left( \frac{\pi}{E} \right) \right),
\end{cases}
\]
provided that $\delta$ and $E$ are such that the quantity
$-\norm{\bb_0}{\ell^1} -\log \delta + \log \cos \left( \frac{\pi}{E} \right)$ remains positive.
Observe that for sufficiently small $\delta>0$ such $E$ exists, since $f(E) = \log \cos \left( \frac{\pi}{E} \right)$ is
a monotonically increasing function, with $f(E) \to -\infty$ for $E \to 2$ and $f(E) \to 0$ for $E \to \infty$,
and $-\log \delta$ is positive for sufficiently small $\delta$. %
In particular, for any $\delta>0$, we choose $E=E_\delta$ such that $w_{0,j} = \tau_{0,j}$, which leads to
\[
\frac{\pi}{E_\delta} = -\norm{\bb_0}{\ell^1} -\log \delta + \log \cos \left( \frac{\pi}{E_\delta} \right).
\]
We observe that with this choice, $\tau_{0,j}$ (and hence $w_{0,j}$)
actually does not depend on $\finiteset$, and that we can define the sequence
$\vec{\tau}_0 = \{\tau_{0,j}\}_{j \geq 1}$.

We are now in the position to compute the Bernstein polyellipses that touch the boundary of $\mathcal{R}_\delta$ on the
real and imaginary axes. For the real axis, we have to enforce %
\[
\frac{\ellrad_{j,\text{real}}+\ellrad_{j,\text{real}}^{-1}}{2} = 1+\tau_{0,j} \Rightarrow \ellrad_{j,\text{real}} = 1 + \tau_{0,j} + \sqrt{(1+\tau_{0,j})^2-1},
\]
while for the imaginary axis we have to enforce
\[
\frac{\ellrad_{j,\text{imag}}-\ellrad_{j,\text{imag}}^{-1}}{2} = \tau_{0,j} \Rightarrow  \ellrad_{j,\text{imag}} = \tau_{0,j} + \sqrt{\tau_{0,j}^2+1}\,.
\]
The proof is concluded by observing that $\ellrad_{j,\text{imag}} \leq \ellrad_{j,\text{real}}$, i.e.,
the only polyellipse entirely contained in $\mathcal{R}_\delta$, and hence in $\Sigma_{\finiteset,\delta}$,
is the one touching $\mathcal{R}_\delta$ on the imaginary axis, which also implies that
the bound
$\sup_{\zz \in \mcE_{\vec{\ellrad}}^\finiteset}
  \norm{u(\cdot,\zz)}{H^{1}(\mathscr{B})} \leq C_{0,u} = \frac{\norm{\forcing}{H^{-1}(\mathscr{B})}}{\delta} <\infty$
holds independently of $\finiteset$.
\end{proof}

\begin{lemma}[Holomorphic complex continuation of $u$ in $H^{1+s}(\mathscr{B};\mathbb{C})$
  in a Bernstein polyellipse]\label{lemma:polyell_analyticity_high-order}
  For a given $s=1,2,\ldots,s_{\max}$,
  let $\vec{\ellrad}_s = \{\ellrad_{s,j}\}_{j \geq 1}$, with
  \begin{align}
    \ellrad_{s,j} & = \tau_{s,j} + \sqrt{\tau_{s,j}^2+1} > 1,  \label{eq:ellrad_dep_on_b_s} \\
    \tau_{s,j}    & = \frac{\pi (b_{s,j})^{p_s-1}}{E_\delta \norm{\bb_s}{\ell^{p_s}}^{p_s}},
    \label{eq:taun_def_H1_norm_s}
  \end{align}
  with $\bb_s$ as in \eqref{eq:bbarj_def}, $p_s$ as in \eqref{eq:p_barp_def},
  and $E_\delta$ as in Lemma~\ref{lemma:polyell_analyticity}.
  For any finite set  $\finiteset \subset \nset_+$, %
  $u:\mathbb{C}^\finiteset \rightarrow H^{1+s}(\mathscr{B},\mathbb{C})$ is holomorphic
  in the Bernstein polyellipse $\mcE_{\vec{\ellrad}_s}^\finiteset\subset \Sigma_{\finiteset,\delta}$,

  \begin{equation}\label{eq:supz-u-H-1-s-bound}
  \sup_{\zz \in \mcE_{\vec{\ellrad}_s}} \norm{u(\cdot,\zz)}{H^{1+s}(\mathscr{B})}
  \leq C_{s,u} = C(\tilde{\delta},s, \forcing, \mathscr{B}) M
  < \infty
\end{equation}
  with $M = \frac{s!}{(\log 2)^s} e^{K  \frac{\pi}{E_\delta}} \left( 1 + K \frac{\pi}{E_\delta} \right)^s$,
  $K= \left(2 + \frac{1}{\min_{j \geq 1} \tau_{s,j}^2}\right)^{1/2}$,
  $\tilde{\delta}= e^{- K \frac{\pi}{E_\delta}}$, $C(\tilde{\delta},s, \forcing, \mathscr{B})$ as in Assumption~\ref{assump:shift},
  and $C_{s,u}$ independent of $\finiteset$.
\end{lemma}
\begin{proof}
From Assumption~\ref{assump:shift}, $u : \mathbb{C}^\finiteset \to H^{1+s}(\mathscr{B},\mathbb{C})$
is complex differentiable for every $\zz$ in $\Sigma_{\finiteset,\varepsilon}$ for any $\varepsilon > 0$.
It is therefore holomorphic in $\Sigma_{\finiteset,\varepsilon}$.
Similarly to the previous lemma, we look for a region in which we have
an a-priori bound on the $H^{1+s}(\mathscr{B},\mathbb{C})$ norm of $u$ uniformly on $\zz$.
Again from Assumption~\ref{assump:shift}, we have that this is true in the region
\[
\Xi_{\finiteset,\varepsilon}(M) = \{ \zz \in \mathbb{C}^\finiteset : \norm{a(\cdot,\zz)}{C^s(\overline{\mathscr{B}})} \leq M \}
\cap \Sigma_{\finiteset,\varepsilon} \quad \mbox{ for any } \varepsilon>0.
\]
However, contrary to the previous lemma, in this proof, we do not
derive the expression of a polyellipse contained in
$\Xi_{\finiteset,\varepsilon}(M)$, but content ourselves with verifying
that the polyellipses,
$\mathcal{E}_{\finiteset,\vec{\ellrad}_s}$, proposed in the statement
of the lemma (that we have obtained simply by replacing
$b_{0,j}$ with
$b_{s,j}$ in \eqref{eq:taun_def_H1_norm}) satisfy the requirement,
i.e., $\mathcal{E}_{\finiteset,\vec{\ellrad}_s} \subset
\Xi_{\finiteset,\tilde{\delta}}(M)$, for every finite set, $\finiteset
\subset \nset_+$, and for a certain
$\tilde{\delta}$ that we specify later to control the coercivity of
the problem.  To this end, let us consider the univariate polyellipse
$\mathcal{E}_{\ellrad_{s,j}}$.  We first prove that this polyellipse
is contained in the following complex rectangle:
\[
\mathcal{R}_j = \{ z \in \mathbb{C} : |\real{z}| \leq \sqrt{1+\tau_{s,j}^2}, \,\, |\im{z}| \leq \tau_{s,j}  \}.
\]
The bound on the imaginary part of
$z$ is a  consequence of the choice of the polyellipse in
\eqref{eq:bernstein_polyellipse} and
\eqref{eq:ellrad_dep_on_b_s}. For the real part, we compute the point
$z_0$ where the polyellipse intersects the real axis by equating
\begin{align*}
  z_0 =
  & \frac{\ellrad_{s,j} + \frac{1}{\ellrad_{s,j}}}{2}
  = \frac{\ellrad_{s,j}^2 + 1}{2 \ellrad_{s,j}}
  = \frac{\tau_{s,j}^2 + 1 + \tau_{s,j} \sqrt{\tau_{s,j}^2+1}}{\tau_{s,j} +\sqrt{\tau_{s,j}^2+1}} \\
  & = \left( \tau_{s,j}^2+1 + \tau_{s,j} \sqrt{\tau_{s,j}^2+1} \right) \left( \sqrt{\tau_{s,j}^2+1} - \tau_{s,j} \right)
  = \sqrt{1+\tau_{s,j}^2}.
\end{align*}
Furthermore, we observe that
$|z| \leq \sqrt{1 + 2 \tau_{s,j}^2} \leq K \tau_{s,j}$
for every $z \in \mathcal{R}_j$ and some $K>0$;
for instance, we could look for the smallest $\tau_{s,j}$, say
$\tau_{s,j^*}$, choose $K$ accordingly, i.e., such that
$(K^2-2)\tau_{s,j^*}^2 \geq 1$, and obtain the value in the statement
of the lemma.
Next, according to \eqref{eq:C_s_norm_of_a} and Assumption~\ref{assump:b_and_p},
\[
\norm{a(\cdot,\zz)}{C^s(\overline{\mathscr{B}},\mathbb{C})}
\leq  \frac{s!}{(\log 2)^s}\norm{a(\cdot,\zz)}{C^0(\overline{\mathscr{B}},\mathbb{C})}(1+\norm{\kappa(\cdot,\zz)}{C^s(\overline{\mathscr{B}},\mathbb{C})})^s
\leq \frac{s!}{(\log 2)^s} e^{\sum_{j \in \finiteset} b_{0,j} |z_j|} \left(1+\sum_{j \in \finiteset} b_{s,j} |z_j|\right)^s
\]
holds.
We finish the proof by observing that for every, $\zz \in \mathcal{E}_{\finiteset,\vec{\ellrad}_s}$, we have
\[
\sum_{j \in \finiteset} b_{0,j} |z_j|
\leq \sum_{j \in \finiteset} b_{s,j} |z_j|
\leq K \sum_{j \in \finiteset} b_{s,j} \tau_{s,j}
=  K  \frac{\pi}{E_\delta} \sum_{j \in \finiteset} b_{s,j} \frac{(b_{s,j})^{p_s-1}}{\norm{\bb_s}{\ell^{p_s}}^{p_s}}
\leq K  \frac{\pi}{E_\delta},
\]
which gives uniform control of the norm of $\norm{a(\cdot,\zz)}{C^s(\overline{\mathscr{B}},\mathbb{C})}$ within
$\mcE_{\vec{\ellrad}_s}^\finiteset$ as required. More precisely, we have
\[
\norm{a(\cdot,\zz)}{C^s(\overline{\mathscr{B}},\mathbb{C})}
\leq M = \frac{s!}{(\log 2)^s} e^{K  \frac{\pi}{E_\delta}} \left( 1 + K \frac{\pi}{E_\delta} \right)^s,
\quad \forall \zz \in \mathcal{E}_{\finiteset,\vec{\ellrad}_s},
\]
which together with Assumption~\ref{assump:shift} gives the desired
bound on
$\norm{u(\cdot,\zz)}{H^{1+s}(\mathscr{B})}$ in \eqref{eq:supz-u-H-1-s-bound} and
\[
\real{a(\xx,\zz)} \geq e^{- K \frac{\pi}{E_\delta}} =: \tilde{\delta} > 0.
\]

\end{proof}
The following result from
  \cite{hajiali.eal:MISC1,nobile.eal:optimal-sparse-grids} is also needed.
Since this result is concerned with the finite-dimensional case, i.e., $\finiteset=\{1,2,\ldots,N\}$
and $\vec{\ellrad} \in \rset^N$, we write, for ease of notation, $\mcE_{\vec{\ellrad}}$ instead of $\mcE_{\vec{\ellrad}}^\finiteset$,
i.e., $\mcE_{\vec{\ellrad}} = \{ \zz \in \mathbb{C}^N : z_j \in \mcE_{\ellrad_j} \mbox{ for } j=1,2,\ldots,N\}$.

\begin{lemma}[Chebyshev expansion of a holomorphic function]\label{lemma:cheb_conv}
  Given $q_j \in \nset$, let $\phi_{q_j}$ be the family of Chebyshev %
  polynomials of the first kind on $[-1,1]$, i.e.,
    $|\phi_{q_j}(y)|\leq 1$ for all $y \in [-1,1]$, and, for
    $N \in \nset_+$ and any $\pp \in \nset^N$, let
    $\Phi_\pp(\yy) =\prod_{j=1}^N \phi_{p_j}(y_j)$.  If
    $f:[-1,1]^N \to \rset$ admits a holomorphic complex extension in a
    Bernstein polyellipse, $\mcE_{\vec{\ellrad}}$, for some
    $\vec{\ellrad} \in (1, \infty)^N$ and if there exists $0< C_f < \infty$ such that
    $\sup_{\zz \in \mcE_{\vec{\ellrad}}} |f(\zz)| \leq
    C_f$, %
    then $f$ admits the following Chebyshev expansion:
  \begin{align*}
  & f(\yy)  = \sum_{\pp \in \nset^N} f_\pp\Phi_\pp(\yy), \\
  & f_\pp  = \frac{1}{\int_{[-1,1]^N} \Phi_\pp^2(\yy) \pdf_C(\yy) d\yy}
  \int_{[-1,1]^N} f(\yy) \Phi_\pp(\yy) \pdf_C(\yy) d\yy,
  \quad \pdf_C(\yy) = \prod_{j=1}^N \left(\sqrt{1-y_j^2}\right)^{-1},
  \end{align*}
  which converges uniformely in $\mcE_{\vec{\ellrad}}$. Moreover the
  following bound on the coefficients $f_\pp$ holds:
  \begin{equation*}
    |f_\pp| \leq \sup_{\zz \in \mcE_{\vec{\ellrad}}} |f(\zz)|
    2^{|\pp|_0} \prod_{j=1}^N \ellrad_{j}^{-p_j},
  \end{equation*}
  where $|\pp|_0$ denotes the number of non-zero elements of $\pp$.
\end{lemma}
\section{The Multi-index Stochastic Collocation method}\label{s:MISC-method}

In this section, we introduce approximations of $\E{\qoi}$ along the
deterministic and stochastic dimensions and their decomposition in
terms of tensorizations of univariate difference operators.  We then
recall the so-called mixed difference operators and the construction of
the MISC estimator, that was first introduced in
\cite{hajiali.eal:MISC1} in a general setting.

\subsection{Approximation along the deterministic and stochastic variables}

\paragraph{A tensorized deterministic solver.}
Let $\{\mesh_i\}_{i=1}^D$ be the triangulations/meshes of each of
  the subdomains $\{\mathscr{B}_i\}_{i=1}^D$ composing the domain
  $\mathscr{B}$; denote by $\{h_i\}_{i=1}^D$ the mesh-size on the mesh
  $\mesh_i$; and let $\bigotimes_{i=1}^D \mesh_i$ be the mesh for
  $\mathscr{B}$.  Then, consider a numerical method for the
  approximation of the solution of \eqref{eq:PDE1} for a fixed value
  of the random variables, $\yy$, based on such a mesh, e.g., finite
  differences, finite volumes,  tensorized finite elements, or $h$-refined
  splines, such as those used in the isogeometric context.  The
  values of $h_i$ are actually given as functions of a positive integer
  value, $\alpha \geq 1$, referred to as a ``deterministic
  discretization level'', i.e., $h_i=h_i(\alpha)$.  Observe that,
  in a more general setting, the mesh-size needs not be a constant value over the subdomain
  $\mathscr{B}_i$ and could be for instance the result of a grading
  function intended to refine subregions of $\mathscr{B}_i$ as in
  Figure~\ref{subf-nont} (see also \cite[Remark
  2.2]{abdullatif:meshMLMC} for further comments on locally refined
  meshes in the context of Multi-Level Monte Carlo methods). In this
  work, we restrict ourselves to constant $h$ for ease of
  presentation.  Given a multi-index, $\valpha \in \nset^D_+$, we
  denote by $u^\valpha(\xx,\yy)$ the approximation of $u$ obtained by
  setting $h_i = h_i(\alpha_i)$ and use notation
  $\qoi^\valpha(\yy) = \Theta[u^\valpha(\cdot,\yy)]$.
  More specifically, in the following we will consider
    \begin{equation}\label{eq:choice_of_h_and_FEM}
      h_i=h_{0,i} 2^{-\alpha_i}, \quad \mbox{for }  i=1,\ldots,D
    \end{equation}
    and a method obtained by tensorizing piecewise multi-linear finite element spaces on each mesh,
  $\{\mesh_i\}_{i=1}^D$, discretizing each $\{\mathscr{B}_i\}_{i=1}^D$.

  As already mentioned in the previous section, MISC could also be
  applied to more general domains, such as those discussed in
  Figure~\ref{fig:domains}, as long as some kind of ``tensor
  structure'' can be induced from the shape of the domain to the
  solver of the deterministic problem and the vector $\valpha$
  determines the refinement level of each component of such a tensor
  structure. The reason why we need such tensor strucure will be made
  clear when we introduce the classic sparse-grids approach to
  solve the problem.
  For non-tensorial domains, we can always set $D=1$ and consider
  an unstructured mesh for the whole domain, $\mathscr{B}$, having
  only one discretization level $\alpha \in \nset_+$.
  In this way, we give up the sparse-grid approach on the
  deterministic part of the problem and obtain a variant of the
  Multi-Level Stochastic Collocation method discussed in
  \cite{teckentrup.etal:MLSC,van-wyk:MLSC}, yet with a different algorithm for
  combining spatial and stochastic discretizations. See Remark \ref{remark:MISC_is_more_general}
  stated next and \cite{hajiali.eal:MISC1} for additional discussion on this aspect.

It would be straightforward to extend this setting to discretization
methods based on degree-elevation rather than on mesh-refinement, such
as spectral methods, $p$-refined finite elements or $p$- and
$k$-refined splines. However, here we limit ourselves to the setting
defined above. It would also be
possible to include time-dependent problems in this framework, but in
this case we might need to take care of possible constraints on
discretization parameters, such as CFL conditions; a broader
generalization could also include ``non-physical'' parameters such as
tolerances for numerical solvers.  Finally, more general problems,
e.g., those depending on random variables with probability
distributions other than uniform distributions or with uncertain
boundary conditions and/or forcing terms could also be addressed
  with suitable modifications of the MISC methodology.

\paragraph{Tensorized quadrature formulae for expected value approximation.}
Similarly to what was presented for the deterministic problem, we base
our approximation of the expected value of $\qoi^\valpha(\yy)$ on a
tensorization of quadrature formulae over the stochastic domain,
$\Gamma$.  Assumptions~\ref{assump:shift} and~\ref{assump:b_and_p}
guarantee that $\qoi^\valpha(\yy)$ is actually a continuous function,
even holomorphic, over $\Gamma$. A quadrature approach is thus sound.

Let $C^0([-1,1])$ be the set of real-valued continuous functions over $[-1,1]$,
$\beta \in \nset_+$ be referred to as a ``stochastic discretization level'',
and $m: \nset \rightarrow \nset$ be a strictly increasing function
with $m(0)=0$ and $m(1)=1$, that we call a ``level-to-nodes function''.
At level $\beta$, we  consider a set of $m(\beta)$ distinct quadrature points in $[-1,1]$,
$\quadpointset^{m(\beta)} = \{y_\beta^1, y_\beta^2, \ldots, y_\beta^{m(\beta)}\} \subset [-1,1]$,
and a set of quadrature weights,
$\mathcal{W}^{m(\beta)}=\{\varpi_\beta^1, \varpi_\beta^2, \ldots,
\varpi_\beta^{m(\beta)}\}$. We then define the quadrature operator as
\begin{equation}\label{eq:unvariate_quad_def}
  Q^{m(\beta)} : C^0([-1,1]) \rightarrow \rset, \qquad
  Q^{m(\beta)}[f] = \sum_{j=1}^{m(\beta)} f(y_\beta^j) \varpi_\beta^j.  %
\end{equation}
The quadrature weights are selected such that
$Q^{m(\beta)}[y^k] = \int_{-1}^1 \frac{y^k}{2} dy, \quad \forall \,
k=0,1,\ldots, m(\beta)-1$. The quadrature points are chosen to
optimize the convergence properties of the quadrature error (the
specific choice of quadrature points is discussed later in this
section). In particular, for symmetry reasons, we define the trivial
operator $Q^{1}[f] = f(0) \,,\, \forall f \in C^0([-1,1])$.

Defining a quadrature operator over $\Gamma$ is more delicate, since
$\Gamma$ is defined as a countable tensor product of intervals. To this end,
we follow \cite{schillings.schwab:inverse} and define,
for any finitely supported multi-index $\vbeta \in \finitesuppset_+$,
\[%
  \mathscr{Q}^{m(\vbeta)}: \Gamma \rightarrow \rset, \quad \mathscr{Q}^{m(\vbeta)}  = \bigotimes_{j \geq 1} Q^{m(\beta_j)},
\]%
where the $j$-th quadrature operator is understood to act only on the $j$-th variable of $f$, and
the tensor product is well defined since it is composed of finitely
many non-trivial factors (see \cite{schillings.schwab:inverse} again).
In practice, the value of $\mathscr{Q}^{m(\vbeta)}[f]$ can be obtained by considering
the tensor grid $\mathscr{T}^{m(\vbeta)} = \bigtimes_{j\geq 1} \quadpointset^{m(\beta_j)}$
with cardinality $ \#\mathscr{T}^{m(\vbeta)} = \prod_{j \geq 1} m(\beta_j)$ and computing
\begin{align*}
\mathscr{Q}^{m(\vbeta)}[f] = \sum_{j=1}^{\# \mathscr{T}^{m(\vbeta)} } f(\widehat{\yy}_j) w_j,
\end{align*}
where $\widehat{\yy}_j \in \mathscr{T}^{m(\vbeta)}$ and $w_j$
are (infinite) products of weights of the univariate quadrature rules.
Notice that having $m(1)=1$  is essential in this construction so that the cardinality
of $\mathscr{T}^{m(\vbeta)}$ is finite for any $\vbeta \in \finitesuppset_+$ and
$\varpi_{\beta_j}^1=1$ whenever $\beta_j=1$. All weights, $w_j$, are thus bounded.

Coming back to the choice of the univariate quadrature points, it is
recommended, for optimal performance, that they are chosen according
to the underlying measure, $\di{\lambda}/2$. Moreover, since we aim at
a hierarchical decomposition of the operator,
$\mathscr{Q}^{m(\vbeta)}$, it is useful (although not necessary, see
e.g., \cite{nobile.eal:optimal-sparse-grids}) that the nodes be
\emph{nested} collocation points, i.e.,
$\quadpointset^{m(\beta)} \subset \quadpointset^{m(\beta+1)}$ for any
$\beta \geq 1$.  Thus, we consider Clenshaw-Curtis points that are
defined as
\begin{equation}\label{eq:CC_def}
  y_\beta^j=\cos\left(\frac{(j-1) \pi}{m(\beta)-1}\right), \quad 1\leq j \leq m(\beta).
\end{equation}
Clenshaw-Curtis points are nested provided that the level-to-nodes function is defined as
\begin{equation}\label{eq:m_CC}
m(0)=0,\,\,m(1)=1,\,\, m(\beta)=2^{\beta-1}+1.
\end{equation}
We close this section by mentioning that another family of nested
points for uniform measures available in the literature is the Leja
points, whose performance is equivalent to that of
Clenshaw-Curtis points for quadrature purposes. See, e.g.,
\cite{Chkifa:leja, nobile.etal:leja,narayan:Leja,schillings.schwab:inverse} and
references therein for definitions and comparison.

\subsection{Construction of the MISC estimator}%

It is straightforward to see that a direct approximation,
$\E{\qoi} \approx \mathscr{Q}^{m(\vbeta)} [\qoi^\valpha]$, is not a
viable option in practical cases, due to the well-known ``curse of
dimensionality'' effect.  In \cite{hajiali.eal:MISC1}, we proposed to
use MISC as a computational strategy to combine spatial and stochastic
discretizations in such a way as to obtain an effective approximation
scheme for $\E{\qoi}$.

MISC is based on a classic sparsification approach in which approximations
like $\mathscr{Q}^{m(\vbeta)}[\qoi^\valpha]$
are decomposed in a hierarchy of operators. Only the most important of
these operators are retained in the approximation.
In more detail, let us denote for brevity $\mathscr{Q}^{m(\vbeta)} [\qoi^\valpha] = \descqoi{\valpha}{\vbeta}$
and introduce the first-order difference operators for the deterministic and
stochastic discretization operators, denoted respectively by
$\Delta_{i}^\textnormal{det}$ with $1\le i\le D$ and $\Delta_{j}^\textnormal{stoc}$ with $j \geq 1$:
\begin{align*}%
& \Delta_{i}^\textnormal{det}[\descqoi{\valpha}{\vbeta}] =
\begin{cases}
  \descqoi{\valpha}{\vbeta} - \descqoi{\valpha-\ee_i}{\vbeta}, & \text{if }\alpha_i > 1,\\
  \descqoi{\valpha}{\vbeta} & \text{if } \alpha_i=1,
  \end{cases} \\
& \Delta_{j}^\textnormal{stoc}[\descqoi{\valpha}{\vbeta}]    =
\begin{cases}
  \descqoi{\valpha}{\vbeta} - \descqoi{\valpha}{\vbeta-\ee_j}, & \text{if }\beta_j > 1,\\
  \descqoi{\valpha}{\vbeta} & \text{if } \beta_j=1.
  \end{cases} %
\end{align*}
As a second step, we define the so-called mixed difference operators,
\begin{align}
& \vec \Delta^\textnormal{det}[\descqoi{\valpha}{\vbeta}]
= \bigotimes_{i=1}^D \Delta_i^\textnormal{det}[\descqoi{\valpha}{\vbeta}]
= \Delta_1^\textnormal{det} \left[ \, \Delta_2^\textnormal{det} \left[ \, \cdots
  \Delta_D^\textnormal{det} \left[ \descqoi{\valpha}{\vbeta} \right] \, \right] \, \right] %
= \sum_{\ii \in \{0,1\}^D} (-1)^{|\ii|} \descqoi{\valpha-\ii}{\vbeta},
\label{eq:delta_det_def} \\
& \vec \Delta^\textnormal{stoc}[\descqoi{\valpha}{\vbeta}]
= \bigotimes_{j \geq 1} \Delta_j^\textnormal{stoc}[\descqoi{\valpha}{\vbeta}]
= \sum_{\jj \in \{0,1\}^{\nset_+}} (-1)^{|\jj|} \descqoi{\valpha}{\vbeta-\jj} \, ,
\label{eq:delta_stoc_def}
\end{align}
with the convention that $\descqoi{\vec{v}}{\vec{w}}=0$ whenever a component of $\vec{v}$ or $\vec{w}$ is zero.
Notice that, since $\vbeta$ has finitely many components larger than 1, the sum on the right-hand side
of \eqref{eq:delta_stoc_def} contains only a finite number of terms.
Finally, letting
\begin{equation}\label{eq:deltaF_def}
  \vec \Delta[\descqoi{\valpha}{\vbeta}] =
  \vec \Delta^\textnormal{stoc}[ \vec \Delta^\textnormal{det}[\descqoi{\valpha}{\vbeta}]],
\end{equation}
we define the Multi-index Stochastic Collocation (MISC) estimator of $\E{\qoi}$ as
\begin{equation}\label{eq:misc_estimator}
  \mathscr{M}_{\mathcal{I}}[\qoi]
  = \sum_{[\valpha, \vbeta] \in \mathcal I} \vec \Delta[\descqoi{\valpha}{\vbeta}]
  = \sum_{ [\valpha, \vbeta] \in \mathcal I}  c_{\valpha,\vbeta} \descqoi{\valpha}{\vbeta}, \quad
  c_{\valpha,\vbeta} =  \sum_{\substack{[\ii,\jj] \in \{0,1\}^{D+\nset} \\[1pt] [\valpha+\ii,\vbeta+\jj] \in \mathcal{I}}}(-1)^{|[\ii,\jj]|_0},
\end{equation}
where $\mathcal{I} \subset \nset_+^{D} \times \finitesuppset_+$.  The
second form of the MISC estimator is known as the ``combination
technique'', since it expresses the MISC approximation as a linear
combination of a number of tensor approximations,
$\descqoi{\valpha}{\vbeta}$, and might be useful for the practical
implementation of the method; we observe in particular that many of its
coefficients, $c_{\valpha,\vbeta}$, are zero.

The effectiveness of MISC crucially depends on the choice of the index
set, $\mathcal{I}$.  Given the hierarchical structure of MISC, a
natural requirement is that $\mathcal{I}$ should be downward closed,
i.e.,
\begin{align*}
  \forall \,[\valpha, \vbeta] \in \mathcal{I}, \quad
  \begin{cases}
    [\valpha - \vec{e}_i, \vbeta] \in \mathcal{I} \mbox{ for all } 1 \leq
    i \leq D \mbox{ such that } \alpha_i > 1,\\
    [\valpha, \vbeta - \vec{e}_j] \in \mathcal{I} \mbox{ for all }  j \geq
    1 \mbox{ such that } \beta_j > 1 %
  \end{cases}
\end{align*}
(see also \cite{nobile.eal:optimal-sparse-grids,wasi.wozniak:cost.bounds,b.griebel:acta}).
In addition to this general constraint, in \cite{hajiali.eal:MISC1} we have
detailed a procedure to derive an efficient set, $\mathcal{I}$,
based on an optimization technique inspired by the Dantzig algorithm for the approximate
solution of the knapsack problem (see \cite{martello:knapsack}).
In the following, we briefly summarize this procedure and refer to \cite{hajiali.eal:MISC1}
as well as to \cite{nobile.eal:optimal-sparse-grids,back.nobile.eal:optimal,b.griebel:acta}
for a thorough discussion on the similarities between this procedure and the Dantzig algorithm.

The first step of our optimized construction consists of introducing
the ``work contribution'', $\Delta W_{\valpha, \vbeta}$, and ``error
contribution'', $\Delta E_{\valpha, \vbeta}$, for each operator,
$\vec
\Delta[\descqoi\valpha\vbeta]$. %
The work contribution measures the computational cost (measured,
e.g., as a function of the total number of degrees of freedom, or in
terms of computational time) required to add
$\vec \Delta[\descqoi\valpha\vbeta]$ to
$\mathscr{M}_\mathcal{I}[\qoi]$, i.e.,
\begin{equation}\label{eq:DeltaW_def}
\Delta W_{\valpha, \vbeta}
= \work{\mathscr{M}_{\mathcal{I} \cup \{[\valpha,\vbeta]\}}} - \work{\mathscr{M}_{\mathcal{I}}}
= \work{\vec \Delta[\descqoi\valpha\vbeta]}.
\end{equation}
Similarly, the error contribution measures how much the error,
$|\E{\qoi} - \mathscr{M}_\mathcal{I}[\qoi]| $, would decrease if the
operator $\vec \Delta[\descqoi\valpha\vbeta]$ were added to
$\mathscr{M}_\mathcal{I}[\qoi]$,
\begin{equation}\label{eq:DeltaE_def}
  \Delta E_{\valpha, \vbeta}
= \left| \mathscr{M}_{\mathcal{I} \cup \{[\valpha,\vbeta]\}}[\qoi] - \mathscr{M}_{\mathcal{I}}[\qoi] \right|
= \left| \vec \Delta[\descqoi{\valpha}{\vbeta}] \right|.
\end{equation}
We observe that the following decompositions of the total work and error of the MISC estimator hold:
\begin{align}
& \work{\mathscr{M}_\mathcal{I}} = \sum_{[\valpha, \vbeta] \in \mathcal{I}} \Delta W_{\valpha, \vbeta},  \label{eq:work_decomp} \\
& |\E{\qoi} - \mathscr{M}_\mathcal{I}[\qoi]|
  = \left| {\sum_{[\valpha, \vbeta] \notin \mathcal I} \vec \Delta[\descqoi\valpha\vbeta]} \right|
 \leq \sum_{[\valpha, \vbeta] \notin \mathcal I} \left| { \vec \Delta[\descqoi\valpha\vbeta] } \right|
 = \sum_{[\valpha, \vbeta] \notin \mathcal I} \Delta E_{\valpha, \vbeta}. \label{eq:error-decomp}
\end{align}

Although it would be tempting to define $\mathcal{I}$ as the set of
couples $[\valpha,\vbeta]$ with the largest error contribution, this
choice could be far from optimal in terms of computational cost. As
suggested in the literature on the knapsack problem (see
\cite{martello:knapsack}), the benefit-to-cost ratio should be taken
into account in the decision (see also
\cite{hajiali.eal:MISC1,nobile.eal:optimal-sparse-grids,back.nobile.eal:optimal,b.griebel:acta,griebel.knapek:optimized}).
More precisely, we propose to build the MISC estimator by first
assessing the so-called ``profit'' of each operator
$\vec \Delta[\descqoi{\valpha}{\vbeta}]$, i.e., the quantity
 \[%
   P_{\valpha, \vbeta} = \frac{\Delta E_{\valpha, \vbeta}}{\Delta W_{\valpha, \vbeta}}.
 \]
Then we build an index set for the MISC estimator:
\begin{equation}\label{eq:opt_set}
\mathcal{I} = \mathcal{I}(\epsilon) =\left\{ [\valpha, \vbeta] \in \nset_+^D \times \finitesuppset_+\::\:
  P_{\valpha, \vbeta} \geq \epsilon \right\},
\end{equation}
for a suitable $\epsilon>0$. We observe that the obtained set is
not necessarily downward-closed; we have to enforce this condition
a posteriori. Obviously, $\Delta E_{\valpha, \vbeta}$ and
$\Delta W_{\valpha, \vbeta}$ are not, in general, at our disposal. In
practice, we base the construction of the MISC estimator on a-priori
bounds for such quantities.  More precisely, we derive a-priori
ansatzes for these bounds from theoretical considerations and then fit
the constants appearing in the ansatzes with some auxiliary
computations. We refer to the entire strategy as a priori/a
posteriori.

\begin{remark}\label{remark:MISC_is_more_general}
  We remark that the general form of the MISC estimator
  \eqref{eq:misc_estimator} is quite broad and includes other related
  methods (i.e., methods that combine different spatial and stochastic
  discretization levels to optimize the computational effort)
  available in the literature, such as the ``Multi Level Stochastic
  Collocation'' method \cite{teckentrup.etal:MLSC,van-wyk:MLSC} and
  the ``Sparse Composite Collocation'' method
  \cite{bieri:sparse.tensor.coll}; see also \cite{hps13}. The main
  novelty of the MISC estimator
  \eqref{eq:misc_estimator}-\eqref{eq:opt_set} with respect to such
  methods is the profit-oriented selection of difference
  operators. Another difference from
    \cite{teckentrup.etal:MLSC,van-wyk:MLSC} is the fact that
  difference operators in our approach are introduced in both the
  spatial and stochastic domains. See also
    \cite{bieri:sparse.tensor.coll,griebel.eal:non-nested-MLMC} for a
    similar construction, in which no optimization is performed.
  More details on the comparison between the above-mentioned methods
  and MISC can be found in \cite{hajiali.eal:MISC1}.
\end{remark}

\section{Error Analysis of the MISC method}\label{s:err}

In this section, we state and prove a convergence theorem for the
profit-based MISC estimator based on the multi-index set
\eqref{eq:opt_set}. The theorem is based on a result from the previous
work \cite{nobile.eal:optimal-sparse-grids}, which was proved in the
context of sparse-grid approximation of Hilbert-space-valued
functions. Since the sparse grid and the MISC constructions are
identical, this theorem can be used verbatim here. In particular, it
links the summability of the profits to the convergence rate of
methods such as MISC and Sparse Grids Stochastic Collocation, i.e.,
based on a sum of difference operators.  To use this result, we have
to assess the summability properties of the profits. We thus introduce
suitable estimates of the error and work contributions,
$\Delta E_{\valpha,\vbeta}$ and $\Delta W_{\valpha,\vbeta}$,
respectively.  In particular, the estimate of
$\Delta E_{\valpha,\vbeta}$ depends on the spatial regularity of the
solution, on the convergence rate of the method used to solve the
deterministic problems, and on the summability property of the
Chebyshev expansion of the solution over the parameter space.

\begin{theorem}[Convergence of the profit-based MISC estimator, see  \cite{nobile.eal:optimal-sparse-grids}]\label{theo:general-conv-result}
If the profits, $P_{\valpha,\vbeta}$, satisfy the weighted summability condition
\[%
\left( \sum_{  [\valpha,\vbeta] \in \nset_+^D \times \finitesuppset_+} P_{\valpha,\vbeta}^{p} \Delta W_ {\valpha,\vbeta} \right)^{1/p} =C_P(p) < \infty
\]%
for some $0<p\leq 1$, then
\[%
\big| \E{\qoi} - \mathscr{M}_\mathcal{I}[\qoi] \big|  \leq
C_P(p) \work{\mathscr{M}_{\mathcal{I}}}^{1-1/p} ,
\]%
where $\work{\mathscr{M}_{\mathcal{I}}}$ is given by
\eqref{eq:work_decomp}. %
\end{theorem}

We begin by introducing an estimate for the size of the work contribution,
$\Delta W_{\valpha,\vbeta}$. %
To this end, let $\Delta W_{\valpha}^{\textnormal{det}}=\work{\vec\Delta^{\textnormal{det}}[\qoi^{\valpha}]}$,
i.e., let it be the cost of computing $\vec\Delta^{\textnormal{det}}[\qoi^{\valpha}]$ according to
\eqref{eq:delta_det_def}.
\begin{assumption}[Spatial work contribution]\label{assump:work}
There exist $\gamma_i \in [1, \infty)$ for $i=1,\ldots,D$ and $C_W > 0$ such that
\begin{equation}\label{eq:dW_det_assump}
  \Delta W_{\valpha}^{\textnormal{det}} \leq C_{W} 2^{  \sum_{i=1}^D \gamma_i d_i \alpha_i },
  \end{equation}
  where $2^{\sum_{i=1}^D d_i \alpha_i}$ is proportional to the number of
  degress of freedom in the mesh on level $\valpha$, cf. equation \eqref{eq:choice_of_h_and_FEM},
  and $\gamma_i$ are related to the used deterministic solver
  and to the sparsity structure of the linear system,
  which might be different on each $\mathscr{B}_i$ depending on the chosen discretization.
\end{assumption}
\begin{lemma}[Total work contribution]\label{lemma:work}
  When using Clenshaw-Curtis points for the discretization over the parameter space,
  the work contribution, $\Delta W_{\valpha,\vbeta}$, of each difference operator,
  $\vec \Delta[\descqoi{\valpha}{\vbeta}]$, can be decomposed as
  \[
  \Delta W_{\valpha,\vbeta} \leq C_{W} 2^{\sum_{i=1}^D \gamma_i d_i \alpha_i + |\vbeta-\oone|},
  \]
  with $\gamma_i$ and $C_W$ as in Assumption~\ref{assump:work}.
\end{lemma}
\begin{proof}
  Combining \eqref{eq:DeltaW_def} and \eqref{eq:deltaF_def}, we have
  \[%
  \Delta W_{\valpha, \vbeta} = \work{ \vec \Delta^\textnormal{stoc}[ \vec \Delta^\textnormal{det}[\descqoi{\valpha}{\vbeta}]] }
  = \work{ \vec \Delta^\textnormal{stoc}[ \vec \Delta^\textnormal{det}[\mathscr{Q}^{m(\vbeta)}[F^\valpha(\cdot)]]] }.
  \]%
  Since the Clenshaw-Curtis points are nested, computing $\Delta
  W_{\valpha,\vbeta}$ (i.e., adding $[\valpha,\vbeta]$ to the set
  $\mathcal{I}$ that defines the current MISC estimator) amounts to
  evaluating $F^\valpha(\yy)$ in the set of ``new'' points added to the estimator
  by $\vec{\Delta}^{\textnormal{stoc}}[\cdot]$, i.e.,
  $\bigtimes_{j : \beta_j > 1} \left\{ \quadpointset^{m(\beta_j)} \setminus \quadpointset^{m(\beta_j-1)} \right\}$,
  whose cardinality is $\prod_{j \geq 1} ( m(\beta_j) - m(\beta_j-1))$.
  The proof is then concluded by observing that the definition of
  $m(\beta)$ in \eqref{eq:m_CC} immediately gives
  $m(\beta_j) - m(\beta_j-1) \leq 2^{\beta_j-1}$ and recalling Assumption~\ref{assump:work}.
\end{proof}
\begin{remark}
  We observe that the exponent $\vbeta
  -1$ guarantees that the directions along which no quadrature is
  actually performed (i.e., $\beta_j=1$ for any $j\geq
  1$) do not contribute to the total work.
\end{remark}

Next, we prove a sequence of lemmas that allow us to conclude
that an analogous estimate holds for the error contribution %
as well, i.e., that $\Delta E_{\valpha,\vbeta}$ can be bounded as a product of a term related
to the spatial discretization and a term related to the approximation over the parameter
space.
To this end, we need to introduce the quantity
\[
\Leb_{m(\beta)} =
\sup_{f \in C^0([-1,1]), \|f\|_{L^\infty(-1,1)} = 1}
\left| Q^{m(\beta)}[f] -  Q^{m(\beta-1)}[f] \right| \quad \forall \beta \in \nset_+,
\]
where $Q^{m(\beta)}[\cdot]$ are the univariate quadrature operators
introduced in \eqref{eq:unvariate_quad_def}, and observe that
$\Leb_{1}=1$. %
Next, let $\maxLebD = \max_{\beta \geq 1} \textnormal{Leb}_{m(\beta)}$,
and note that $ \maxLebD\leq 2$ since $Q^{m(\beta)}$ has positive weights. Moreover,
a much smaller bound on $\maxLebD$ can be obtained for Clenshaw--Curtis points.
Indeed, since Clenshaw--Curtis points are nested, we can also bound
$\Leb_{m(\beta)} \leq \LebD_{m(\beta)}$ with
\[
\LebD_{m(\beta)} = \sum_{y^j_\beta \in \quadpointset^{m(\beta)} \cap \quadpointset^{m(\beta-1)}}  \left| \varpi_\beta^j - \varpi_{\beta-1}^j \right|
      + \sum_{y_j \in \quadpointset^{m(\beta)} \setminus \quadpointset^{m(\beta-1)}} \left| \varpi_{\beta}^j  \right|,
\]
and it can be verified numerically that $\LebD_{m(\beta)}$ is bounded,
attains its maximum for $\beta=3$ and converges to 1 for
$\beta \to \infty$, see Figure \ref{fig:lebtilde_cc}.
Therefore, we have $\maxLebD \leq \widetilde{\textnormal{Leb}}_{m(3)} \approx 1.067$.

\begin{figure}[htbp]
  \centering
  \includegraphics[width=0.5\textwidth]{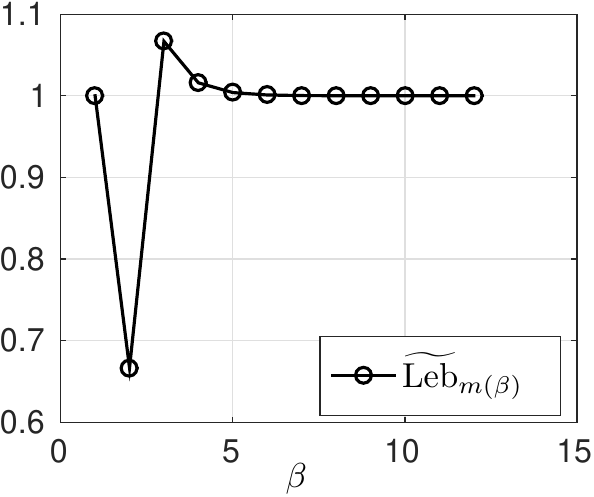}
  \caption{Numerical evaluation of $\LebD_{m(\beta)}$ for Clenshaw-Curtis points.}
  \label{fig:lebtilde_cc}
\end{figure}

\begin{lemma}[Stochastic error contribution]\label{lemma:deltaE_stoc_decay}
  Let $f:\Gamma \to \rset$ and $\vbeta \in \finitesuppset_+$, and
  assume that the quadrature operator, $\mathscr{Q}^{m(\vbeta)}$, is
  built with Clenshaw-Curtis abscissae.  If there exists a sequence,
  $\vec{\rho}=\{\rho_j\}_{j \geq 1}$, $\rho_j > 1$ for all $j$, such
  that
  \begin{enumerate}
    \item $\sum_{j \geq 1} \frac{1}{\rho_j} < \infty$,
  	\item there exists $0 < C_f < \infty$ such that for any finite set,
      $\finiteset' \subset \nset_+$ with $\# \finiteset' < \infty$,
      the restriction of $f$ on $([-1,1])^{\finiteset'}$ admits
      a holomorphic complex extension in a Bernstein polyellipse,
      $\mcE_{\vec{\rho}}^{\finiteset'}$ with
      $\sup_{\zz \in \mcE_{\vec{\rho}}^{\finiteset'}} |f(\vec z)| \leq C_f$.
  \end{enumerate}
  Then, the set
  \[
  \mathcal J = \left\{ j \geq 1 : \rho_j \leq 2 (\maxLebD)^{\frac{1}{3}}\right\},
  \]
  has finite cardinality, i.e., $\#\mathcal J < \infty$ and
  \[%
  \left| \vec  \Delta^{\textnormal{stoc}}[\mathscr{Q}^{m(\vbeta)}f] \right|
  \leq C_{\textrm{SE}}({\vec \rho}) \sup_{\zz \in
    \mcE_{\vec{\rho}}^\finiteset} |f(\zz)| e^{-\sum_{j \geq 1} g_jm(\beta_j-1)}
  \leq C_{\textrm{SE}}({\vec \rho}) C_f e^{-\sum_{j \geq 1} g(\rho_j) m(\beta_j-1)},
  \]%
  holds, where $\finiteset$ is the support of $\vbeta-\oone$,
\[
  0 < g(\rho_j) = \begin{cases}
    \log \rho_j, & \text{for } j \in \mathcal J, \\
    \log \rho_j - \log 2 - \frac{1}{3} \log \left( \maxLebD \right), &
    \text{otherwise},
\end{cases}
\]
and $C_{\textrm{SE}}({\vec \rho})
  < \infty
  $
  is independent of $\vbeta$.
\end{lemma}
\begin{proof}
  Let $\finiteset$ be the support of $\vbeta-\oone$ with cardinality
  $\#\finiteset  < \infty$, $\kk \in \nset_+^\finiteset$, and let
  $\Phi_{\finiteset,\kk}$ denote the Chebyshev polynomials of the
  first kind with degree $k_j$ along $y_{j}$ for $j \geq 1$.
  We observe that
  $\Phi_{\finiteset,\kk}$ are equivalent to the $\#\finiteset$-variate
  Chebyshev polynomials over $[-1,1]^{\#\finiteset}$ thanks to the
  product structure of the multivariate Chebyshev polynomials and to
  the fact that $\phi_0(y)=1$. Next, consider the holomorphic
  extension of $f:\mathbb{C}^\finiteset \to \mathbb{C}$, and its
  Chebyshev expansion over $\Phi_{\finiteset,\kk}$ introduced in
  Lemma~\ref{lemma:cheb_conv}, where
  \begin{align*}%
    |\vec \Delta^{\textnormal{stoc}}[ \mathscr{Q}^{m(\vbeta)} f ]|
    = \Big| \vec \Delta^{\textnormal{stoc}} \Big[ \mathscr{Q}^{m(\vbeta)} \Big[\sum_{\kk \in \nset_+^\finiteset} f_\kk\Phi_{\finiteset,\kk} \Big] \Big] \Big|
    = \Big| \sum_{\kk \in \nset_+^\finiteset} f_\kk \vec \Delta^{\textnormal{stoc}} \Big[\mathscr{Q}^{m(\vbeta)}\big[\Phi_{\finiteset,\kk} \big] \Big]  \Big|
  \end{align*}
  holds.  By construction of hierarchical surplus, we have
  $\vec \Delta^{\textnormal{stoc}}[
  \mathscr{Q}^{m(\vbeta)}[\Phi_{\finiteset,\kk }]] = 0$ for all
  Chebyshev polynomials, $\Phi_{\finiteset,\kk}$, such that
  $\exists \,j \in \finiteset : k_j < m(\beta_j -1)$ (i.e., for
  polynomials that are integrated exactly at least in one direction by
  both quadrature operators along that direction). Therefore, the
  previous sum reduces to the multi-index set
  $\kk \geq m(\vbeta - \oone)$. Furthermore, by the triangular inequality,
  we have
  \[
    |\vec \Delta^{\textnormal{stoc}}[ \mathscr{Q}^{m(\vbeta)} f ]|
    \leq \sum_{\kk \geq m(\vbeta-\oone)} | f_\kk |
    \left| \vec \Delta^{\textnormal{stoc}} \Big[\mathscr{Q}^{m(\vbeta)}\big[\Phi_{\finiteset,\kk} \big] \Big] \right|.
  \]
  Next, using the definitions of $\vec \Delta^{\textnormal{stoc}}$ and
  $\Leb_{m(\beta)}$ and recalling that Chebyshev polynomials of the first kind
  on $[-1,1]$ are bounded by $1$, that $\Leb_{m(\beta)} \leq \LebD_{m(\beta)} \leq 1$ for
  $\beta=1,2$ and $\Leb_{m(\beta)}\leq\maxLebD$ for $\beta \geq 3$, we have %
  \begin{equation}
    \label{eq:lebtilde_bound}
    \aligned
    \Big|
    	\vec \Delta^{\textnormal{stoc}} \Big[
    		\mathscr{Q}^{m(\vbeta)}\big[
    			\Phi_{\finiteset,\kk}\big]
    	\Big]
    \Big|
    & = \left| \bigotimes_{j \in \finiteset} \Delta^{\textnormal{stoc}}_j[ Q^{m(\beta_j)}[\phi_{k_j}]] \right|  \\
    & \leq \prod_{j \in \finiteset}  \LebD_{m(\beta_j)}
      \norm{\phi_{k_j}}{L^\infty([-1,1])}   \leq  \prod_{\beta_j \geq 3} \maxLebD.
      \endaligned
  \end{equation}
  We then bound $| f_\kk |$ by Lemma~\ref{lemma:cheb_conv} to obtain
  \begin{align*}
    |\vec \Delta^{\textnormal{stoc}}[ \mathscr{Q}^{m(\vbeta)} f ]|
     & \leq \sup_{\zz \in \mcE_{\vec{\rho}}^\finiteset} |f(\zz)| \Bigg( \prod_{\beta_j \geq 3} \maxLebD \Bigg)
       \sum_{\kk \geq m(\vbeta - \oone)} \prod_{j \in \finiteset} 2^{|k_j|_0}
       \rho_j^{-k_j}  \\
     & \leq \sup_{\zz \in \mcE_{\vec{\rho}}^\finiteset} |f(\zz)| \Bigg( \prod_{\beta_j \geq 3} \maxLebD \Bigg)
       \left(\prod_{j \in \finiteset} \sum_{k_j \geq m(\beta_j - 1)}
       2^{|k_j|_0} \rho_j^{-k_j} \right)\\
     & = \sup_{\zz \in \mcE_{\vec{\rho}}^\finiteset} |f(\zz)|
       \Bigg( \prod_{\substack{j \in \mathcal J \\ \beta_j \geq 3}}
       \maxLebD \Bigg)
    \Bigg( \prod_{\substack{j \notin \mathcal J \\ \beta_j \geq 3}} \maxLebD \Bigg)
    \\
     &\hskip 1cm  \left(\prod_{j \in \finiteset \cap \mathcal J } \sum_{k_j \geq m(\beta_j - 1)}
       2^{|k_j|_0} \rho_j^{-k_j} \right)
       \left(\prod_{j \in \finiteset \setminus \mathcal J} \sum_{k_j \geq m(\beta_j - 1)}
       2^{|k_j|_0} \rho_j^{-k_j} \right).
  \end{align*}
Next, we observe that $|k_j|_0 \leq \min\{1, k_j\}$ for $k_j \geq 0$ and $1 \leq \frac{1}{3}m(\beta_j - 1)$ for all $\beta_j \geq 3$. Then
\begin{align*}
  |\vec \Delta^{\textnormal{stoc}}[ \mathscr{Q}^{m(\vbeta)} f ]| &
  \leq \sup_{\zz \in \mcE_{\vec{\rho}}^\finiteset} |f(\zz)| \maxLebD^{\# \mathcal J}
  \Bigg( \prod_{j \notin \mathcal J} \maxLebD^{\frac{1}{3}m(\beta_j-1)} \Bigg)
  \\
  &\hskip 1cm 2^{\# \mathcal J} \left(\prod_{j \in \finiteset \cap
      \mathcal J } \sum_{k_j \geq m(\beta_j - 1)} \rho_j^{-k_j}
  \right) \left(\prod_{j \in \finiteset \setminus \mathcal J} \sum_{k_j \geq
      m(\beta_j - 1)}
    2^{k_j} \rho_j^{-k_j} \right)\\
  & \leq \left(2 \maxLebD \right)^{\# \mathcal J} \sup_{\zz \in
    \mcE_{\vec{\rho}}^\finiteset} |f(\zz)|
  \left(\prod_{j \in \finiteset} \sum_{k_j \geq m(\beta_j - 1)}
    e^{-g(\rho_j) k_j} \right)
  \\
  & = \left(2 \maxLebD \right)^{\# \mathcal J} \sup_{\zz \in
    \mcE_{\vec{\rho}}^\finiteset} |f(\zz)|
  \left(\prod_{j \in \finiteset} \frac{1}{1-e^{- g(\rho_j)}}\right)
  \left(\prod_{j \in \finiteset}
    e^{-g(\rho_j) m(\beta_j-1)}\right) \\
  &\leq C_{\textrm{SE}}({\vec \rho})\sup_{\zz \in
    \mcE_{\vec{\rho}}^\finiteset}  |f(\zz)| \prod_{j \geq 1} e^{-g(\rho_j)
    m(\beta_j-1)},
\end{align*}
where the last inequality is due to the fact that $m(\beta_j - 1) = 0$ whenever $j \notin \finiteset$ or
equivalently $\beta_j=1$ and
\[
  \prod_{j \in \finiteset} \frac{1}{1-e^{- g(\rho_j)}} \leq
  \prod_{j > 1} \frac{1}{1-e^{- g(\rho_j)}},
\]
since $g(\rho_j) > 0$ for all $j\geq 1$. Note that $C_{\textrm{SE}}({\vec \rho})$ is
independent of $\vbeta$ and is bounded since %
\[ \aligned
  \prod_{j \geq 1} \frac{1}{1-e^{-{g}(\rho_j)}} < \infty
  &\Longleftrightarrow -\sum_{j \geq 1} \log(1-e^{-{g}(\rho_j)}) < \infty
  \Longleftrightarrow \sum_{j \geq 1} e^{-{g}(\rho_j)} < \infty \\
  &\Longleftrightarrow \sum_{j \in \mathcal J} \frac{1}{\rho_j} + \sum_{j
    \notin \mathcal J} \frac{2\left(\maxLebD\right)^{\frac{1}{3}}}{\rho_j} <
  \infty,
\endaligned
\]
which was assumed. Moreover, to show that
$\# \mathcal J < \infty$, note that
$\sum_{j \in \mathcal J} \rho_j^{-1}$ is otherwise unbounded, which
contradicts the first assumption of the theorem, namely
$\sum_{j \geq 1} \rho_j^{-1} < \infty$.
\end{proof}
\begin{remark}
  Sharper estimates could be obtained by exploiting the structure of
  the Chebyshev polynomials when bounding
  $\left|{\Delta^{\mathrm{stoc}}[Q^{m(\beta_j)}[\phi_{k_j}]]}\right|$ in
  \eqref{eq:lebtilde_bound} (for instance, the fact that
  $Q^{m(\beta_j)}[\phi_{k_j}]=0$ whenever $k_j$ is odd and larger than
  1) rather than using the general bound
  $\Delta^{\mathrm{stoc}}[ Q^{m(\beta_j)}[\phi_{k_j}]] \leq
  \LebD_{m(\beta_j)} \norm{\phi_{k_j}}{L^\infty([-1,1])}$.
\end{remark}

We are now almost in the position to prove the estimate on the error contribution %
(see Lemma~\ref{lemma:error}); before doing this,
we need another auxiliary lemma that gives conditions
for the summability of certain sequences that will be considered in the proof of Lemma~\ref{lemma:error}
as well as in the proof of the main theorem on the convergence of MISC.

\begin{lemma}[Summability of stochastic rates]\label{lemma:gn_summ}
  Recall the definitions of $\ellrad_{0,j}$ in
  Lemma~\ref{lemma:polyell_analyticity}, of $\ellrad_{s,j}$ in
  Lemma~\ref{lemma:polyell_analyticity_high-order} and of $g(\cdot)$ in
  Lemma~\ref{lemma:deltaE_stoc_decay}.  Under
  Assumption~\ref{assump:b_and_p}, for all $s=0, 1, \ldots, s_{\max}$,
  the sequences $\{e^{-g(\ellrad_{s,j})}\}_{j \in \nset_+}$ and
  $\left\{\frac{1}{\ellrad_{s,j}}\right\}_{j \in \nset_+}$ are
  $\ell^{p}$-summable for $p \geq \tilde{p}_s = \frac{p_s}{1-p_s}$, with
  $\tilde{p}_s < 1$.
\end{lemma}
\begin{proof}
  First note that, by definition of $g(\cdot)$, we have
  \[ \sum_{j \geq 1} e^{-p g(\ellrad_{s,j})} \leq
    2^{p} \left(\maxLebD\right)^{3p} \sum_{j \geq 1} \ellrad_{s,j}^{-p}.\]
  Then, from \eqref{eq:ellrad_dep_on_b_s}--\eqref{eq:taun_def_H1_norm_s}, or
  \eqref{eq:ellrad_dep_on_b}--\eqref{eq:taun_def_H1_norm} for $s=0$, we can bound
  $ 2\tau_{s,j} \leq \ellrad_{s,j}$ and obtain
  \[
    \sum_{j \geq 1} \ellrad_{s,j}^{-p} \leq 2^{-p} \sum_{j \geq 1}
    \tau_{s,j}^{-p} = 2^{-p}\left( \frac{\pi}{E_\delta
        \norm{\bb_s}{\ell^{p_s}}^{p_s}} \right)^{-p} \sum_{j \geq 1}
    b_{s,j}^{(1-p_s)p}.
  \]
  From Assumption~\ref{assump:b_and_p}, we know that
  $\bb_s \in \ell^{p_s}$ for $p_s \leq \frac{1}{2}$, therefore we have the condition
  \[
  (1-p_s)p \geq p_s
  \Rightarrow
  p \geq \frac{p_s}{1-p_s} < 1.
  \]
\end{proof}

\begin{lemma}[Total error contribution]\label{lemma:error}
  Assume that the deterministic problem is solved with a method obtained by
  tensorizing piecewise multi-linear finite element spaces on each mesh,
  $\{\mesh_i\}_{i=1}^D$, discretizing each
  $\{\mathscr{B}_i\}_{i=1}^D$, and let $h_i$ as in \eqref{eq:choice_of_h_and_FEM} %
  be the mesh size of each $\{\mesh_i\}_{i=1}^D$.
  Then, under Assumptions \ref{assump:shift} and \ref{assump:b_and_p},
  the error contribution, $\Delta E_{\valpha,\vbeta}$, of each difference
  operator, $\vec \Delta[\descqoi{\valpha}{\vbeta}],$ can be
  decomposed as
  \begin{equation}\label{eq:deltaE_bound_with_min}
    \Delta E_{\valpha,\vbeta} \leq \min_{s=0,1,\ldots,s_{\max}} C_s \Delta E_{\valpha}^{\textnormal{det}}(s)\Delta E_{\vbeta}^{\textnormal{stoc}}(s),
  \end{equation}
  for a constant $C_s < \infty$ independent of $\valpha$ and $\vbeta$ and
  \begin{align}%
    \Delta E_{\valpha}^{\textnormal{det}}(s) & = 2^{-\valpha\cdot \vec{r}_{\mathrm{FEM}}(s\qq)} \label{eq:deltaE_factors_alpha_s},\\
    \Delta E_{\vbeta}^{\textnormal{stoc}}(s) & = e^{-\sum_{j\geq 1} m(\beta_j-1)g_{s,j}} \label{eq:deltaE_factors_beta_s},
  \end{align}
  with $g_{s,j} = g(\ellrad_{s,j})$ as in Lemma~\ref{lemma:deltaE_stoc_decay} and
  $r_{\mathrm{FEM}}(s\qq)_i=\min\left\{1,q_is\right\}$
  for $i=1,\ldots,D$ with $\qq \in \rset_+^D$ s.t. $|\qq|=1$.%
\end{lemma}

\begin{proof}
 Combining the definition of $\vec \Delta[\descqoi{\valpha}{\vbeta}]$, cf. \eqref{eq:deltaF_def},
 and the definition of $\Delta^{\textnormal{det}}[\descqoi{\valpha}{\vbeta}]$, cf. \eqref{eq:delta_det_def},
 we have
 \begin{align*}
   \Delta E_{\valpha,\vbeta} = | \vec \Delta[\descqoi{\valpha}{\vbeta}] |
 & = \vec \Delta^\textnormal{stoc}[ \vec \Delta^\textnormal{det}[\descqoi{\valpha}{\vbeta}]]
   = \vec \Delta^\textnormal{stoc}[ \sum_{\jj \in \{0,1\}^D} (-1)^{|\jj|} \descqoi{\valpha-\jj}{\vbeta} ] \\
 & = \vec \Delta^\textnormal{stoc}[ \sum_{\jj \in \{0,1\}^D} (-1)^{|\jj|} \mathscr{Q}^{m(\vbeta)}[\Theta[u^{\valpha-\jj}(\cdot,\yy)]] \\
 & = \vec \Delta^\textnormal{stoc}[ \mathscr{Q}^{m(\vbeta)}\Theta[ \sum_{\jj \in \{0,1\}^D} (-1)^{|\jj|} u^{\valpha-\jj}(\cdot,\yy)]] \\
 & = \vec \Delta^\textnormal{stoc}[ \mathscr{Q}^{m(\vbeta)}\Theta[ \vec{\Delta}^\textnormal{det}[ u^{\valpha}(\cdot,\yy)]]].
 \end{align*}
 We observe that $f(\yy) = \Theta[ \vec{\Delta}^\textnormal{det}[u^{\valpha}(\cdot,\yy)]]$
 is a linear combination of some $u^\valpha$
 and that each of these $u^\valpha$ is
 an $H^1_0(\mathscr{B},\mathbb{C})$-holomorphic function, since the
 finite-element approximations of $u$
 are holomorphic in the same complex region as $u$ itself; hence,
 $f(\yy)$ is also holomorphic.
 Then, thanks to the summability properties in Lemma \ref{lemma:gn_summ},
 we can apply Lemma~\ref{lemma:deltaE_stoc_decay} in the polyellipses defined
 in Lemmas \ref{lemma:polyell_analyticity} and \ref{lemma:polyell_analyticity_high-order}
 by $\vec{\ellrad}_s$ in \eqref{eq:ellrad_dep_on_b_s} (or \eqref{eq:ellrad_dep_on_b} for $s=0$)
 and obtain
 \begin{align*}
 \Delta E_{\valpha,\vbeta}
 &\leq C_{\mathrm{SE}}(\vec \ellrad_s) \sup_{\zz \in \mcE_{\vec{\ellrad}_s}^\finiteset} |\Theta [\vec{\Delta}^{\textnormal{det}} [u^{\valpha}(\xx,\zz)]|
  e^{- \sum_{j\geq 1} g(\ellrad_{s,j})  m(\beta_j-1)} \\
 &\leq C_{\mathrm{SE}}(\vec \ellrad_s)  \norm{\Theta}{H^{-1}(\mathscr{B})}\sup_{\zz \in \mcE_{\vec{\ellrad_s}}^\finiteset}
 \norm{\vec{\Delta}^{\textnormal{det}} [u^{\valpha}(\cdot,\zz)]}{H^1(\mathscr{B},\mathbb{C})}
  e^{- \sum_{j \geq 1} g(\ellrad_{s,j}) m(\beta_j-1)},
 \end{align*}
 where $\finiteset \subset \nset_+$ denotes the support of $\vbeta-\oone$.
 Next, assuming that the spatial discretization consists of piecewise
 linear finite elements with spatial mesh sizes
 \eqref{eq:choice_of_h_and_FEM} and combining the a-priori bounds on
 the decay of the difference operators coming from the Combination
 Technique theory (see, e.g., \cite[proof of
 Theorem2]{griebel.harbrecht:tensor}) with
 \eqref{eq:mixed_to_standard_H_inclusions} and the fact that
 $u \in H^{1+s}(\mathscr B)$ for any $s=0, 1,\ldots,s_{\max}$,
 we have the following bound for every $\zz$ in the Bernstein polyellipse, $\mcE_{\vec{\ellrad}_s}^\finiteset$:
 \begin{align}
   \norm{\vec{\Delta}^{\textnormal{det}} [u^{\valpha}(\cdot,\zz)]}{H^1(\mathscr{B},\mathbb{C})}
   & \leq C_{CT} \norm{u(\cdot,\zz)}{\mathcal{H}^{\oone+s\qq}(\mathscr{B},\mathbb{C})}2^{-\sum_{i=1}^D \alpha_i \min\{1,q_i s\}} \\
   & \leq C_{CT} C_{s,\qq}\norm{u(\cdot,\zz)}{H^{1+s}(\mathscr{B},\mathbb{C})}2^{-\valpha \cdot \vec{r}_{\mathrm{FEM}}(s\qq)},
 \end{align}
 for $\qq \in \rset_+^D$ s.t. $q_i$ s.t. $|\qq|=1$, some $C_{CT}>0$ independent of $u$,
 and where $C_{s,\qq}$ is the embedding constant between
 $\mathcal{H}^{\oone+s\qq}(\mathscr{B},\mathbb{C})$ and $H^{1+s}(\mathscr{B},\mathbb{C})$.
 We then have the following bound:
 \begin{align*}
 \Delta E_{\valpha,\vbeta}
 & \leq C_{\mathrm{SE}}(\vec \ellrad_s) \norm{\Theta}{H^{-1}(\mathscr{B})} C_{CT}C_{s,\qq}
 \sup_{\zz \in \mcE_{\vec{\ellrad}_s}^\finiteset} \norm{u(\cdot,\zz)}{H^{1+s}(\mathscr{B},\mathbb{C})}
  e^{- \sum_j g_{s,j} m(\beta_j-1)} 2^{-\valpha \cdot \vec{r}_{\mathrm{FEM}}(s\qq)},
 \end{align*}
 where the constant, $C_{\mathrm{SE}}(\vec \ellrad_s)$ is bounded
 independently of $\vbeta$, thanks again to
 Lemma~\ref{lemma:deltaE_stoc_decay}. The proof is then concluded by
 recalling that
 $\sup_{\zz \in \mcE_{\vec{\ellrad}_s}^\finiteset}
 \norm{u(\cdot,\zz)}{H^{1+s}(\mathscr{B},\mathbb{C})} \leq C_{s,u}$
 independently of $\vbeta$ and $\finiteset$ due to
 Lemmas~\ref{lemma:polyell_analyticity}
 and~\ref{lemma:polyell_analyticity_high-order}.
\end{proof}

Observe that the result in Lemma 8 gives a bound on $\Delta E_{\valpha,\vbeta}$ parametric on the vector $\qq$.
The optimal choice for such $\qq$ will be discussed later on, in the proof of the main theorem,
namely Theorem \ref{theo:MISC_conv}.%

\begin{remark}[Relaxing the simplifying assumption]\label{rem:keep_it_simple}
  We now clarify why the assumption $p_s < \frac{1}{2}$, for
  $s=0,1,\ldots,s_{\max}$, is not essential.  Due to a suboptimal
  choice of $\ellrad_{s,j}$ in
  Lemma~\ref{lemma:polyell_analyticity_high-order} (and
  Lemma~\ref{lemma:polyell_analyticity} for $s=0$), the sequence
  $\{e^{-g_{s,j} }\}_{j \geq 1}$ in Lemma~\ref{lemma:gn_summ}, which
  is related to the solution $u$ and which appears in the proof of the
  MISC convergence theorem, has worse summability $\tilde{p}_s=\frac{p_s}{1-p_s}$
  than the sequence $\{b_{s,j} \}_{j \geq 1}$, whose summability coefficient is $p_s$,
  and which is related to the  diffusion coefficient, $a$.
  As we see in the main theorem stated below, we need $\tilde{p}_s<1$
  to guarantee  convergence of MISC, which implies $p_s < \frac{1}{2}$. %
  By choosing the polyellipses in Lemmas~\ref{lemma:polyell_analyticity}
  and~\ref{lemma:polyell_analyticity_high-order} by the more elaborate
  strategy presented in \cite{cohen.devore.schwab:nterm2},
  it would be possible to obtain the better summability $\tilde{p}_s=p_s$
  for the sequence $\{e^{-g_{s,j} }\}_{j \geq 1}$, which would only imply
  the less stringent condition $p_s<1$ and a better estimate for the MISC
  convergence rate.  However,
  for ease of exposition, we maintain the sub-optimal choice, which is
  enough for the purpose of presenting the argument that proves
  convergence of MISC.  The restriction $p_s<\frac{1}{2}$ formally
  prevents us from applying the MISC convergence analysis to diffusion
  coefficients with low spatial regularity. In practice, we see in
  Section~\ref{s:numerics} that the convergence estimates are
  numerically verified beyond this restriction.
\end{remark}

Before proving the main theorem of this section, we finally need the following technical lemma.

\begin{lemma}[Bounding a sum of double exponentials]\label{lemma:double_exp_sum}
  For $a>0$, $b \geq 2$ and $0\le c < a b$,
  \[
  \sum_{k=1}^\infty e^{-a b^k+c k} \leq e^{-ab + \varepsilon(a,b,c)}
  \]
  holds, where for each fixed $c$ and $b$, $\varepsilon(\cdot,b, c)$ is a
  monotonically decreasing, strictly positive function with %
  $\varepsilon(a,b,c) \rightarrow c \mbox{ as } a \rightarrow +\infty$.
\end{lemma}
\begin{proof}
  \begin{align*}
    \sum_{k\ge 1} e^{-a b^k+c k}
    &= e^{-ab+c} + \sum_{k\ge 2} e^{-a b^k+c k}
      =    e^{-ab+c} + \sum_{k\ge 1} e^{-a b^{k+1}+c(k+1)} \\
    &=   e^{-ab} \left( e^{c}+  e^{c}\sum_{k\ge 1} e^{-a b (b^{k}-1) + ck} \right).
  \end{align*}
  We observe that for $b\geq 2$ we have $b^k-1 \geq k$ for $k \geq 1$
  integer. Therefore, $e^{-a b (b^{k}-1)} \leq e^{-abk}$ and we have
  \[
  \sum_{k\ge 1} e^{-a b^k+c k}
  \leq e^{-ab} \left( e^{c}+  e^{c} \sum_{k\ge 1} e^{k(c -a b)} \right)
  = e^{-ab} \left( e^{c}+  \frac{e^{2c-ab}}{1-e^{c-ab}}  \right).
  \]
  Then,
  \[
  \varepsilon(a,b,c) = \log\left( e^{c}+ \frac{e^{2c-ab}}{1-e^{c-ab}}\right),
  \]
  and we finish by verifying that the function, $\varepsilon$, has the required properties.
\end{proof}

\noindent We are now ready to state and prove our main result.

\begin{theorem}[MISC convergence theorem]\label{theo:MISC_conv}
  Under Assumptions~\ref{assump:shift}---\ref{assump:work}, the
  profit-based MISC estimator, $\mathscr{M}_\mathcal{I}$, built using
  the set $\mathcal{I}$ defined in \eqref{eq:opt_set}, Stochastic
  Collocation with Clenshaw-Curtis points as in \eqref{eq:CC_def}-\eqref{eq:m_CC},
  and spatial discretization
  obtained by tensorizing multi-linear piecewise finite element spaces on
  each mesh, $\{\mesh_i\}_{i=1}^D$,
  with mesh-sizes $h_i$ as in \eqref{eq:choice_of_h_and_FEM}
  for solving the deterministic problems, satisfies, for any $\delta > 0$,
  \[
    \big| \E{\qoi} - \mathscr{M}_\mathcal{I}[\qoi] \big| \leq
    C_{\delta}\work{\mathscr{M}_{\mathcal{I}}}^{-r_{\mathrm{MISC}}
    + \delta},
  \]
  for some constant $C_\delta>0$ that is independent of
  $\work{\mathscr{M}_{\mathcal{I}}}$ and
  tends to infinity as $\delta \to 0$. Moreover, $\work{\mathscr{M}_{\mathcal{I}}}$ is
  given by \eqref{eq:work_decomp}, and
  \[
  r_{\mathrm{MISC}} = \max_{s=0,\ldots,s_{\max}}
  \begin{cases}
    r_{\mathrm{det}(s)},
    &\text{if } r_{\mathrm{det}}(s) \leq \frac{1}{p_s} - 2,\\
    \left( \frac{1}{p_0}-2 \right) \left( 1 +
      \frac{1}{r_{\mathrm{det}}(s)} \left( \frac{1}{p_0} -
        \frac{1}{p_s} \right) \right)^{-1}, & \text{otherwise,}
  \end{cases}
  \]
where
\[
  r_{\mathrm{det}}(s) =
  \min \left\{ \frac{1}{\max_{i=1,\ldots,D} \gamma_i d_i} , \frac{s}{\sum_{j=1}^D \gamma_j d_j} \right\}.
\]
\end{theorem}

\begin{proof}
In this proof, we use Theorem~\ref{theo:general-conv-result}. Therefore,
we need to estimate the $p$-summability of the weighted profits for some $p<1$.
To this end, we use Lemma~\ref{lemma:error}.
Observe that Lemma~\ref{lemma:error} provides a family of bounds for each $\Delta E_{\valpha,\vbeta}$, depending on $s$;
therefore we would then ideally choose the best $s$ for each $\Delta E_{\valpha,\vbeta}$.
However, this optimization problem is too complex and we simplify it
by assuming that
\begin{itemize}
	\item only two values of $s$ will be considered, $s=0$ and $s=s^*$ (which will not
      necessarily coincide with $s_{\max}$);
    \item the optimization between $s=0$ and $s=s^*$
      will not be carried out individually on each $\Delta E_{\valpha,\vbeta}$, but we will
      rather take a ``convex combination'' of the two corresponding estimates and
      choose the best outcome only at the end of the proof.
\end{itemize}
To this end, we first need to rewrite the result of the Lemma~\ref{lemma:error} in a more suitable form
for any fixed $s^* \in \{1,2,\ldots,s_{\max}\}$; note that from here on, with a slight abuse of notation,
we drop the superscript $^*$ and simply use $s$ to denote the fixed value.
Thus, for such fixed $s$, consider the statement of Lemma~\ref{lemma:error},
let $C_E = \max\{C_{0}, C_{s}\}$, $\chi_{j,s} = g_{s,j} \log_2 e$, and $\theta_{j,s} = (g_{0,j} - g_{s,j}) \log_2 e$
and combine \eqref{eq:deltaE_bound_with_min}--\eqref{eq:deltaE_factors_beta_s}, obtaining
  \begin{align*}
    \Delta E_{\valpha,\vbeta}
    & \leq \min_{t=\{0,s\}} C_t \Delta E_{\valpha}^{\textnormal{det}}(t)\Delta E_{\vbeta}^{\textnormal{stoc}}(t)
      \leq  C_E	\min_{\eta\in \{0,1\}} 2^{-\eta \vec{r}_{\mathrm{FEM}}(s\qq) \cdot \valpha}
    	\prod_{j \geq 1} e^{-m(\beta_j-1) [g_{s,j} + (1-\eta)(g_{0,j}-g_{s,j})]} \\
    & = C_E	\min_{\eta\in \{0,1\}} 2^{-\eta \vec{r}_{\mathrm{FEM}}(s\qq) \cdot \valpha} \prod_{j\geq 1} 2^{-m(\beta_j-1) [g_{s,j} + (1-\eta)(g_{0,j}-g_{s,j})] \log_2 e} \\
    & = C_E	\min_{\eta\in \{0,1\}} 2^{-\eta \vec{r}_{\mathrm{FEM}}(s\qq) \cdot \valpha } \prod_{j\geq 1} 2^{-m(\beta_j-1) [\chi_{j,s} + (1-\eta)\theta_{j,s}] } \\
    & = C_E 2^{- \sum_{j\ge 1} m(\beta_j-1) \chi_{j,s}- \max_{\eta\in \{0,1\}}\left( \eta \vec{r}_{\mathrm{FEM}}(s\qq) \cdot \valpha+\sum_{j\ge 1} m(\beta_j-1)(1-\eta)\theta_{j,s} \right) } \\
    & = C_E 2^{ - \sum_{j\ge 1} m(\beta_j-1) \chi_{j,s}- \max \left\{\sum_{j\ge 1} m(\beta_j-1) \theta_{j,s},\,\,\vec{r}_{\mathrm{FEM}}(s\qq) \cdot \valpha  \right\} }.
  \end{align*}
for an arbitrary $\qq \in \rset_+^D$ with $|\qq| = 1$ that we will choose later.
We can now bound the weighted sum of the profits as follows:
\begin{align}
  &\sum_{[\valpha,\vbeta] \in \nset_+^D \times \finitesuppset_+} P_{\valpha,\vbeta}^p \Delta W_{\valpha, \vbeta}
    = \sum_{[\valpha,\vbeta] \in \nset_+^D \times \finitesuppset_+} \Delta E_{\valpha,\vbeta}^p \Delta W_{\valpha}^{1-p} \Delta W_{\vbeta}^{1-p} \nonumber \\[7pt]
  &\leq C_E^p C_W^{1-p} \sum_{[\valpha,\vbeta] \in \nset_+^D \times \finitesuppset_+}
    2^{-p[\max\{\vec{r}_{\mathrm{FEM}}(s\qq) \cdot \valpha,\sum_{j \geq 1} m(\beta_j-1)\theta_{s,j}\} + \sum_{j \geq 1}
    m(\beta_j-1)\chi_{s,j}]} \nonumber\\[-12pt]
    &\hskip 4cm \cdot 2^{(1-p) \sum_{i=1}^D  \gamma_i d_i \alpha_i +(1-p) \sum_{j \geq 1} (\beta_j-1)} \nonumber \\[7pt]
  &= C_E^p C_W^{1-p} \sum_{[\valpha,\vbeta] \in \nset_+^D \times \finitesuppset_+} \min_{\lambda \in [0,1]}
    2^{-p[\lambda \vec{r}_{\mathrm{FEM}}(s\qq) \cdot \valpha +
    (1-\lambda) \sum_{j \geq 1} m(\beta_j-1)\theta_{s,j} + \sum_{j \geq 1}
    m(\beta_j-1) \chi_{s,j}]} \nonumber\\[-12pt]
  & \hskip 5cm \cdot 2^{(1-p)\sum_{i=1}^D  \gamma_i d_i \alpha_i +(1-p) \sum_{j \geq 1}(\beta_j-1)}
    \nonumber \\[7pt]
  &\leq C_E^p C_W^{1-p} \min_{\lambda \in [0,1]} \sum_{[\valpha,\vbeta] \in \nset_+^D \times \finitesuppset_+}
    2^{-p[\lambda \vec{r}_{\mathrm{FEM}}(s\qq) \cdot \valpha +
    (1-\lambda) \sum_{j \geq 1} m(\beta_j-1)\theta_{s,j} + \sum_{j \geq 1}
    m(\beta_j-1) \chi_{s,j}]} \nonumber \\[-12pt]
    & \hskip 5cm \cdot 2^{(1-p) \sum_{i=1}^D  \gamma_i d_i \alpha_i +(1-p) \sum_{j \geq 1}(\beta_j-1)}
    \nonumber \\[5pt]
  & = C_E^p C_W^{1-p} \min_{\lambda \in [0,1]}
	\left(\prod_{i=1}^D \sum_{k=1}^\infty 2^{-[p(\lambda
    r_{\mathrm{FEM}}(s\qq)_i+\gamma_i d_i) -\gamma_i d_i] k } \right)
    \nonumber \\
    & \hskip 3cm \cdot \left(\prod_{j=1}^\infty \sum_{k=1}^\infty 2^{-pm(k-1)((1-\lambda)
    \theta_{s,j} + \chi_{s,j}) +(1-p)(k-1)} \right).
    \label{eq:main_theo_opt}
\end{align}
We then investigate under what conditions each of the two factors is
finite (the constants $C_E,C_W$ are bounded, cf.
Lemmas~\ref{lemma:work} and~\ref{lemma:error}). Before proceeding,
we comment on the equations above. As we already
mentioned at the beginning of the proof, here we are working in a
suboptimal setting in which instead of choosing a different $s$ for
each $\Delta E_{\valpha,\vbeta}$, we restrict ourselves to
choosing between only two values, $s=0$ or a certain $s>0$
(second line). Observe that we have an equality between the second and the third
line since $2^x$ is a monotone function of $x$. Hence, the minimum
is always attained at either $\lambda=0$ or $\lambda=1$. However,
when it comes to switching the order of the sum and the minimum in
the fourth line, i.e., bounding the sum by choosing the same
$\lambda$ to bound every term in the sum, allowing for fractional
$\lambda$ gives a tighter bound on the overall sum than just
considering $\lambda \in \{0,1\}$. Roughly speaking, we are somehow
``mimicking'' the fact that the optimal bound of the sum would use a
different value of $\lambda$ for every term by choosing an overall
$\lambda$ that is between the two possible values.

To investigate the condition for which each of the two factors in \eqref{eq:main_theo_opt} are bounded,
we immediately have for the first factor for all $i=1,\ldots, D$
that
\begin{equation}\label{eq:sum_constraints}
  {p\left(\lambda r_{\mathrm{FEM}}(s\qq)_i+\gamma_i d_i\right) -\gamma_i d_i}  > 0 \quad
  \Longrightarrow\quad
  p > \frac{\gamma_i d_i}{\lambda r_{\mathrm{FEM}}(s\qq)_i + \gamma_i d_i } =
  \left( \lambda \frac{r_{\mathrm{FEM}}(s\qq)_i}{\gamma_i d_i} + 1\right)^{-1}
  \,\,.
\end{equation}
Our goal is to optimize the above constraint for the summability exponent $p$.
To this end, we will minimize the right hand side of \eqref{eq:sum_constraints},
observing that it decreases with respect to $r_{\mathrm{FEM}}(s\qq)_i$.
Hence, recalling the dependence of $r_{\mathrm{FEM}}(s\qq)_i$
(cf. Lemma \ref{lemma:error}) on the vector of weights $\qq$ we consider
\begin{align*}
r_{\mathrm{det}}(s)
& = \max_{\qq\in \rset^D_+, |\qq|=1}\min_{i=1,\ldots,D} \frac{r_{\mathrm{FEM}}(s\qq)_i}{\gamma_i d_i} \\ %
& = \max_{\qq\in \rset^D_+, |\qq|=1}\min_{i=1,\ldots,D} \min\left\{ \frac{1}{\gamma_i d_i} , \frac{s q_i}{\gamma_i d_i} \right\} \\
& = \max_{\qq\in \rset^D_+, |\qq|=1}\min\left( \frac{1}{\max_{i=1,\ldots,D}\gamma_i d_i} , \min_{i=1,\ldots,D}\frac{s q_i}{\gamma_i d_i} \right) \\
& = \min \left( \frac{1}{\max_{i=1,\ldots,D} \gamma_i d_i} , \max_{\qq\in \rset^D_+, |\qq|=1} \min_{i=1,\ldots,D} \frac{s q_i}{\gamma_i d_i} \right)
\end{align*}
which is maximized by making $\frac{s q_i}{\gamma_i d_i}$ constant over $i$, i.e.
\[
q_i = \frac{\gamma_i d_i }{\sum_{j=1}^D \gamma_j d_j} \Rightarrow
r_{\mathrm{det}}(s) =
\min \left\{ \frac{1}{\max_{i=1,\ldots,D} \gamma_i d_i} , \frac{s}{\sum_{j=1}^D \gamma_j d_j} \right\}.
\]
With this optimal choice then \eqref{eq:sum_constraints} becomes simply%
\begin{equation}\label{eq:pspatial}
  p > \left( \lambda r_{\mathrm{det}}(s) + 1\right)^{-1}\,\,.
\end{equation}

For the second factor, denoting the generic term of the inner
sum as $a_{j,k}$ for brevity and observing that $a_{j,1} = 1$ for
every $j$, we have
  \begin{align*}
    \prod_{j=1}^\infty \sum_{k=1}^\infty a_{j,k}
    &	\leq \prod_{j=1}^\infty \left( 1 + \sum_{k=2}^\infty a_{j,k} \right)
     	= \exp \left(\sum_{j=1}^\infty \log \left( 1 + \sum_{k=2}^\infty a_{j,k} \right) \right)
     	\leq \exp \left(\sum_{j=1}^\infty \sum_{k=2}^\infty a_{j,k}\right).
  \end{align*}
We thus only have to discuss the convergence of the sum
\begin{equation}\label{eq:summab_disc1}\aligned
  \sum_{j=1}^\infty\sum_{k=2}^\infty 2^{-pm(k-1)[(1-\lambda) \theta_{s,j} +  \chi_{s,j}] +(1-p)(k-1)}
  &= \sum_{j=1}^\infty\sum_{k=1}^\infty 2^{-pm(k)[(1-\lambda) \theta_{s,j} +  \chi_{s,j}] +(1-p)k}\\
  &\leq\sum_{j=1}^\infty\sum_{k=1}^\infty 2^{-p 2^{k-1}[(1-\lambda) \theta_{s,j} +  \chi_{s,j}] +(1-p)k}\,,
\endaligned
\end{equation}
where the last step is a consequence of the fact that, for
Clenshaw-Curtis points,
$m(k) \geq 2^{k-1}$ for $k \geq 1$, cf. \eqref{eq:m_CC}. Moreover,
$(1-\lambda) \theta_{s,j} + \chi_{s,j} \geq 0$. To study the
summability of \eqref{eq:summab_disc1}, we want to use
Lemma~\ref{lemma:double_exp_sum} to bound the inner sum in
\eqref{eq:summab_disc1}. First, we rewrite
\begin{align*}
\sum_{k=1}^\infty 2^{-p 2^{k-1}[(1-\lambda) \theta_{s,j} +  \chi_{s,j}] +(1-p)k}
&= \sum_{k=1}^\infty \exp\left( -p \frac{\log 2}{2}[(1-\lambda) \theta_{s,j} + \chi_{s,j}] 2^k +(1-p)k \log 2 \right)  \\
&= \sum_{k=1}^\infty \exp\left( -a2^k +c k \right) \\
& \mbox{ with }
a = p \frac{\log 2}{2}[(1-\lambda) \theta_{s,j} + \chi_{s,j}] > 0, \quad
c=(1-p)\log 2 > 0,
\end{align*}
where we have used the notation in Lemma~\ref{lemma:double_exp_sum}. Note that this lemma
holds true under the assumptions that $a>0$ and $ 0 \leq c < 2a$, where the
latter has to be verified as follows
\[
2a > c
\Leftrightarrow
p \log 2[(1-\lambda) \theta_{s,j} + \chi_{s,j}] > (1-p)\log 2
\Leftrightarrow
(1-\lambda) \theta_{s,j} + \chi_{s,j} > \frac{(1-p)}{p},
\]
which is true whenever
\[
\chi_{s,j} > r_{\text{det}}(s),
\]
due to \eqref{eq:pspatial}, $\theta_{s,j} \geq 0$ and $\lambda \leq 1$.
Define $\overline{\mathcal J} = \{ j \geq 1 \::\: \chi_{s,j} \leq r_{\text{det}}(s)\}$ which has a finite cardinality since
$\chi_{s,j} \to \infty$ as $j \to \infty$.
Resuming from \eqref{eq:summab_disc1}, we have, due to  Lemma~\ref{lemma:double_exp_sum},
\begin{align*}
  \sum_{j=1}^\infty \sum_{k=1}^\infty 2^{-p 2^{k-1}[(1-\lambda) \theta_{s,j} +  \chi_{s,j}] +(1-p)k}
  & \leq C(\overline{\mathcal J}) +  \sum_{j \notin \overline{\mathcal
    J}} \sum_{k=1}^\infty \exp\left( -a2^k +c k \right) \\
  & \leq C(\overline{\mathcal J}) + \sum_{j \notin \overline{\mathcal
    J}} e^{-2a + \varepsilon(a,2,c)},
\end{align*}
where $C(\overline{\mathcal J})$ is bounded, since $\# \overline{\mathcal J} < \infty$, and $\varepsilon(a,2,c)$
is a monotonically decreasing function with limit $c = (1-p) \log 2$
independent of $j$. Therefore, the previous series converges if and only if
\[
\sum_{j \notin \overline{\mathcal J}}^\infty e^{- 2a}
= \sum_{j \notin \overline{\mathcal J}}^\infty e^{-p \log 2[(1-\lambda) \theta_{s,j} + \chi_{s,j}]}
= \sum_{j \notin \overline{\mathcal J}}^\infty 2^{-p [(1-\lambda) \theta_{s,j} + \chi_{s,j}]}
\]
converges. Inserting the expression of $\theta_{s,j}$ and $\chi_{s,j}$, we get
\begin{align*}
\sum_{j \notin \overline{\mathcal J}}^\infty 2^{-p [(1-\lambda) \theta_{s,j} + \chi_{s,j}]}
& = \sum_{j \notin \overline{\mathcal J}}^\infty 2^{-p [(1-\lambda) (g_{0,j}-g_{s,j}) + g_{s,j}] \log_2 e}
= \sum_{j \notin \overline{\mathcal J}}^\infty e^{-p [(1-\lambda) (g_{0,j}-g_{s,j}) + g_{s,j}]} \\
& =  \sum_{j \notin \overline{\mathcal J}}^\infty e^{-p (1-\lambda) g_{0,j}}e^{ - p \lambda g_{s,j}}.%
\end{align*}
After applying the H\"older inequality in the previous summation with
exponents $\summab_1^{-1}+\summab_2^{-1}=1$
we need to simultaneously ensure the boundedness of the following sums:
\[
\sum_{j \notin \overline{\mathcal J}}^\infty e^{-p (1-\lambda) g_{0,j} \summab_2}
\quad\text{and}\quad
\sum_{j \notin \overline{\mathcal J}}^\infty e^{ - p \lambda g_{s,j} \summab_1}.
\]
Recalling the summability result in Lemma~\ref{lemma:gn_summ}, we
understand that the following two conditions must hold:
\[
\begin{cases}
  \displaystyle p(1-\lambda)\summab_2 \geq \frac{p_0}{1-p_0} \\[10pt]
  \displaystyle p\lambda \summab_1 \geq \frac{p_s}{1-p_s}
\end{cases}
\Rightarrow
\begin{cases}
  \displaystyle p \geq \frac{p_0}{1-p_0} \frac{1}{1-\lambda} \frac{1}{\summab_2}\\[10pt]
  \displaystyle p \geq \frac{p_s}{1-p_s}\frac{1}{\lambda}\left(1 - \frac{1}{\summab_2} \right),
\end{cases}
\]
which closes the discussion of the summability of the second factor of
\eqref{eq:main_theo_opt} for a fixed $s$.
Recalling the constraint \eqref{eq:pspatial} coming from the first factor of \eqref{eq:main_theo_opt},
we finally have to solve the following optimization problem:
\[%
  p >\min_{ \lambda \in [0,1], 1\leq \summab_2}\max\left\{
    \left( \lambda r_{\mathrm{det}}(s) + 1\right)^{-1},
    \frac{p_s}{1-p_s}\frac{1}{\lambda}\left(1 - \frac{1}{\summab_2} \right),
    \frac{p_0}{1-p_0} \frac{1}{1-\lambda} \frac{1}{\summab_2} \right\}
\]%
i.e., we have to choose $\summab_2$ and $\lambda$ to minimize the lower
bound on $p$.  We first optimally select $\summab_2$ given $\lambda$, i.e.,
we take $\summab_2 = \summab_2^*$ such that
\[
\frac{p_s}{1-p_s}\frac{1}{\lambda}\left(1 - \frac{1}{\summab_2^*} \right) = \frac{p_0}{1-p_0} \frac{1}{1-\lambda} \frac{1}{\summab_2^*}
\,\, \Rightarrow \,\,
\summab_2^* = 1+\frac{1-p_s}{p_s} \frac{p_0}{1-p_0} \frac{\lambda}{1-\lambda}
\]
Substituting back, we obtain
\begin{align*}
  \frac{p_s}{1-p_s}\frac{1}{\lambda} \frac{\summab_2^*-1}{\summab_2^*}
  &=  \left( {\frac{1}{p_0} - 1 + \lambda \left( \frac{1}{p_s} - \frac{1}{p_0} \right)} \right)^{-1},
\end{align*}
so that the minimization problem reads
\begin{equation}\label{eq:p_bound_semiopt}\aligned
  p &>\min_{\lambda \in [0,1]}
  \max\left\{f_{1}(\lambda,s),f_2(\lambda, s) \right\}, \\
  & f_{1}(\lambda,s) = \left( \lambda r_{\mathrm{det}}(s) + 1\right)^{-1}, \quad
  f_2(\lambda, s) = \left( {\frac{1}{p_0} - 1 + \lambda \left( \frac{1}{p_s} - \frac{1}{p_0} \right)} \right)^{-1}.
\endaligned
\end{equation}
Now, we recall that $p_0 \le p_s$. Hence, $f_2(\lambda, s)$ is
increasing with $\lambda$. Conversely, $f_{1}(\lambda,s)$ is
decreasing with $\lambda$ since %
  $r_{\mathrm{det}}(s)$ is a positive number.  Furthermore, notice
that we cannot have $f_{1}(\lambda,s)<f_2(\lambda,s)$ for all
$\lambda \in [0,1]$.  Indeed, the previous condition is equivalent to
$f_{1}(0, s) \leq f_2(0, s)$, i.e.,
$1 \leq \frac{p_0}{1-p_0} \Rightarrow p_0 \geq \frac{1}{2}$, which
does not satisfy Assumption~\ref{assump:b_and_p}. Note that, in this
case, the lower bound for $p$ in \eqref{eq:p_bound_semiopt} is
minimized for $\lambda =0$, implying that $p > \frac{p_0}{1-p_0} > 1$,
i.e., the method does not converge (cf. the statement of
Theorem~\ref{theo:general-conv-result}).  Thus, we have only two cases
(see also Figure~\ref{fig:main_theo_cartoon}):
\begin{description}
\item[Case 1] $f_{1}(\lambda,s)>f_2(\lambda, s)$ for all
  $\lambda \in [0,1]$, which means that the convergence of the method
  is dictated by the spatial discretization.  Given that
  $f_{1}$ is decreasing and $f_2$ is increasing, the
  previous condition is equivalent to $f_1(1, s) \geq f_2(1, s)$,
  i.e.,
  $ r_{\mathrm{det}}(s) \leq \frac{1}{p_s} - 2$.  In this case, the lower bound
  \eqref{eq:p_bound_semiopt} is minimized for $\lambda =1$, and we
  have
  $p > \left( r_{\mathrm{det}}(s) + 1\right)^{-1} $.
    \item[Case 2]
      There exists $\lambda^* \in (0,1)$ such that
        $f_{1}(\lambda,s)=f_2(\lambda, s)$.  This
      condition is equivalent to the two conditions
      \[
      \begin{cases}
        f_{1}(0,s) \geq f_2(0, s) \\
        f_{1}(1,s) \leq f_2(1, s)
      \end{cases}
      \Rightarrow
      \begin{cases}
        1 \leq \frac{1}{p_0} - 1 \\
        r_{\mathrm{det}}(s) \geq
        \frac{1}{p_s} - 2.
      \end{cases}
      \]
      Solving for $\lambda^*$ yields
      \begin{align*}
        \lambda^* &= \frac{\frac{1}{p_0}-2}
                  {r_{\mathrm{det}}(s) + \frac{1}{p_0} -
                  \frac{1}{p_s}} > 0
      \end{align*}
      which yields $p > \overline p$, where
      \[ \aligned \overline p &= \left( \left(\frac{\frac{1}{p_0}-2}
            {r_{\mathrm{det}}(s) + \frac{1}{p_0} -
              \frac{1}{p_s}}\right) r_{\mathrm{det}}(s) + 1\right)^{-1}\\
        &= \left(1 + \left( \frac{1}{p_0}-2 \right) \left( 1 +
            \frac{1}{r_{\mathrm{det}}(s)} \left( \frac{1}{p_0} -
              \frac{1}{p_s} \right) \right)^{-1}\right)^{-1} \,\,.
            \endaligned
      \]
\end{description}
Since the previous computations were carried out for a fixed
$s$, we can take the minimum over all possible values of
$s$.  Then, we can apply Theorem~\ref{theo:general-conv-result} and
derive the convergence estimate,
\[
\big| \E{\qoi} - \mathscr{M}_\mathcal{I}[\qoi] \big| \leq C_P(p)\work{\mathscr{M}_{\mathcal{I}}}^{1-1/p},
\]
where $1-1/p = 1 - \max_{s=0,\ldots, s_{\max}}(1/{\overline p}) + \delta$ for any $\delta > 0$, which we reformulate as
\[
  \big| \E{\qoi} - \mathscr{M}_\mathcal{I}[\qoi] \big| \leq
  C_P\left(\frac{1}{1+r_{\mathrm{MISC}} -
      \delta}\right)\work{\mathscr{M}_{\mathcal{I}}}^{-r_{\mathrm{MISC}}
  + \delta},
\]
with $r_{\mathrm{MISC}}  = \max_{s=0,\ldots, s_{\max}}({1}/{\overline p}) -1$.
\end{proof}

\begin{figure}[t]
  \centering
  \includegraphics[width=0.45\textwidth]{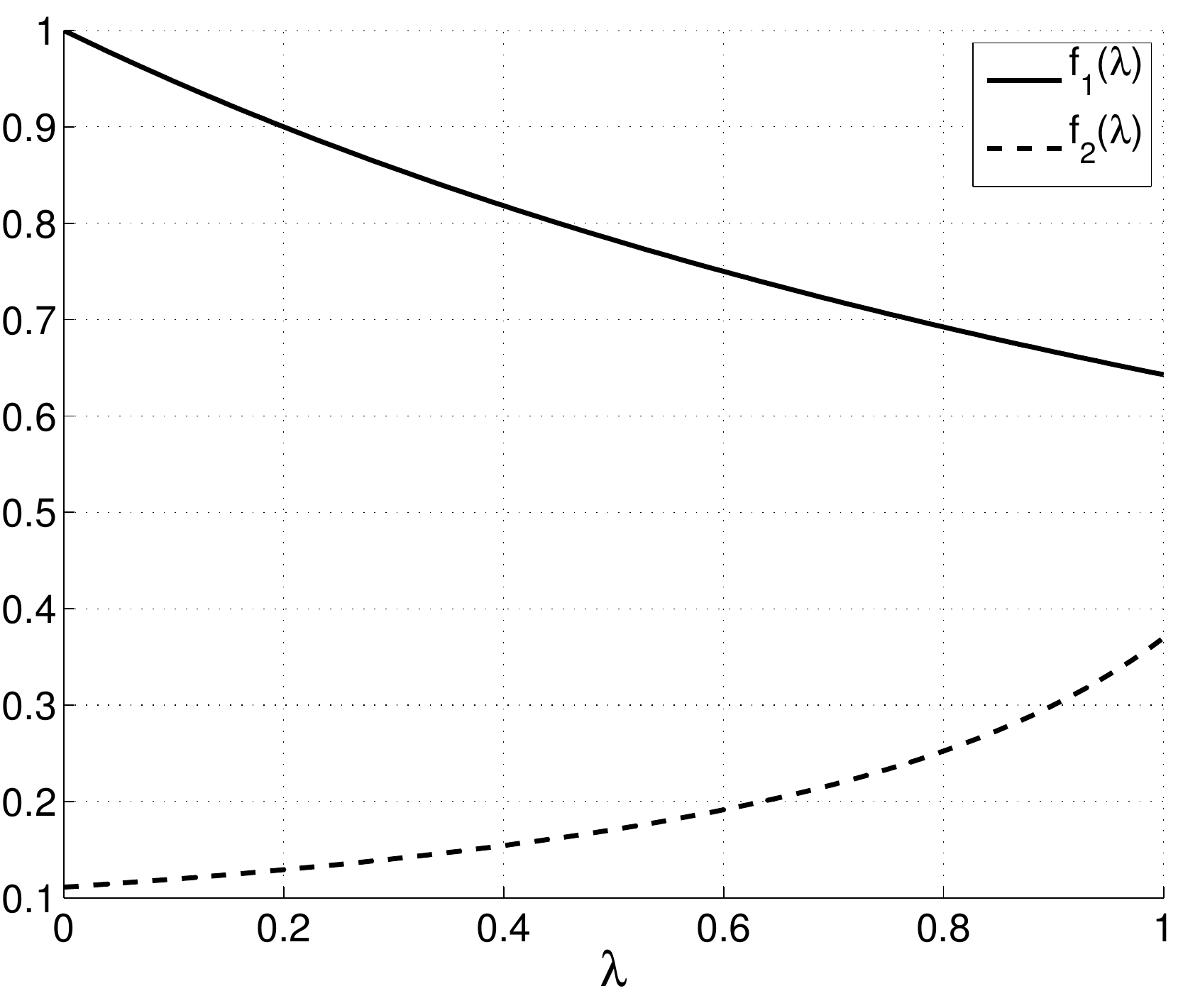}
  \includegraphics[width=0.45\textwidth]{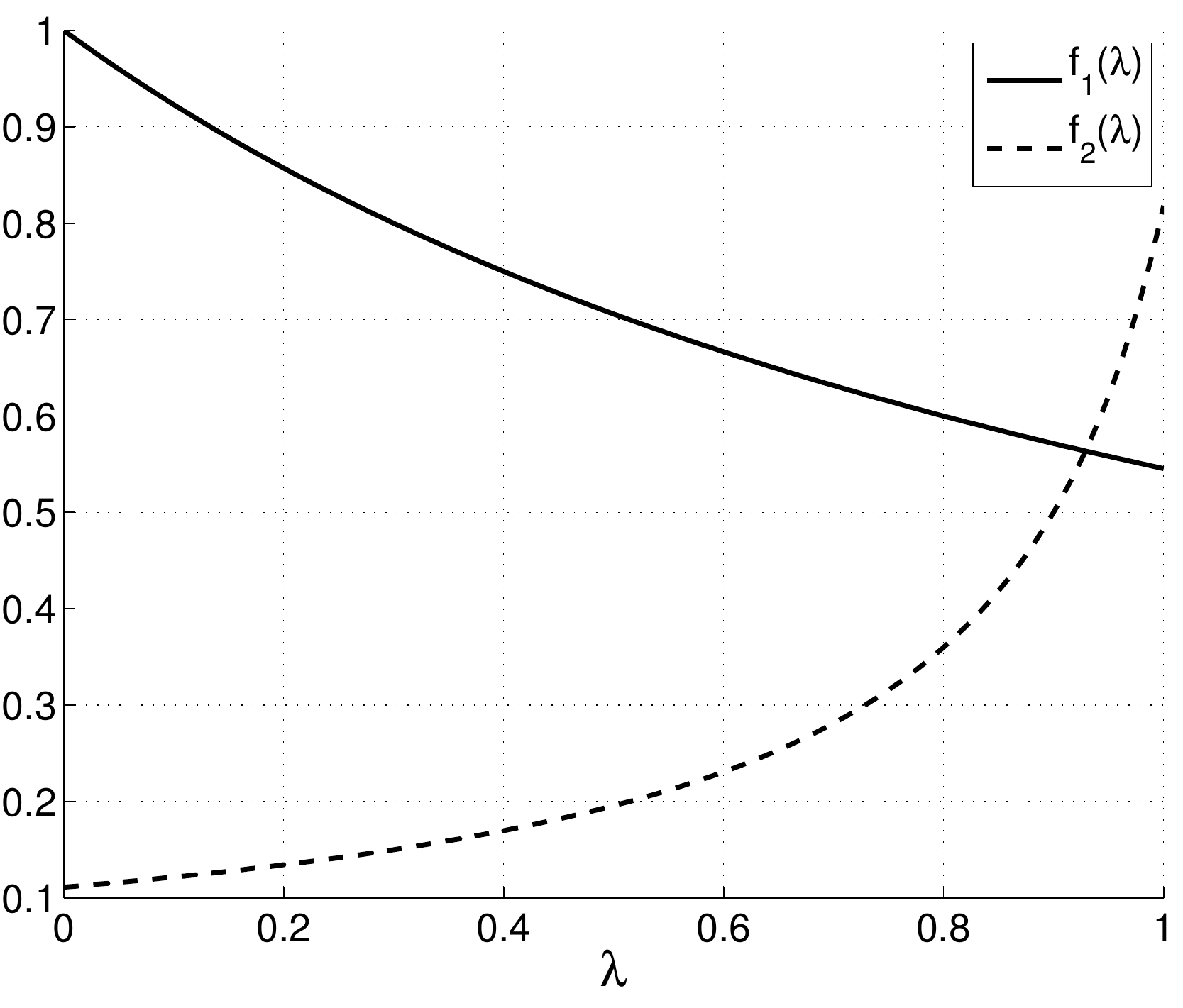}
  \caption{Illustration of the optimization problem
      \eqref{eq:p_bound_semiopt}. As observed in the proof,
      $f_{1}$ is decreasing with $\lambda$ while
      $f_2$ is increasing with
      $\lambda$.  \emph{Left}: case 1 of the proof: the minmax point is
      $\lambda=1$; \emph{Right}: case 2 of the proof: the minmax point is
      $\lambda<1$.  }
  \label{fig:main_theo_cartoon}
\end{figure}

\section{Analysis of Example~\ref{ex:logunif}}\label{s:example_analysis}
In this section, we determine the value of $s_{\max}$ and the sequence
$\{p_s\}_{s=0}^{s_{\max}}$ for
Example~\ref{ex:logunif}.
Since we will work with localized quantities of interest
far from the boundary, cf. equation \eqref{eq:qoi_def} written below, we believe that the effect of the boundary is
negligible and the regularity $s_{\max}$ is mainly limited by the summability
properties of $\kappa$. %
See Appendix~\ref{sec:app-regularity} for a
slightly modified problem on the same domain where we can prove that the regularity
is only limited by the summability properties of $\kappa$. Let us define the following family of auxiliary functions,
\[
\Upsilon_{\vec k, \vec \ell}(\vec x) = \prod_{i=1}^\exmpld
    \left(\cos\left({\pi }  k_i  x_i \right)\right)^{\ell_i}
    \left(\sin\left({\pi }  k_i  x_i \right)\right)^{1-\ell_i},
\]
so that $\kappa$ from \eqref{eq:our_matern} can be written as
\begin{align*}
\kappa(\vec x, \vec y ) &=
\sum_{\vec k \in \nset^\exmpld} A_{\vec k}
    \sum_{\vec \ell \in \{0,1\}^\exmpld} y_{\vec k,\vec \ell }
                          \Upsilon_{\vec k, \vec \ell}  (\vec x)
  \\
  &= \sum_{j=0}^\infty \sum_{\left\{\vec k \in \nset^\exmpld \::\: |\vec k| =
    j\right\}} A_{\vec k}
    \sum_{\vec \ell \in \{0,1\}^\exmpld} y_{\vec k,\vec \ell }
    \Upsilon_{\vec k, \vec \ell} (\vec x).
 \end{align*}
 Based on this expression, for $s \geq 0$, we analyze the summability
 of
 $\left\{A_{\vec k} \|D^{\vec s}\Upsilon_{\vec k, \vec
     \ell}\|_{L^\infty(\mathscr B)} \right\}$ for $|\vec s| = s$ to
 determine the permissible values of $p_s$. First, for $|\vec s| = s$,
 observe that for a constant $c$ independent of $\vec k$ we have
\[  \| D^{\vec s} \Upsilon_{\vec k, \vec \ell} (\vec x)\|_{L^{\infty}(\mathscr
    B)} = \prod_{j=1}^\exmpld  \left({\pi}k_j\right)^{s_j} \leq c |\vec k|^s. \]
Then for all $s \geq 0$, we have
\[
  \aligned \sum_{j=0}^\infty \sum_{\left\{\vec k \in \nset^\exmpld \::\:
      |\vec k| = j\right\}}\sum_{\vec \ell \in \{0,1\}^\exmpld} A_{\vec
    k}^{p_s} \|D^{\vec s}\Upsilon_{\vec k, \vec \ell}\|_{L^\infty(\mathscr
    B)}^{p_s} &\leq c 2^{\exmpld} \sum_{j=0}^\infty \sum_{\left\{\vec k \in
      \nset^\exmpld \::\: |\vec k| = j\right\}} 2^{p_s {\frac{|\vec
        k|_0}{2}}} |\vec k|^{p_s s} (1 + |\vec
  k|^2)^{-\frac{p_s\left(\nu+\frac{\exmpld}{2}\right)}{2}} \\
  &\leq c 2^{\exmpld} + c 2^{\exmpld+p_s\frac{\exmpld}{2}}\sum_{j=1}^\infty \sum_{\left\{\vec k \in
      \nset^\exmpld \::\: |\vec k| =
      j\right\}} j^{-p_s\left(\nu+\frac{\exmpld}{2} - s\right)} \\
  &= c 2^{\exmpld} + c \frac{2^{\exmpld+p_s\frac{\exmpld}{2}}}{(\exmpld-1)!} \sum_{j=1}^\infty
  j^{-p_s\left(\nu+\frac{\exmpld}{2} - s\right)}
  \prod_{i=1}^{\exmpld-1} (j+i) \\
  &= c2^{\exmpld} + c \frac{2^{\exmpld+p_s\frac{\exmpld}{2}}}{(\exmpld-1)!} \sum_{j=1}^\infty
  j^{-p_s\left(\nu+\frac{\exmpld}{2} - s\right) + \exmpld-1}
  \left(1+\frac{\exmpld-1}{j}\right)^{\exmpld-1} \\
  &\leq c2^{\exmpld} + \frac{c2^{\exmpld+p_s\frac{\exmpld}{2}} \exmpld^{\exmpld-1}}{(\exmpld-1)!} \sum_{j=1}^\infty
  j^{-p_s\left(\nu+\frac{\exmpld}{2} - s\right) + \exmpld-1}.
\endaligned
\]
From here, we obtain the bound
\begin{equation}
\label{eq:sum1_p}
 p_s > \left(\frac{\nu}{\exmpld}+\frac{1}{2} - \frac{s}{\exmpld}\right)^{-1},
\end{equation}
for all $s \geq 0$. Moreover, imposing $p_0 < \frac{1}{2}$ and $p_{s_{\max}} < \frac{1}{2}$ gives the bounds
\begin{equation}
\label{eq:nu_s_bounds}
 \nu > \frac{3\exmpld}{2} \qquad \text{ and }\ \qquad s_{\max} < \nu -
  \frac{3\exmpld}{2},
\end{equation}
respectively.  Since
$\Upsilon_{\vec k, \vec \ell} \in C^{\infty}(\mathscr B)$, the bounds
in \eqref{eq:nu_s_bounds} are the only bounds on the value of
$s_{\max}$.  To determine an upper bound on the value of
$r_{\text{MISC}}$, up to a small $\delta$,
we set $\gamma_1=\ldots=\gamma_D=\gamma$ (motivated by the fact that all subdomains $\mathscr{B}_i$ are equal),
we substitute $r_{\text{FEM}, i}(s) = \min(1, \frac{s}{\exmpld})$ for $i=1,\ldots,d$
and the lower bound of $ p_s$ in Theorem~\ref{theo:MISC_conv} and
obtain after simplifying
\[ \aligned
r_{\mathrm{MISC}} &= \max_{s=0,\ldots,s_{\max}} \begin{cases}
\frac{\min(1, \frac{s}{d})}{\gamma} &
\text{ if } \frac{\min(1, \frac{s}{d})}{\gamma} \leq \frac{\nu}{d} - \frac{3}{2} -\frac{s}{d},\\
\left(\frac{\nu}{d} - \frac{3}{2} \right) \left( 1 +
   \frac{\gamma}{\min(1, \frac{s}{d})}\frac{s}{d}
    \right)^{-1} &
\text{ if } \frac{\min(1, \frac{s}{d})}{\gamma} \geq \frac{\nu}{d} - \frac{3}{2}
-\frac{s}{d},
\end{cases}\\
&=\max_{s=0,\ldots,s_{\max}}
\begin{cases}
  \frac{s}{d \gamma} & \text{ if } \frac{s}{d \gamma} \leq \frac{\nu}{d} -
  \frac{3}{2}
  -\frac{s}{d} \text{ and } s \leq d,\\
  \left(\frac{\nu}{d} - \frac{3}{2} \right) \left( 1 +
    \gamma \right)^{-1} & \text{ if }
\frac{s}{d \gamma} \geq \frac{\nu}{d} -
  \frac{3}{2}
  -\frac{s}{d} \text{ and } s \leq d,\\
  \frac{1}{\gamma} & \text{ if } \frac{1}{\gamma} \leq \frac{\nu}{d} - \frac{3}{2}
  -\frac{s}{d} \text{ and } s \geq d, \\
  \left(\frac{\nu}{d} - \frac{3}{2} \right) \left( 1 +
    \frac{s \gamma}{d} \right)^{-1} &\text{ if }
\frac{1}{\gamma} \geq \frac{\nu}{d} -
  \frac{3}{2}
  -\frac{s}{d} \text{ and } s \geq d.
\end{cases}
\endaligned
\]
Before continuing, we discuss the four branches of the previous
expression.  If $s_{\max} \leq d$, then only the first two branches
are valid. Since the rates in these two branches increase with $s$,
the maximum is achieved for $s=s_{\max}$. If $s_{\max} \geq d$,
then, since the rates in the third and fourth branches decrease with $s$, the
maximum is achieved for $s=d$. Hence
\[\aligned
  r_{\mathrm{MISC}} &=
\begin{cases}
  \frac{s_{\max}}{d \gamma} & \text{ if } \frac{s_{\max} (1+\gamma)}{d \gamma} \leq \frac{\nu}{d} -
  \frac{3}{2} \text{ and } s_{\max} \leq d  ,\\
  \left(\frac{\nu}{d} - \frac{3}{2} \right) \left( 1 + \gamma
  \right)^{-1} & \text{ if } \frac{s_{\max} (1+\gamma)}{d \gamma} \geq
  \frac{\nu}{d} - \frac{3}{2}
  \text{ and } s_{\max} \leq d , \\
  \frac{1}{\gamma} & \text{ if } \frac{1}{\gamma} \leq \frac{\nu}{d} -
  \frac{3}{2}
  -1\text{ and } s_{\max} \geq d, \\
  \left(\frac{\nu}{d} - \frac{3}{2} \right) \left( 1 + \gamma
  \right)^{-1} & \text{ if } \frac{1}{\gamma} \geq \frac{\nu}{d} -
  \frac{3}{2} -1 \text{ and } s_{\max} \geq d,
\end{cases}\\
&=
\begin{cases}
  \frac{1}{\gamma}\left(\frac{\nu}{d} - \frac{3}{2}\right) &\text{ if
  }
  \frac{1}{\gamma} \leq 0 \text{ and } \frac{\nu}{d} \leq \frac{5}{2}  ,\\
  \left(\frac{\nu}{d} - \frac{3}{2} \right) \left( 1 + \gamma
  \right)^{-1} &\text{ if } \frac{1}{\gamma} \geq 0
  \text{ and } \frac{\nu}{d} \leq \frac{5}{2} , \\
  \frac{1}{\gamma} &\text{ if } \ \frac{1}{\gamma} \leq \frac{\nu}{d}
  - \frac{5}{2}
  \text{ and } \frac{\nu}{d} \geq \frac{5}{2}, \\
  \left(\frac{\nu}{d} - \frac{3}{2} \right) \left( 1 + \gamma
  \right)^{-1} &\text{ if } \frac{1}{\gamma} \geq \frac{\nu}{d} -
  \frac{5}{2} \text{ and } \frac{\nu}{d} \geq \frac{5}{2},
\end{cases}\\
&=
\begin{cases}
  \gamma^{-1} & \text{ if } \
  \frac{\nu}{d} \geq \frac{1}{\gamma}  + \frac{5}{2}, \\
   \left(\frac{\nu}{d} - \frac{3}{2}
  \right)\left( 1 + \gamma \right)^{-1} & \text{ if } \frac{\nu}{d} \leq
  \frac{1}{\gamma} + \frac{5}{2}.
\end{cases}
\endaligned
\]
In Figure~\ref{fig:theory_work_rates}, we plot the upper bound of the
rate of MISC work complexity, $r_{\text{MISC}}$, based on
Theorem~\ref{theo:MISC_conv} and the following analysis variants:
\begin{description}
\item[\textbf{Theory.}] This is based on the summability properties of
  $\left\{A_{\vec k} \|D^{\vec s}\Upsilon_{\vec k, \vec
      \ell}\|_{L^\infty(\mathscr B)} \right\}$.  We also use the value
  $r_{\text{FEM},i}(s) = 2 \min\left(1, \frac{s}{\exmpld}\right)$ for
  all $i=1,\ldots,d$. This is motivated by the fact that we expect to
  double the convergence rate of the underlying FEM method since we
  are considering convergence of a smooth linear functional of the solution.
\item[\textbf{Square summability.}] Motivated by the arguments in
  Lemma~\ref{lem:Wa_inf_bound} in the appendix, we believe that our
  results may be improved by instead considering the summability
  properties of
  $\left\{A_{\vec k}^2 \|D^{\vec s}\Upsilon_{\vec k, \vec
      \ell}\|_{L^\infty(\mathscr B)}^2 \right\}$ for
  $|\vec s| \leq s$. Similar calculations yield the bounds
\begin{equation}
\label{eq:sum2_p}
p_s > \left(\frac{2\nu}{\exmpld}+1 - \frac{2s}{\exmpld}\right)^{-1},
\end{equation}
and the corresponding conditions, $\nu > \frac{\exmpld}{2}$ and
$s_{\max} < \nu - \frac{\exmpld}{2}$.
\item[\textbf{Improved.}] As mentioned in
  Remark~\ref{rem:keep_it_simple}, we could in principle make our
  results sharper by taking $\tilde{p}_s=p_s$ instead of
  $\tilde{p}_s=\frac{p_s}{1-p_s}$.  The modifications of
  Theorem~\ref{theo:MISC_conv} to account for these rates are
  straightforward. Moreover, when considering square summability, the
  conditions become $\nu > 0$ and $s_{\max} < \nu$.
\end{description}

We also include in Figure~\ref{fig:theory_work_rates} the observed
convergence rates of MISC when applied to the example with different
values of $\nu$, as discussed below in Section~\ref{s:numerics}, and
the observed convergence rate of MISC when applied to the same
problem with no random variables and a constant diffusion
coefficient, $a$. In the latter case, MISC reduces to a
deterministic combination technique
\cite{Bungartz.Griebel.Roschke.ea:pointwise.conv}. Note that the
rate of MISC with no random variables is an upper bound for the
convergence rate of MISC with any $\nu>0$.

From this figure, we can clearly see that the predicted rates in our
theory are pessimistic when compared to the observed rates and that
the suggested analysis of using the square summability or using the
improved rates, $\tilde{p}_s$, might yield sharper bounds for the
predicted work rates.  On the other hand, we know from our previous
work \cite[Theorem 1]{hajiali.eal:MISC1} that the work degrades with
increasing $\exmpld$ with a log factor and in fact the expected work
rate for maximum regularity when the number of random variables is
finite is of $\Order{W_{\max}^{-2} \log(W_{\max})^{\exmpld-1}}$. This
can be seen Figure~\ref{fig:theory_work_rates} and $\exmpld=3$, since
in this case the observed work rate seems to be converging to a value
less that $2$.

\begin{figure}
  \centering
  \includegraphics[page=1, scale=0.36]{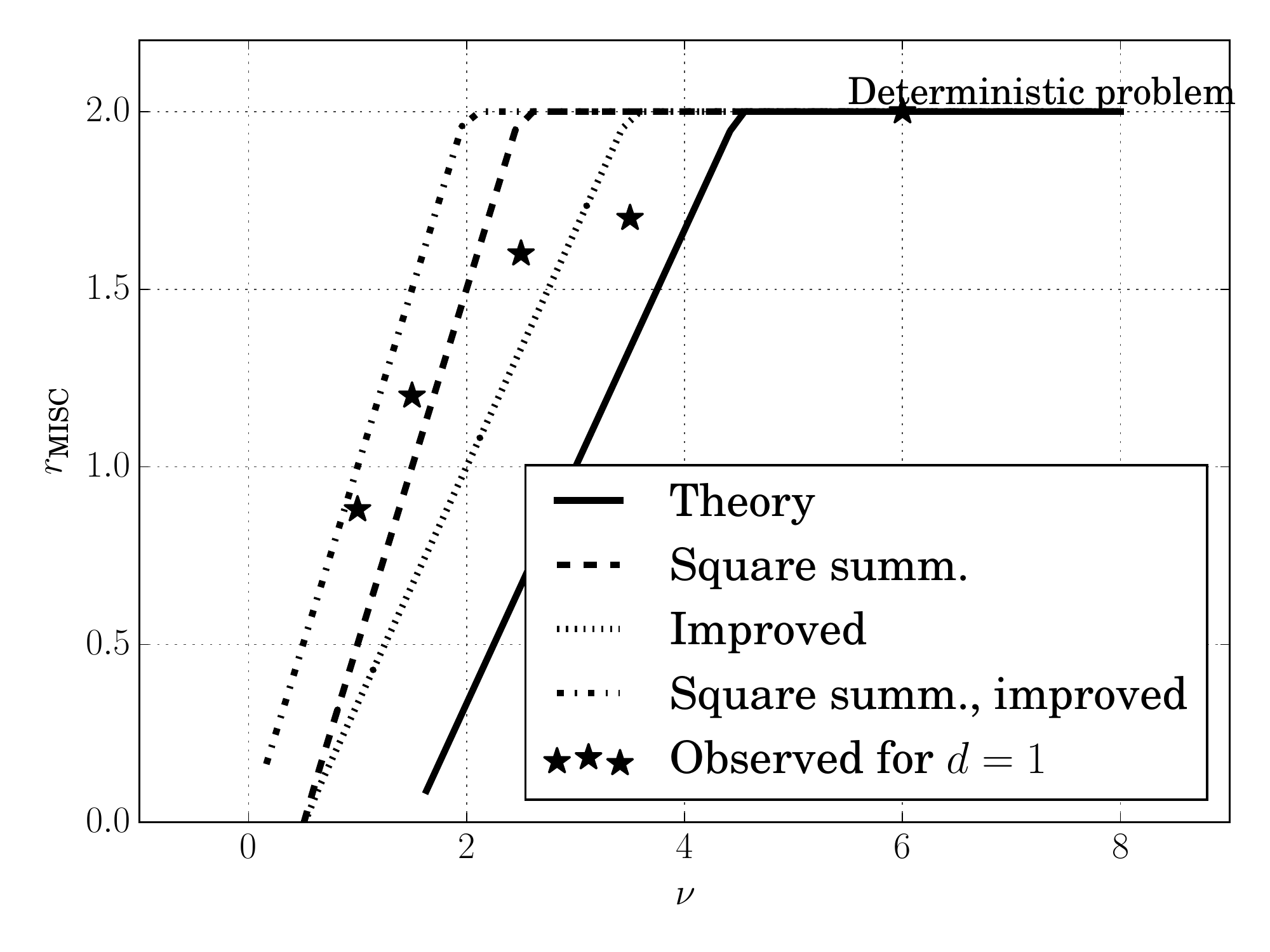}
  \includegraphics[page=2, scale=0.36]{rates.pdf}
  \caption{The upper bound of the MISC rate, $r_{\text{MISC}}$, as
    predicted in Theorem~\ref{theo:MISC_conv} versus the observed
    rates when running the example detailed in
    Section~\ref{s:numerics}.  Refer to
    Section~\ref{s:example_analysis} for an explanation of the
    different curves. Also included are the observed convergence rates
    for a few values of $\nu$ and the observed convergence rate of
    MISC with no random variable and constant diffusion coefficient,
    $a$, as a \emph{horizontal} line. The latter is referred to as
    the ``deterministic problem'' and shows more clearly the
        effect of the logarithmic factor in the work for
    $\exmpld>1$, as shown in Figure~\ref{fig:conv_results_det} and
    proved in \cite[Theorem 1]{hajiali.eal:MISC1}.}
  \label{fig:theory_work_rates}
\end{figure}

\section{Numerical experiments}\label{s:numerics}
We now verify the effectiveness of the MISC approximation on some
instances of the general elliptic equation \eqref{eq:PDE1}, as well as
the validity of the convergence analysis detailed in the previous
sections.  In particular, we consider the domain
$\mathcal B = [0,1]^d$ and the family of random diffusion coefficients
specified in Example~\ref{ex:logunif}.
In more detail, we consider a problem with one physical dimension ($\exmpld=1$)
and another with three dimensions ($\exmpld=3$);
in both cases, we set  $\forcing(\xx)=1$,
and model the diffusion coefficient by the expansion \eqref{eq:our_matern} with different
values of $\nu$.
Finally, the quantity of interest is a local average defined as
\begin{equation}\label{eq:qoi_def}
  \qoi(\yy) =  \frac{10}{(\sigma\sqrt{2\pi})^\exmpld} \int_\mathscr{B} u(\xx,\yy)\exp \left( -\frac{\|\xx-\xx_0\|_2^2}{2\sigma^2} \right) d\xx
\end{equation}
with $\sigma=0.2$ and location $\xx_0=0.3$ for $\exmpld=1$ and
$\xx_0= [0.3, 0.2,0.6]$ for $\exmpld=3$.
The deterministic problems are discretized with a second-order
centered finite differences scheme for which we expect to recover the
same rate in the numerical experiments that we would obtain with
piece-wise multi-linear finite elements on a structured mesh. We choose
the mesh-sizes in \eqref{eq:choice_of_h_and_FEM} and
$h_{0,i}=1/3$ for all $i=1,\ldots,d$, and the resulting linear system is solved with GMRES with
sufficient accuracy. With these values and using the coarsest
discretization, $h_0=1/3$, in all dimensions, the coefficient of
variation of the quantity of interest can be approximated to be
between $90\%$ and $110\%$ depending on the number of dimensions,
$\exmpld$, and the particular value of the parameter, $\nu$, that we
consider below. Finally, the quadrature points on the stochastic
domain are the already-mentioned Clenshaw-Curtis points (see
eq. \eqref{eq:CC_def} and \eqref{eq:m_CC}).

In the plots below, the computational work is compared in terms of the
total number of degrees of freedom to avoid discrepancies in running
time due to implementation details, i.e., using \eqref{eq:work_decomp}
and Assumption~\ref{assump:work}. Moreover, we set $\gamma=1$ in
\eqref{eq:dW_det_assump}, which is motivated by the fact that, for the
tolerances we are interested in, we estimate that the cost of solving
a linear system with GMRES is linear with respect to the number of
degrees of freedom.

In order to evaluate the MISC estimator, we need to build the index
set \eqref{eq:opt_set}.  To do that, we must be able to evaluate two
quantities for every $\valpha$ and $\vbeta$: the work contribution,
$\Delta W_{\valpha, \vbeta}$, and the error contribution,
$\Delta E_{\valpha, \vbeta}$. Evaluating the work contribution is
straightforward thanks to Assumption~\ref{assump:work} and using
$\gamma=1$. On the other hand, evaluating the error contribution is
more involved. We look at two options:
\begin{description}
\item[\textbf{``brute-force'' evaluation.}] We compute
  $\vec \Delta[\descqoi\valpha\vbeta]$ for all $(\valpha, \vbeta)$
  within some ``universe'' index set and set
  $\Delta E_{\valpha, \vbeta} = \left| \vec
    \Delta[\descqoi\valpha\vbeta] \right|$.
  Notice that this method is not practical since the cost of
  constructing the set, $\mathcal I$, would far dominate the cost of
  the MISC estimator. However, within some ``universe'' index
set,
  this method would produce the best possible convergence and serve as
  a benchmark for other MISC sets within that universe.
\item[\textbf{``a-priori'' evaluation.}] We use
  Lemma~\ref{lemma:error} to bound $\Delta E_{\valpha, \vbeta}$. Using
  these bounds instead of exact values produces quasi-optimal index
  sets
  (cf. \cite{back.nobile.eal:optimal,nobile.eal:optimal-sparse-grids}
  ). This method in turn requires the estimation of the parameters
  $r_{\text{FEM}}, \left\{g_{s, j}\right\}_{j \geq 1}$ for all
  $s=0,\ldots,s_{\max}$. Since we use a second-order centered finite
  differences scheme and consider the convergence of a quantity of
  interest, we expect $r_{\text{FEM}}=2 \min\left(1, \frac{\nu}{d}\right)$
  as motivated by the ``improved'' analysis in the
  previous section and considering the summability properties of
  $\left\{A_{\vec k}^2 \|D^{\vec s}\Upsilon_{\vec k, \vec\ell}\|_{L^\infty(\mathscr B)}^2 \right\}$.
  This can also be validated numerically in the usual way by fixing all random
  variables to their expected value and checking the decay of
  $\Delta E_{\valpha, \vec 1}$ with respect to $\valpha$.

  On the other hand, estimating $\left\{g_{s, j}\right\}_{j \geq 1}$
  for $s=0,\ldots,s_{\max}$ is more difficult since, in principle, we
  do not know a priori, for a given $\valpha$ and $\vbeta$, which
  value of $s \in \{0,1,\ldots,s_{\max}\}$ yields the smallest
  estimate of $\Delta E_{\valpha, \vbeta}$.  Instead, we use a
  ``simplified'' model that was used in~\cite{hajiali.eal:MISC1}:
\begin{equation}\label{eq:deltaE_simplified}
  \Delta E_{\valpha,\vbeta} \leq C e^{-\sum_{j\geq 1} m(\beta_j-1)\tilde g_{j}} 2^{-|\valpha|r_{\mathrm{FEM}}},
\end{equation}
where $\tilde g_j$ is some unknown function of $g_{s, j}$ for all
$s = 0,1,\ldots, s_{\max}$. $\tilde g_j$ can be estimated given
$r_{\text{FEM}}$ and a set of evaluations of
$\left| \vec \Delta[\descqoi\valpha\vbeta] \right|$ for some
$(\valpha, \vbeta) \in \mathcal I^*$ by solving a least-squares problem
to fit the linear model
\[ \sum_{j\geq 1} \tilde g_j m(\beta_j-1) = -\log\left(\left| \vec
      \Delta[\descqoi\valpha\vbeta] \right|\right) - |\valpha|
  r_{\text{FEM}},\qquad \text{for all } (\valpha, \vbeta) \in \mathcal
  I^*.\] For our example, these rates are plotted in
Figures~\ref{fig:deltaE_numerical_D1}(a)
and~\ref{fig:deltaE_numerical_D3}(a) for $\exmpld=1$ and $\exmpld=3$,
respectively. In our current implementation, the construction of the
optimal MISC set, $\mathcal I$, is separate from the set
$\mathcal I^*$.  However, it is possible in principle to construct an
algorithm in which the optimal MISC set, $\mathcal I$, is constructed
iteratively by alternating between estimating rates given a set of
indices and evaluating the MISC estimator.
\end{description}

Note that, in the current work, there are certain operations whose
costs we do not track or compare. The first operation is the
estimation of the stochastic rates,
$\left\{\tilde g_j\right\}_{j\geq 1}$. The second operation is the
construction of the optimal set given estimates of error and work
contribution. We believe that the cost of these two operations can be
reduced by using the previously mentioned iterative algorithm. The
cost of these operations is thus dominated by the cost of evaluating
the MISC estimator. The third operation is the assembly of the
stiffness matrix, especially since it scales linearly with the number
of random variables. While the cost of this operation is relevant to
our discussion, it is usually dominated by the cost of the linear
solver, at least for fine-enough discretizations.

Finally, we also compare MISC to the Multi-index Monte Carlo (MIMC)
method as detailed in \cite{abdullatif.etal:MultiIndexMC}, for which
$\Order{W_{\max}^{-0.5}}$ convergence can be proved for
Example~\ref{ex:logunif} with $\gamma=1, \exmpld \leq 3$ and
sufficiently large $\nu$ (see Appendix~\ref{sec:app-summability}).
Moreover, when computing errors, we use the result obtained using a
well-resolved MISC solution as a reference value.

Figures~\ref{fig:deltaE_numerical_D1}(b--d) and
Figures~\ref{fig:deltaE_numerical_D3}(b--d) compare some computed
values of $\left| \vec \Delta[\descqoi\valpha\vbeta] \right|$ versus
the model \eqref{eq:deltaE_simplified} using the estimated rates
$r_{\text{FEM}}=2 \min\left(1, \frac{\nu}{d}\right)$ and
$\left\{\tilde g_j\right\}_{j\geq 1}$. These plots show that the
model~\eqref{eq:deltaE_simplified} is a good fit for the case
$\exmpld=1, \nu=2.5$ and $\exmpld=3, \nu=4.5$. Moreover, similar plots
were produced for other values of $\exmpld$ and $\nu$ that are not
reported here but also show good fit. Figures~\ref{fig:extremes_D1}
and~\ref{fig:extremes_D3} show
\begin{itemize}
\item the maximum space discretization level,
$\max_{(\valpha, \vbeta) \in \mathcal I} \max(\valpha)$,
\item  the maximum
quadrature level,
$\max_{(\valpha, \vbeta) \in \mathcal I} \max(\vbeta)$,
\item the index of
the last activated random variable,
$\max_{(\valpha, \vbeta) \in \mathcal I} \max_{\beta_j > 1} j$,
\item and the maximum number of jointly activated variables,
  $\max_{(\valpha, \vbeta) \in \mathcal I} |\vbeta-\vec 1|_0$.
\end{itemize}
These values
convey the size of the used index set, $\mathcal I$, for different
values of $W_{\max}$.

Figure~\ref{fig:theory_work_rates} shows the observed convergence
rates of MISC vs MIMC for the cases $\exmpld=1$ and $\exmpld=3$ and
different values of $\nu$. This figure shows that the observed rates
are better than those predicted by the theory developed in this work,
which suggests that further improvement in the theory is possible (see
Remark~\ref{rem:keep_it_simple}). Figures~\ref{fig:conv_results_det}
and~\ref{fig:conv_results} show in greater details the observed
convergence curves for $\exmpld=1, \nu=2.5$ and $\exmpld=3, \nu=4.5$
and their respective linear fit in log-log scale.

We recall that, as shown in \cite[Theorem 1]{hajiali.eal:MISC1}, the
convergence rate of MISC with a finite number of random variables is
$\Order{W_{\max}^{-2} \log(W_{\max})^{\exmpld-1}}$. Compare this to the
theory presented here that predicts, as $\nu \to \infty$, a
convergence of $\Order{W_{\max}^{-2+\epsilon}}$ for any
$\epsilon>0$. However, Figure~\ref{fig:conv_results_det} shows that
even for a problem with $\exmpld=3$ and no random variables, MISC (which, in
this case, becomes equivalent to a deterministic combination technique
\cite{Bungartz.Griebel.Roschke.ea:pointwise.conv}) has an observed
convergence rate that is closer to $-1.38$. This is due to the
effect of the logarithmic term that is nonzero for
$\exmpld>1$. Based on this, we should not expect a better convergence rate
for $\exmpld=3$ and any finite $\nu > 0$. This is also numerically validated in
Figure~\ref{fig:conv_results}, which shows the full convergence curves
for $\exmpld=1, \nu=2.5$ and $\exmpld=3, \nu=4.5$.

\begin{figure}%
  \centering \subfigure[]
  {\includegraphics[page=1,width=0.49\textwidth]{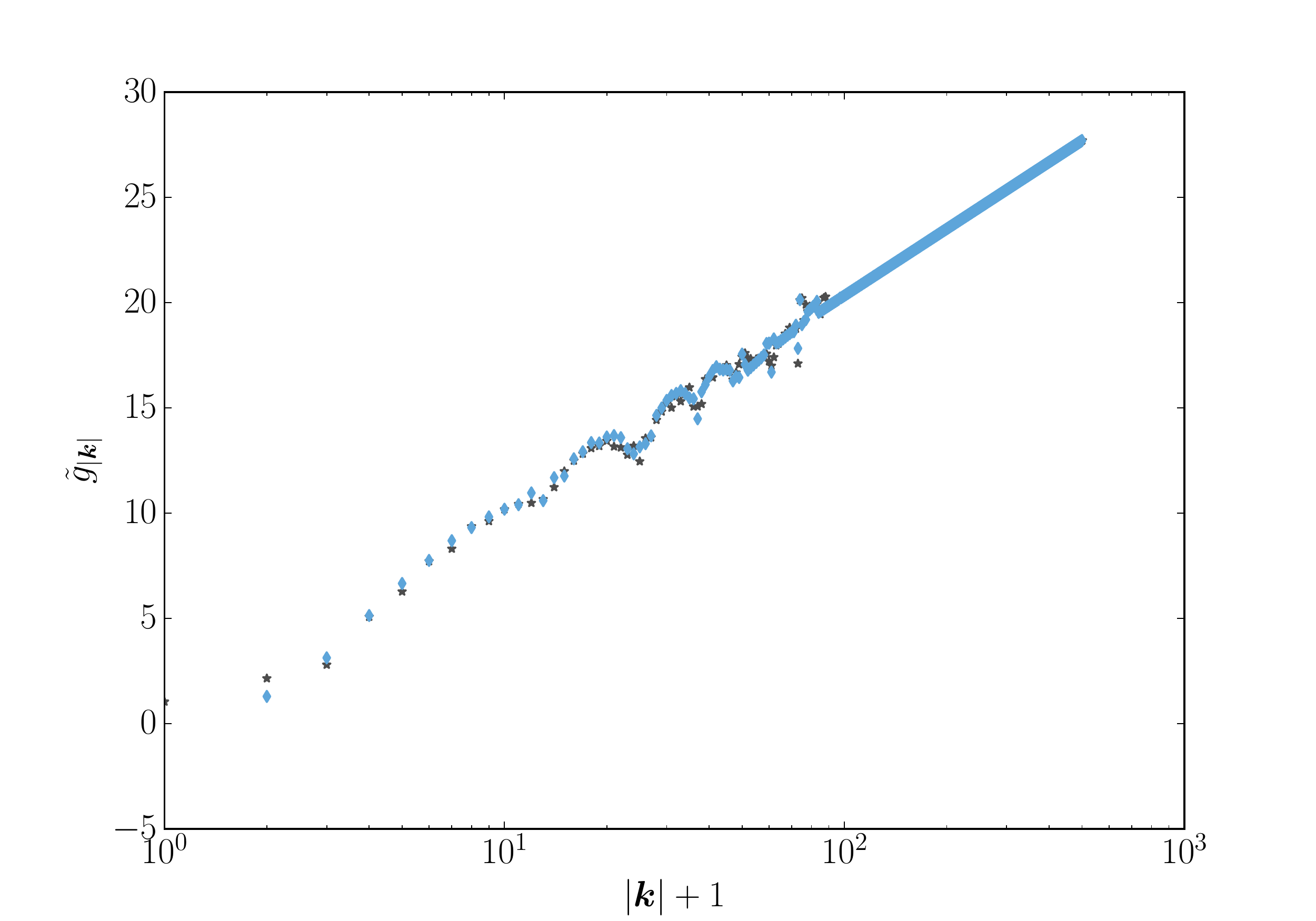}}
  \subfigure[]
  {\includegraphics[page=2,width=0.49\textwidth]{EXAMPLE_1_nu2_5_d1_rates.pdf}}
  \\
  \subfigure[]
  {\includegraphics[page=3,width=0.49\textwidth]{EXAMPLE_1_nu2_5_d1_rates.pdf}}
  \subfigure[]
  {\includegraphics[page=4,width=0.49\textwidth]{EXAMPLE_1_nu2_5_d1_rates.pdf}}
  \caption{Example~\ref{ex:logunif}, $\exmpld=1$ and
    $\nu=2.5$. {\emph{(a)}} A plot of the estimated stochastic rates,
    $\tilde g_{j}$, that are used
    in~\eqref{eq:deltaE_simplified}. Different markers correspond to
    different modes multiplying the same value of $A_{\vec k}$.
    {\emph{(b--d)}} {\emph{solid}} lines show the computed
    approximations of
    $\Delta E_{\vec 1, \vec 1 + j \tilde \vbeta}, \Delta E_{\vec 1 + j
      \tilde \valpha, \vec 1}$ and
    $\Delta E_{\vec 1 + j \tilde \valpha , \vec 1 + j \tilde \vbeta}$,
    respectively versus the model in \eqref{eq:deltaE_simplified}
    represented with
    {\emph{dashed}} lines.}
\label{fig:deltaE_numerical_D1}
\end{figure}

\begin{figure}%
  \centering
  \subfigure{\includegraphics[page=2,width=0.49\textwidth]{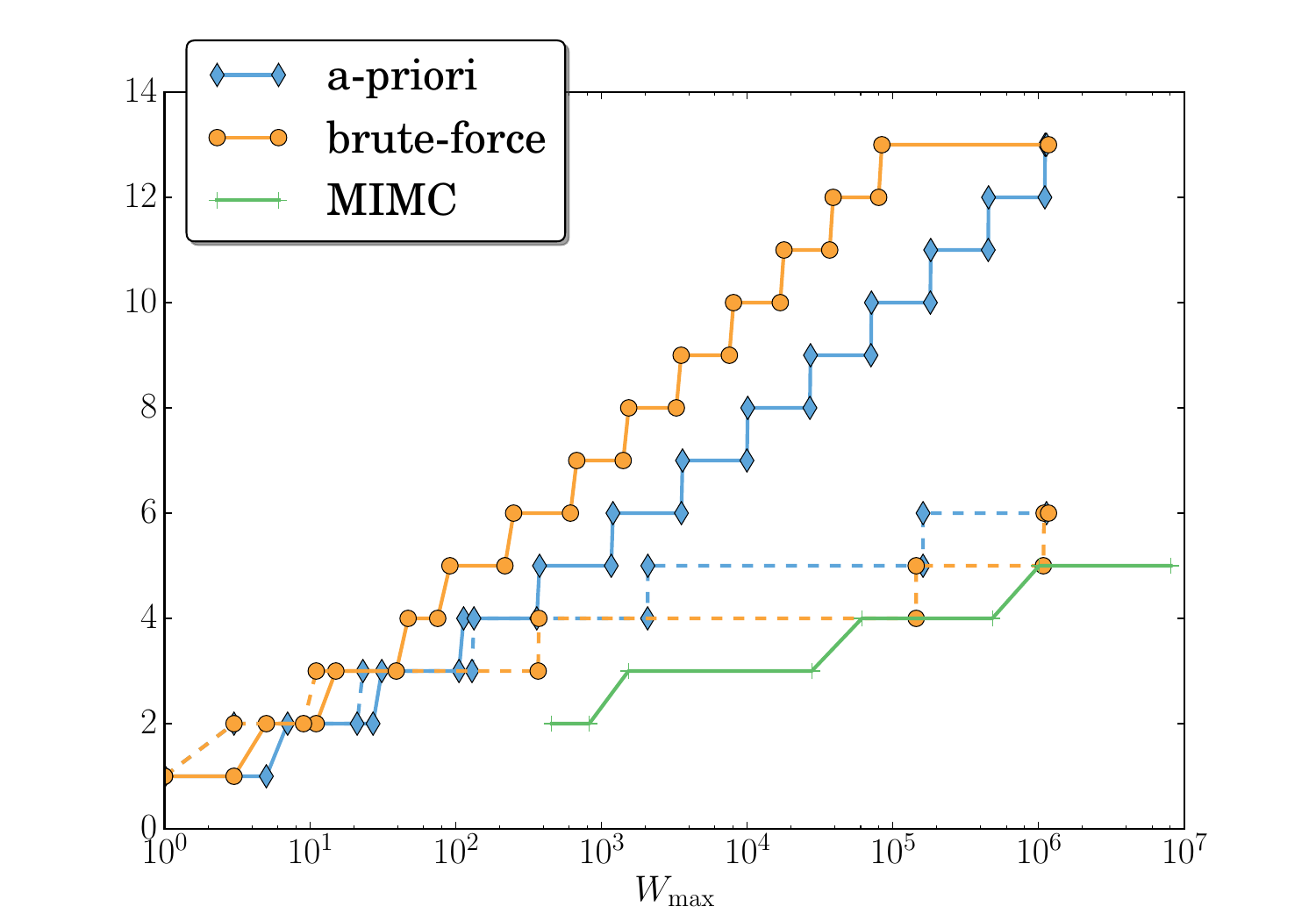}}
  \subfigure{\includegraphics[page=3,width=0.49\textwidth]{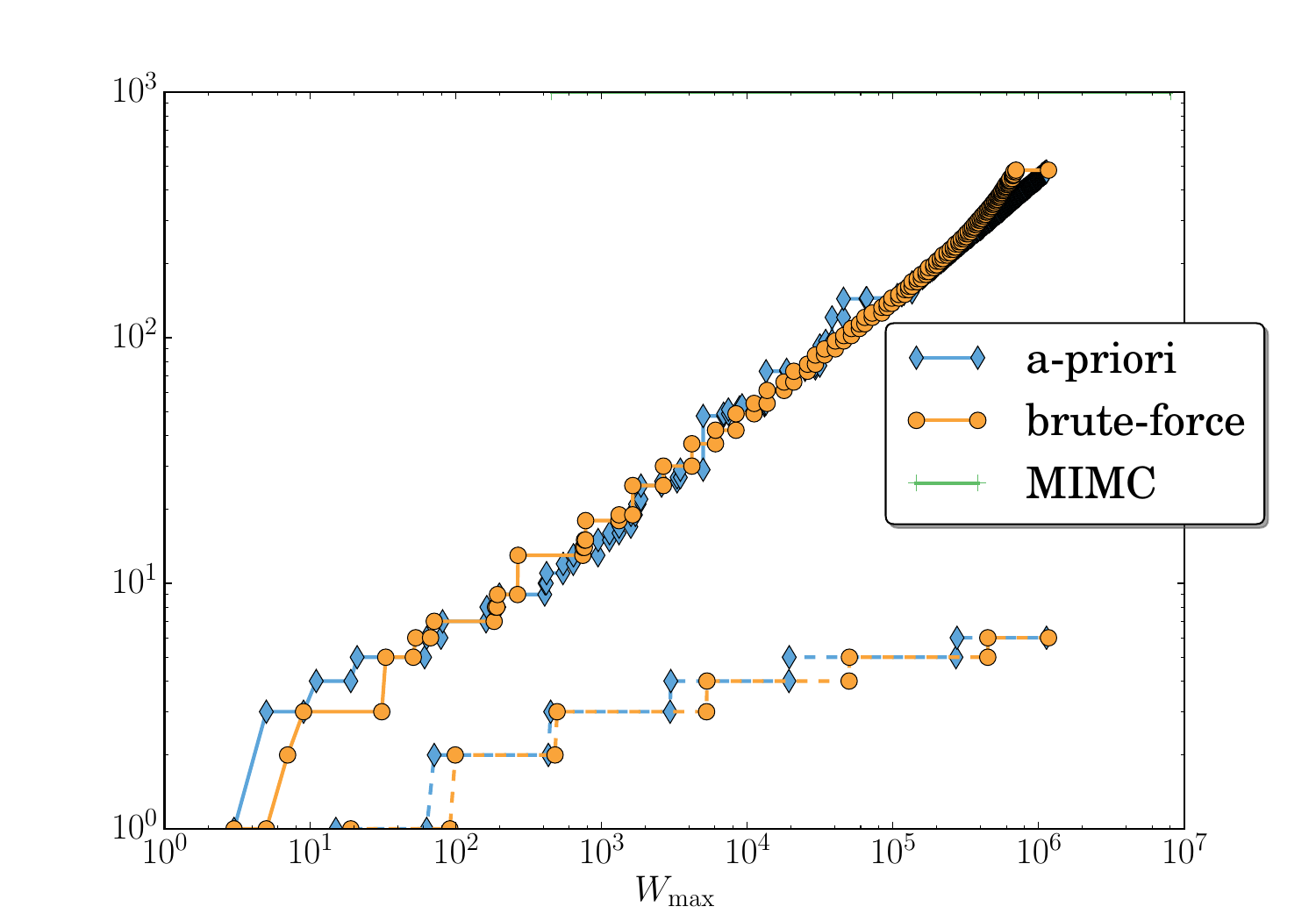}}
\caption{Example~\ref{ex:logunif}, $\exmpld=1$ and $\nu=2.5$. This figure
  shows extreme values of $\valpha$ and $\vbeta$ included in the MISC
  set, $\mathcal{I}$. Specifically,
  the \emph{left-solid} lines is the maximum
  space discretization level,
  $\max_{(\valpha, \vbeta) \in \mathcal I} \left(\max\left(\valpha\right)\right)$,
  the \emph{left-dashed} lines are the maximum quadrature level,
  $\max_{(\valpha, \vbeta) \in \mathcal I} \left(\max\left(\vbeta\right)\right)$,
  the \emph{right-solid} lines are the index of the last activated random
  variable,
  $\max_{(\valpha, \vbeta) \in \mathcal I} \left(\max_{\beta_j > 1} j\right)$, and
  the \emph{right-dashed} lines are the maximum number of jointly activated
  variables, $\max_{(\valpha, \vbeta) \in \mathcal I}
  \left(|\vbeta-\vec 1|_0\right)$.
}
  \label{fig:extremes_D1}
\end{figure}

\begin{figure}%
  \centering
  \subfigure[]
  {\includegraphics[page=1,width=0.49\textwidth]{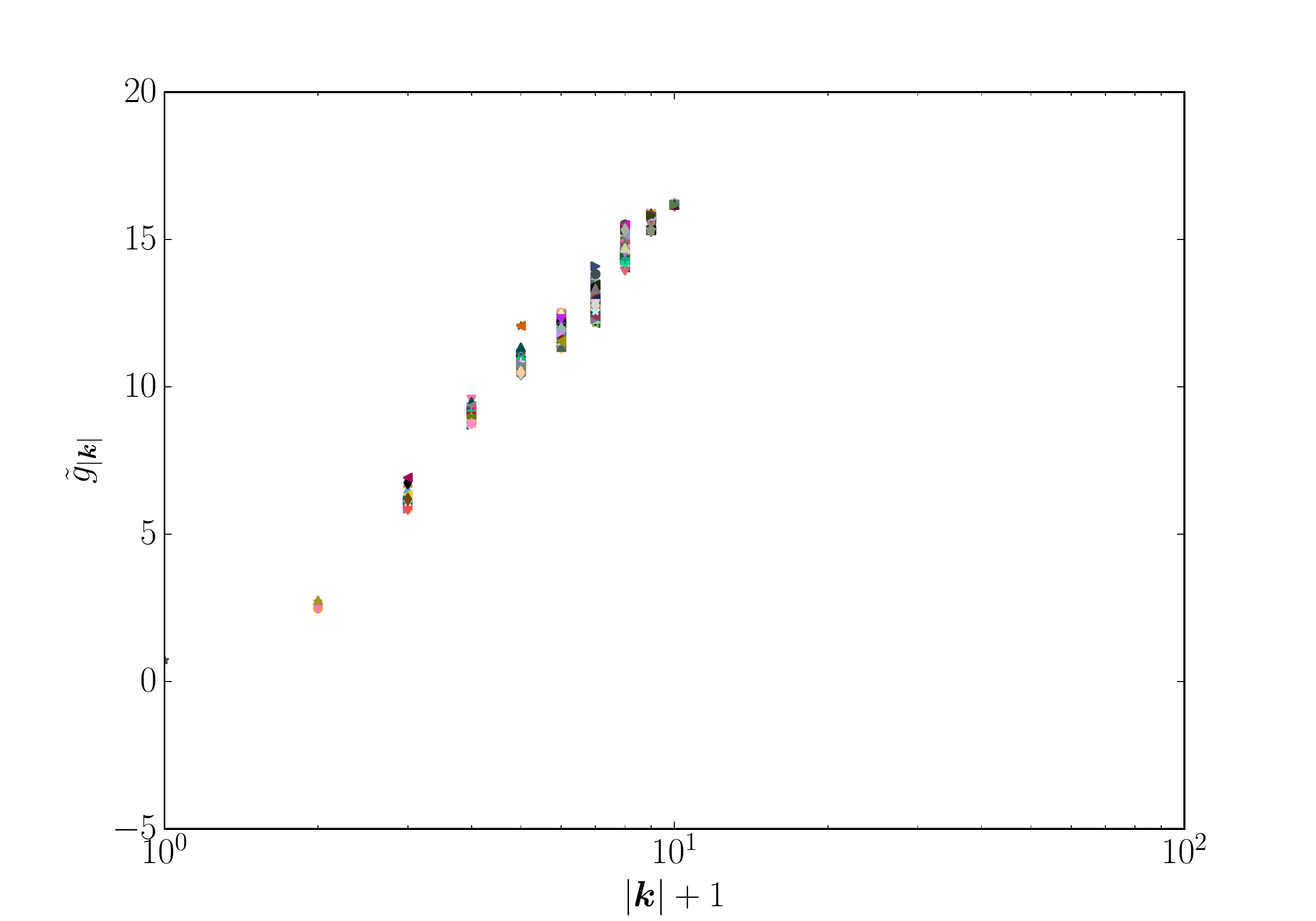}}
  \subfigure[]
  {\includegraphics[page=3,width=0.49\textwidth]{EXAMPLE_1_nu4_5_d3_rates.pdf}}
  \\
  \subfigure[]
  {\includegraphics[page=2,width=0.49\textwidth]{EXAMPLE_1_nu4_5_d3_rates.pdf}}
  \subfigure[]
  {\includegraphics[page=4,width=0.49\textwidth]{EXAMPLE_1_nu4_5_d3_rates.pdf}}
  \caption{Example~\ref{ex:logunif}, $\exmpld=3$ and $\nu=4.5$.
    {\emph{(a)}} A plot of the estimated stochastic rates,
    $\tilde g_{j}$, that are used
    in~\eqref{eq:deltaE_simplified}. Here different markers correspond
    to different modes multiplying the same value of $A_{\vec
      k}$. {\emph{(b--d)}} {\emph{solid}} lines show the computed
    approximations of
    $\Delta E_{\vec 1, \vec 1 + j \tilde \vbeta}, \Delta E_{\vec 1 + j
      \tilde \valpha, \vec 1}$ and
    $\Delta E_{\vec 1 + j \tilde \valpha , \vec 1 + j \tilde \vbeta}$,
    respectively versus the model in \eqref{eq:deltaE_simplified}
    represented with {\emph{dashed}} lines.}
\label{fig:deltaE_numerical_D3}
\end{figure}

\begin{figure}%
  \centering
  \subfigure{\includegraphics[page=2,width=0.49\textwidth]{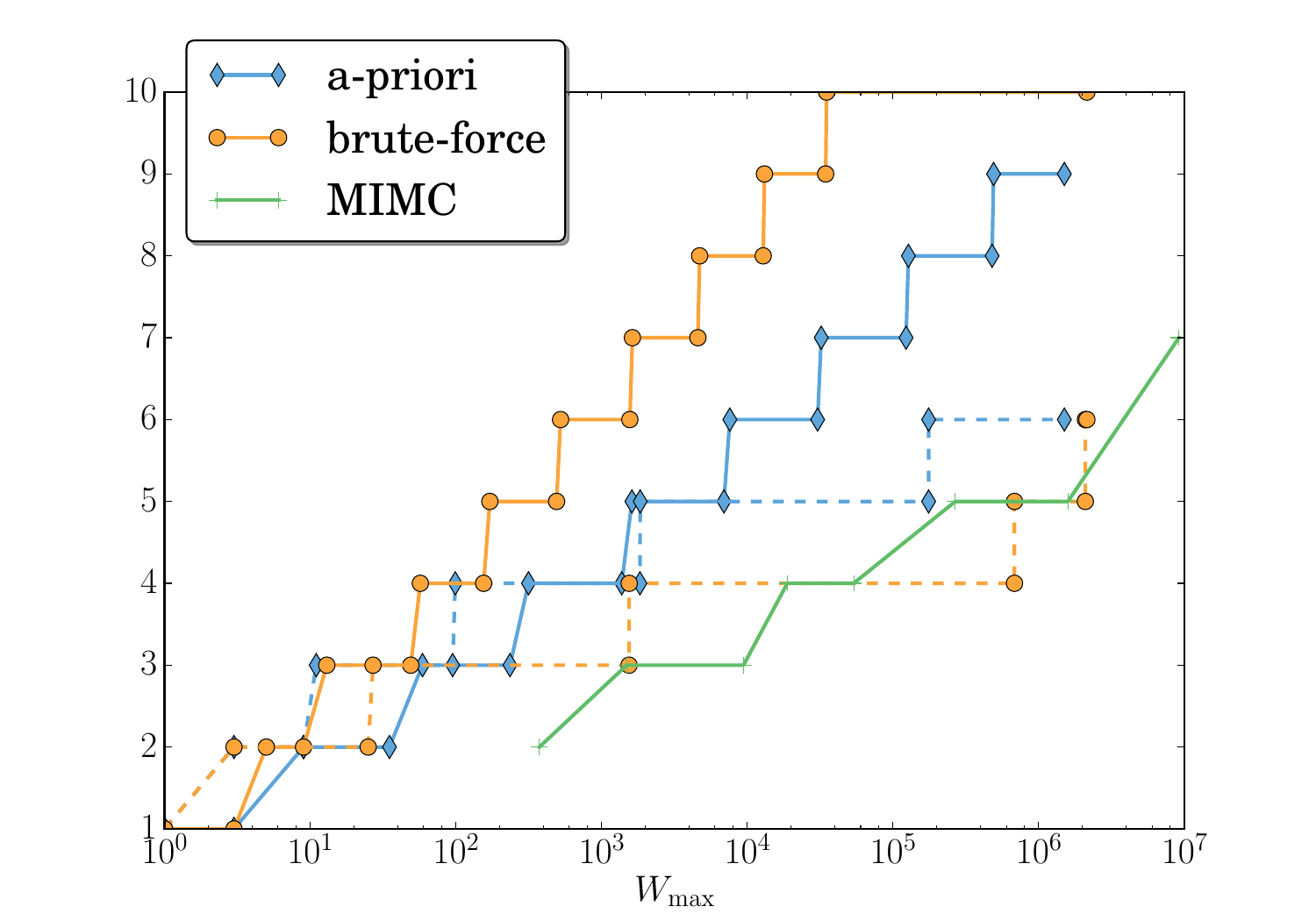}}
  \subfigure{\includegraphics[page=3,width=0.49\textwidth]{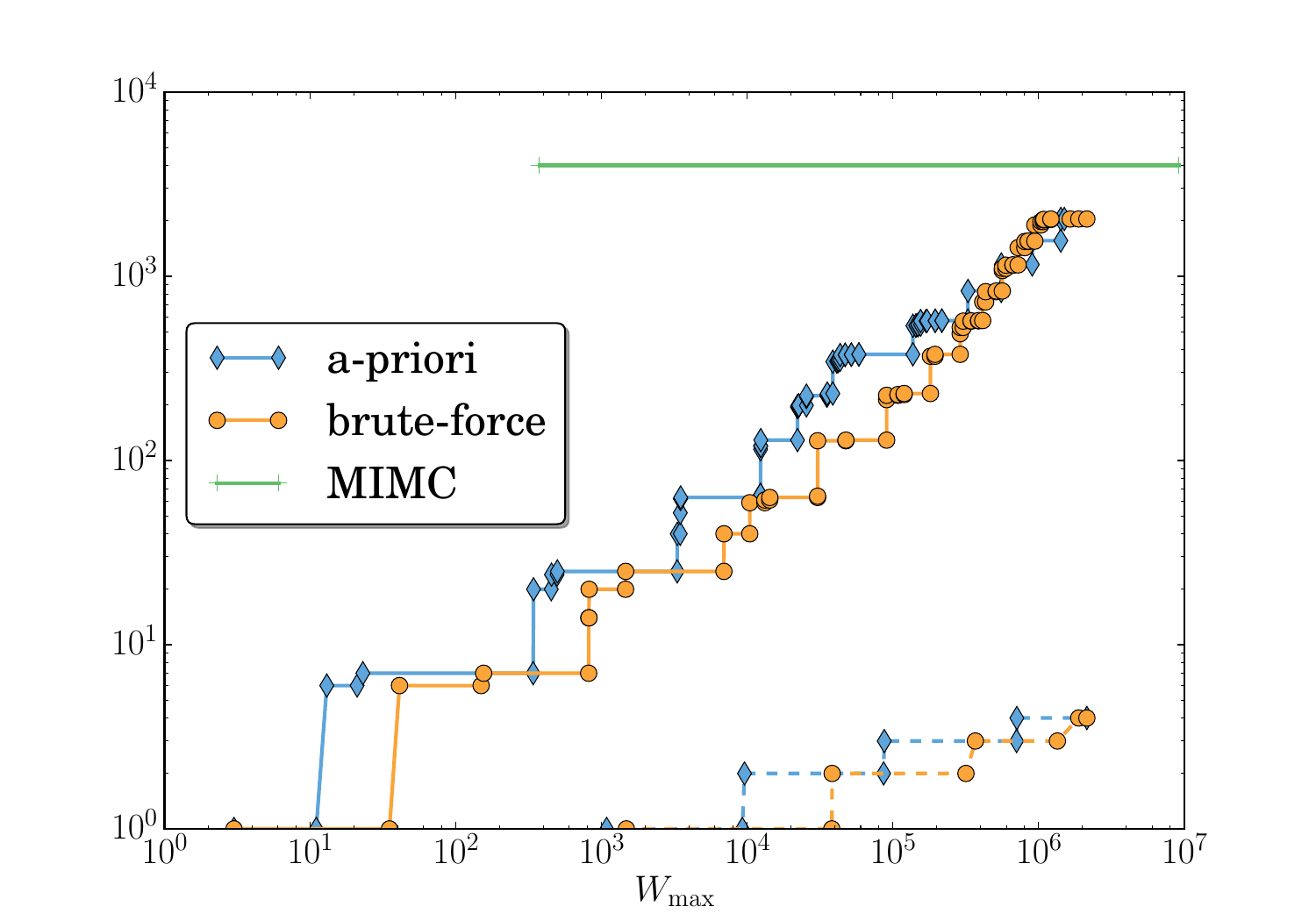}}
\caption{Example~\ref{ex:logunif}, $\exmpld=3$ and $\nu=4.5$. This
  figure shows extreme values of $\valpha$ and $\vbeta$ included in
  the MISC set $\mathcal{I}$. Specifically,
the \emph{left-solid} lines are the maximum
  space discretization level,
  $\max_{(\valpha, \vbeta) \in \mathcal I} \left(\max\left(\valpha\right)\right)$,
  the \emph{left-dashed} are the maximum quadrature level,
  $\max_{(\valpha, \vbeta) \in \mathcal I} \left(\max\left(\vbeta\right)\right)$,
  the \emph{right-solid} are the index of the last activated random
  variable,
  $\max_{(\valpha, \vbeta) \in \mathcal I} \left(\max_{\beta_j > 1} j\right)$, and
  the \emph{right-dashed} are the maximum number of jointly activated
  variables, $\max_{(\valpha, \vbeta) \in \mathcal I}
  \left(|\vbeta-\vec 1|_0\right)$.
}
\label{fig:extremes_D3}
\end{figure}

\begin{figure}%
  \centering
  \subfigure[$\exmpld=1$]{\includegraphics[page=1,width=0.49\textwidth]{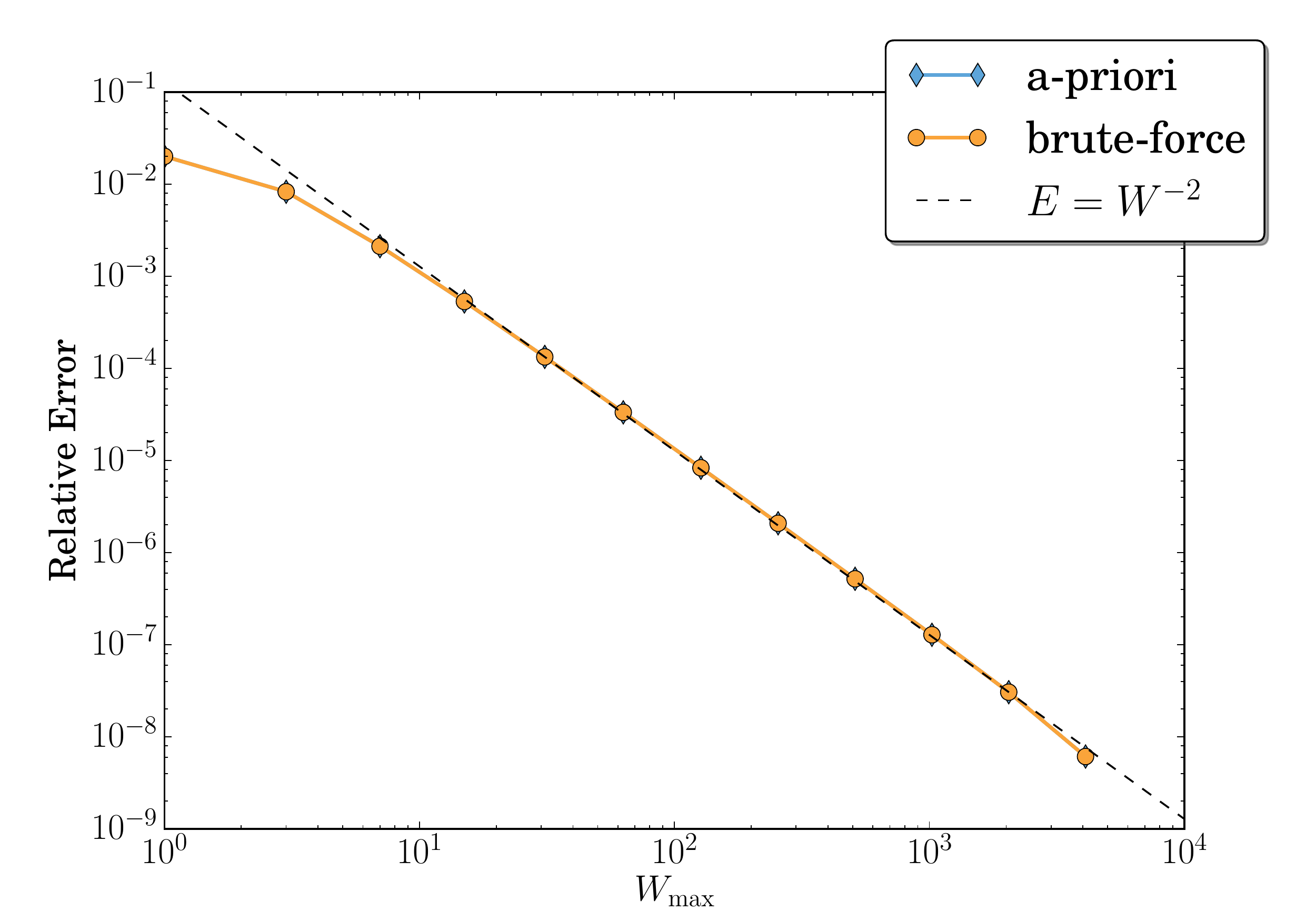}}
  \subfigure[$\exmpld=3$]{
    \includegraphics[page=1,width=0.49\textwidth]{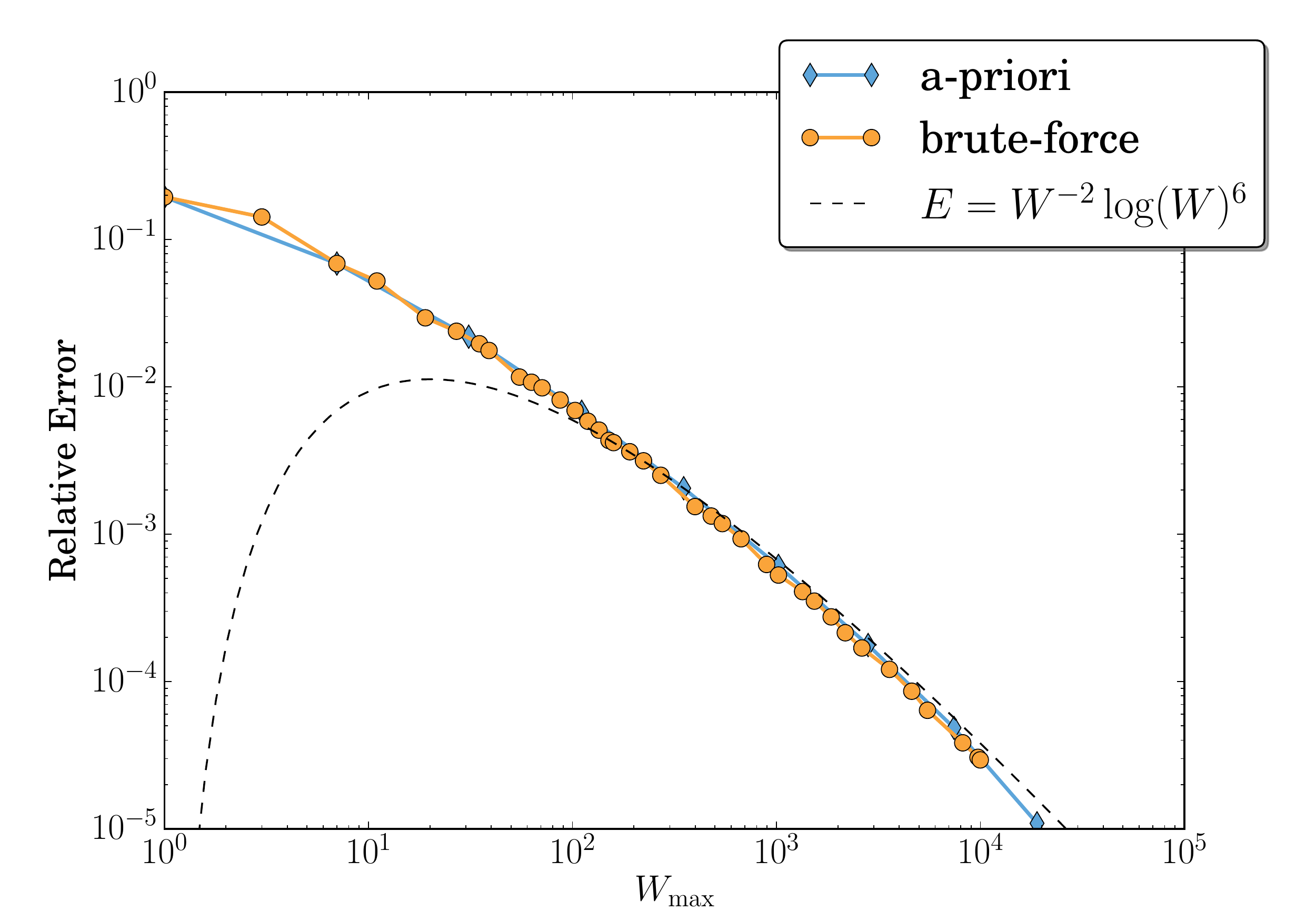}}
  \caption{Convergence results of MISC Example~\ref{ex:logunif} with a
    constant diffusion coefficient, $a$. In this case, MISC is
    equivalent to a deterministic combination technique
    \cite{Bungartz.Griebel.Roschke.ea:pointwise.conv}. These plots
    shows the non-asymptotic effect of the logarithmic factor for
    $\exmpld>1$ (as discussed in \cite[Theorem 1]{hajiali.eal:MISC1}) on the
    linear convergence fit in log-log scale.}
  \label{fig:conv_results_det}
\end{figure}

\begin{figure}%
  \centering
  \subfigure[$\exmpld=1, \nu=2.5$]{\includegraphics[page=1,width=0.49\textwidth]{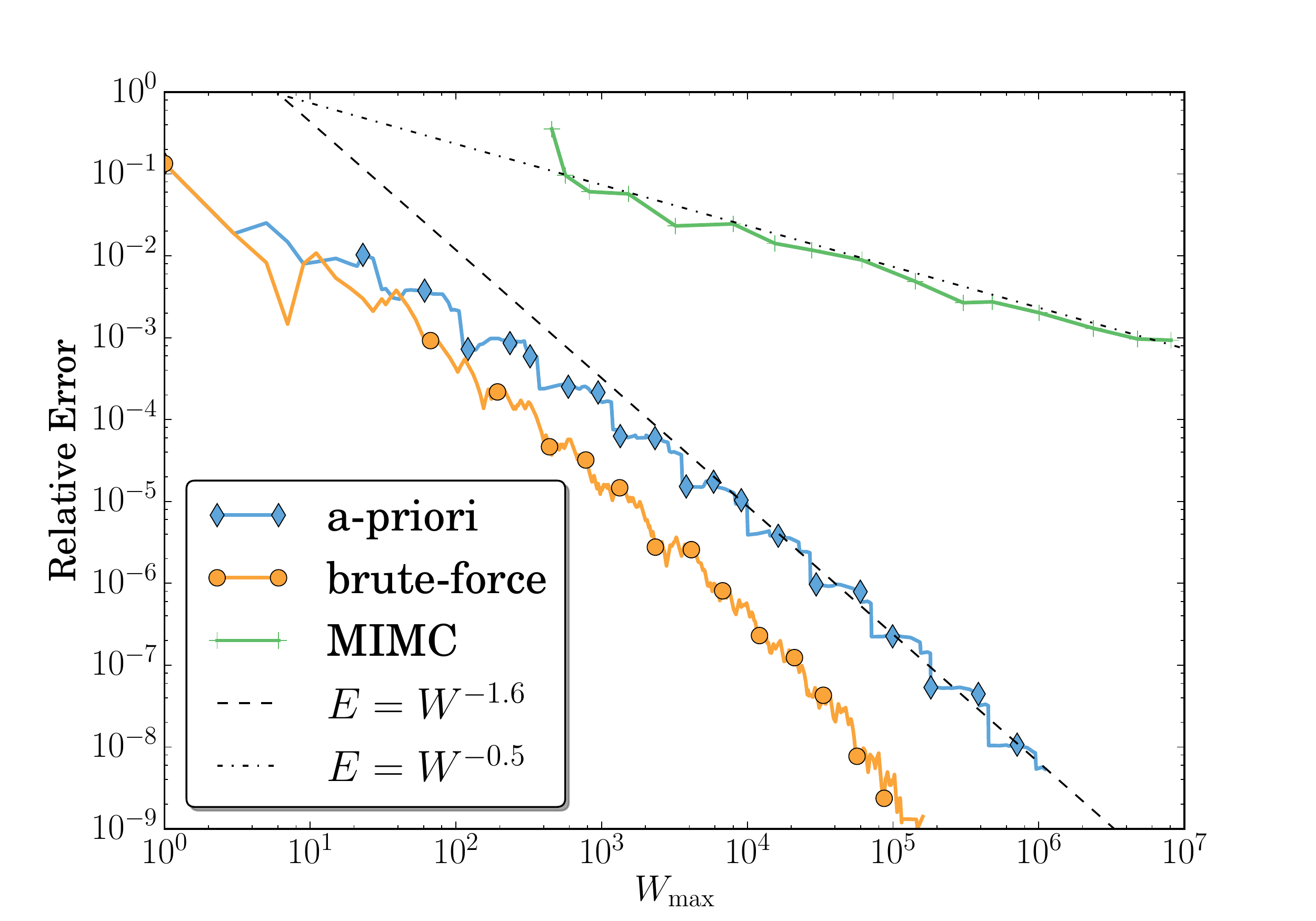}}
  \subfigure[$\exmpld=3, \nu=4.5$]{
    \includegraphics[page=1,width=0.49\textwidth]{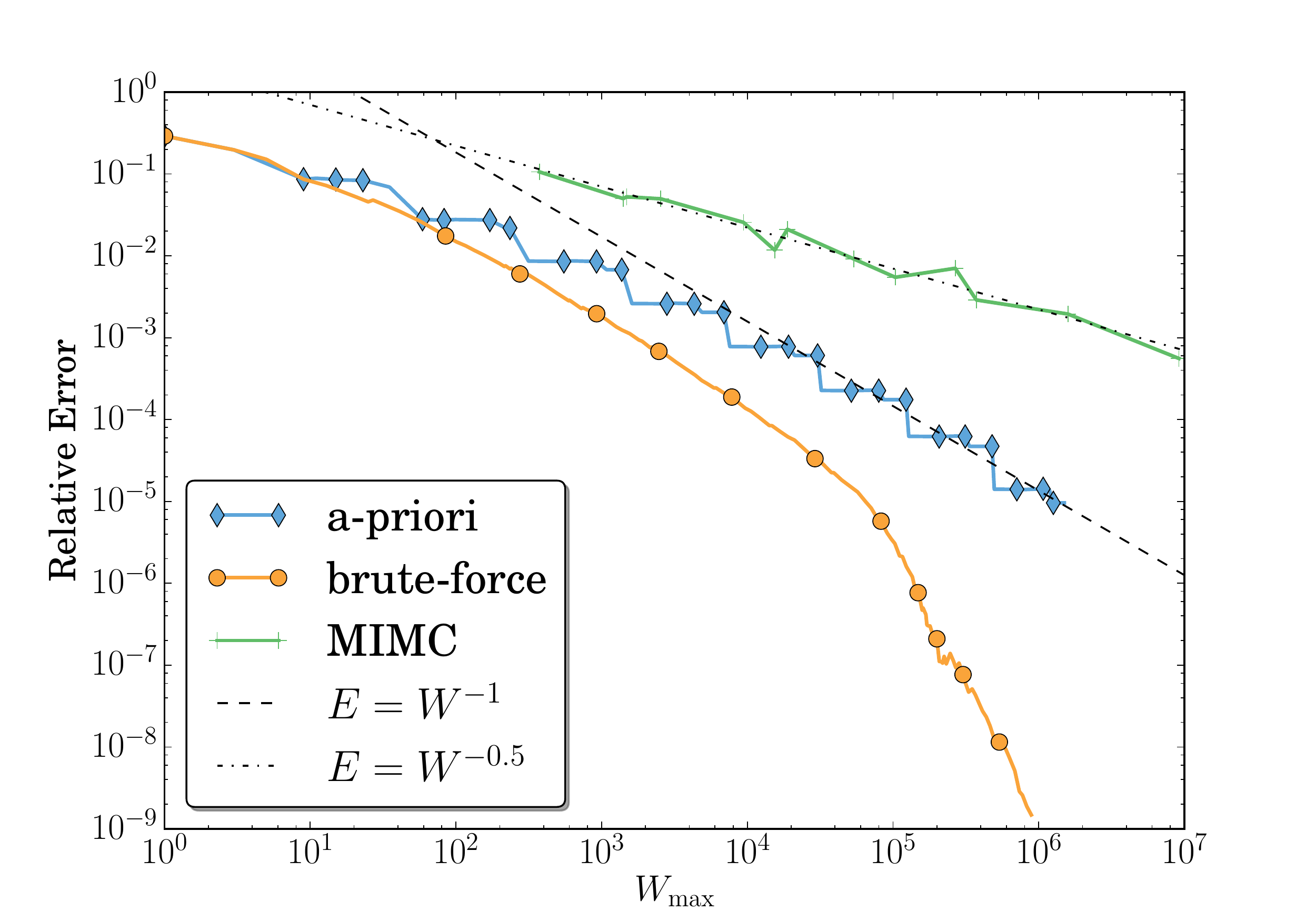}}
  \caption{Convergence results of MISC vs MIMC when applied to
    Example~\ref{ex:logunif}.}
  \label{fig:conv_results}
\end{figure}

\section{Conclusions}\label{s:conclusions}
In this work, we analyzed the performance of the MISC method when
applied to linear elliptic PDEs depending on a countable sequence of
random variables. For ease of presentation, we worked on tensor
product domains, but the results can be extended to more general
domains and non-uniform meshes, as briefly mentioned
Section~\ref{s:MISC-method}. We proved a convergence result using a
summability argument that shows that, in certain cases, the
convergence of the method is essentially dictated by the convergence
properties of the deterministic solver.  We then applied the
convergence theorem to derive convergence rates for the approximation
of the expected value of a functional of the solution of an elliptic
PDE with diffusion coefficient described by a random field,
tracking the dependence of the convergence rate on the spatial
regularity of the realizations of the random field. The theoretical
findings are backed up by numerical experiments that show the
dependence of the convergence rate on the regularity parameter.
Future works includes extending the convergence analysis to
higher-order finite element solvers and improving the estimates of the
error contribution of each difference operator by taking into account
the factorial terms appearing in the estimates for the size of the
Chebyshev coefficients, cf.
\cite{back.nobile.eal:optimal,cohen.devore.schwab:nterm2}. Moreover,
the ideas in \cite{abdullatif:continuationMLMC} can be extended to
design an algorithm that iteratively estimates the optimal MISC set by
alternating between optimizing the set and evaluating the estimator to
ensure that the work to optimize the set is dominated by the work to
evaluate the MISC estimator.

\paragraph{Acknowledgement}
  F. Nobile and L. Tamellini received support from the
  Center for ADvanced MOdeling Science (CADMOS) and partial
  support by the Swiss National Science Foundation under the
  Project No. 140574 ``Efficient numerical methods for flow and
  transport phenomena in heterogeneous random porous media''.
  L. Tamellini also received support from the
    Gruppo Nazionale Calcolo Scientifico - Istituto Nazionale di Alta Matematica
    ``Francesco Severi'' (GNCS-INDAM).
  R. Tempone is a member of the KAUST Strategic Research Initiative,
  Center for Uncertainty Quantification in Computational Sciences and
  Engineering.
  R. Tempone received support from the KAUST CRG3 Award Ref: 2281.

\appendix
\section{Summability of series expansion}\label{sec:app-summability}

We start by recalling a useful multivariate Fa\`a di Bruno formula taken from \cite[Theorem 2.1]{constantine.savits:multivariate}.

\begin{lemma} Let $\mathscr{B}\subset \rset^d$ be an open domain, $g:\mathscr{B}\rightarrow \rset$  and $f:\rset\rightarrow\rset$ be functions of class $C^s$ and denote $h=f\circ g : \mathscr{B}\rightarrow \rset$. For any multi-index $\ii \in \nset^d$, $|\ii| \le s$,  and any $\xx\in\mathscr{B}$,
\begin{equation}\label{eq:faadibruno}
  D^\ii h(\xx) = \ii!\sum_{\lambda=1}^{|\ii|} f^{(\lambda)}(g(\xx)) \sum_{r=1}^{\lambda} \sum_{p_r(\ii,\lambda)}  \prod_{j=1}^r \frac{(D^{\eell_j}g(\xx))^{k_j}}{k_j! (\eell_j!)^{k_j}},
\end{equation}
holds, where
\[
  p_r(\ii,\lambda) = \{ (k_j,\eell_j)\in \nset\times \nset^d_0, \; j=1,\ldots,r: \;\; \boldsymbol{0}\prec \eell_1\prec \eell_2 \prec \cdots\prec \eell_r, \; \sum_{j=1}^r k_j = \lambda, \;\; \sum_{j=1}^r k_j\eell_j = \ii\}
\]
and $\prec$ denotes the lexicographic ordering of multi-indices. The
set $p_r(\ii,\lambda)$ denotes the set of possible decompositions of
$\ii$ as a sum of $\lambda$ multi-indices with $r\le \lambda$ distinct
multi-indices, $\eell_j$, taken with multiplicity $k_j$ such that
$\sum_{j=1}^r k_j=\lambda$.
\end{lemma}

Also from \cite[Corollary 2.9]{constantine.savits:multivariate}, we
have that, for any $\ii\in \nset^d$,
\[
  \ii! \sum_{r=1}^{\lambda} \sum_{p_r(\ii,\lambda)}  \prod_{j=1}^r \frac{1}{k_j! (\eell_j!)^{k_j}} = S_{|\ii|,\lambda},
\]
where $S_{n,k}$ is the \emph{Stirling number of the second kind},
which counts the number of ways to partition a set of $n$ objects into
$k$ non-empty subsets. Similarly, the \emph{Bell number},
$B_n=\sum_{k=0}^{n} S_{n,k}$, counts the number of partitions of a set
of $n$ objects, whereas the \emph{ordered Bell numbers} are defined by
$\tilde{B}_n = \sum_{k=0}^{n} k! S_{n,k}$ and satisfy the recursive
relation $\tilde{B}_n = \sum_{k=0}^{n-1} \binom{n}{k}
\tilde{B}_k$. Clearly, $B_n\le\tilde{B}_n$. Moreover, the bound
\begin{equation}\label{eq:bell_bound}
B_n\le \tilde{B}_n \le n!/(\log 2)^n
\end{equation}
was given in \cite[Lemma A.3]{back.nobile.eal:optimal}.
We now use these results to show the following result
\begin{lemma}
  Let $\mathscr B \subset \rset^d$ be an open-bounded domain and
  $\kappa\in C^s(\overline {\mathscr{B}})$ (real or complex valued) for $s \geq 0$. Then,
  $a = e^\kappa \in C^s(\overline {\mathscr{B}})$ and we have the estimate
\[
  \|a\|_{C^s(\overline{\mathscr{B}})} \le \frac{s!}{(\log
    2)^s}\|a\|_{C^0(\overline{\mathscr{B}})}(1+\|\kappa\|_{C^s(\overline{\mathscr{B}})})^s.
\]
\end{lemma}
\begin{proof}
Using formula \eqref{eq:faadibruno} we have for any $\ii\in\nset^d$, $|\ii|\le s$ and any $\xx\in\mathscr{B}$
\begin{align*}
  |D^\ii e^{\kappa(\xx)}| &= \ii!\sum_{\lambda=1}^{|\ii|} e^{\kappa(\xx)} \sum_{r=1}^{\lambda} \sum_{p_r(\ii,\lambda)}  \prod_{j=1}^r \frac{|D^{\eell_j}\kappa(\xx)|^{k_j}}{k_j! (\eell_j!)^{k_j}} \le \|a\|_{C^0(\overline{\mathscr{B}})} \sum_{\lambda=1}^{|\ii|} \|\kappa\|_{C^s(\overline{\mathscr{B}})}^\lambda S_{|\ii|,\lambda} \\
&\le \|a\|_{C^0(\overline{\mathscr{B}})} (1+\|\kappa\|_{C^s(\overline{\mathscr{B}})})^{|\ii|} B_n.
\end{align*}
The result then follows from the bound on the Bell numbers in \eqref{eq:bell_bound}.
\end{proof}

\subsection{ $L^p(\Gamma)$ summability, pointwise in space}

We now consider a diffusion coefficient as in Assumption~\ref{assump:b_and_p}:
\[
  a(\xx,\yy) = \exp\left\{\sum_{j\geq 1} \psi_j(\xx)y_j\right\} = \prod_{j=1}^\infty e^{y_j\psi_j(\xx)}, \qquad \xx\in\mathscr{B},
\]
with $y_j$, $j \geq 1$, independent random variables, all uniformly
distributed in $[-1,1]$ and recall the definition of the sequence
$\bb_{s}=\{b_{s,j}\}_{j\geq 1}$, for all $s\in\nset$ in
\eqref{eq:bbarj_def}.

\begin{lemma}\label{lemma:C0Lp}
If $\bb_0\in\ell^2$
then $\E{a(\xx)^p}<\infty$ for all $0<p<\infty$ and $\forall\xx\in\mathscr{B}$.
\end{lemma}
\begin{proof}
  For any $\xx\in\mathscr{B}$, we estimate the $p$-th moment of
  $a(\xx,\yy)$, exploiting independence of the random variables,
  $y_j$:
\begin{align*}
\E{a(\xx)^p} = \E{\prod_{j=1}^\infty e^{py_j\psi_j(\xx)}} = \prod_{j=1}^\infty \E{e^{py_j\psi_j(\xx)}} = \prod_{j=1}^\infty\frac{\sinh(p\psi_j(\xx))}{p\psi_j(\xx)} = \exp\left\{\sum_{j=1}^\infty\log\left(\frac{\sinh(p\psi_j(\xx))}{p\psi_j(\xx)}\right)\right\}
\end{align*}
where in the last two equalities we have implicitly assumed that $\sinh(z)/z=1$ for $z=0$. Setting $\theta_0(p;\xx) = \prod_{j=1}^\infty\frac{\sinh(p\psi_j(\xx))}{p\psi_j(\xx)}$ and observing that $\log(\sinh(z)/z) \sim z^2/6$, we have
\[
  \E{a(\xx)^p} = \theta_0(p;\xx) <\infty \quad \forall \xx\in\mathscr{B}, \;\; 0<p<\infty \qquad \Longleftrightarrow  \qquad \sum_{j=1}^\infty \psi_j(\xx)^2 <\infty.
\]
Since $\sum_{j=1}^\infty b_{0,j}^2 <\infty$ implies $\sum_{j=1}^\infty \psi_j(\xx)^2 <\infty$ for any $\xx\in\mathscr{B}$, this concludes the proof.
\end{proof}

A similar result holds for higher-order derivatives of $a$.
\begin{lemma}\label{lemma:CsLp}
  Let $s\in\nset_+$. If $\bb_s\in\ell^2$,
then for any $\ii\in\nset^d$, $|\ii|=s$, $\E{(D^\ii a(\xx))^{2p}}<\infty$ for all $0<p<\infty$ and $\forall\xx\in\mathscr{B}$.
\end{lemma}
\begin{proof}
Since the calculations are tedious, we prove the result here for $s=1$
only. Using the chain rule, we have
\begin{align*}
  (\partial_{x_i}a(\xx,\yy))^{2p} &= \left(\sum_{j \geq 1} a(\xx,\yy)\partial_{x_i}\psi_j(\xx)y_j\right)^{2p} = a(\xx,\yy)^{2p}\sum_{\qq\in\nset^{\nset} \atop |\qq|=2p} (2p)! \prod_{j=1}^\infty \frac{1}{q_j!}(\partial_{x_i}\psi_j(\xx)y_j)^{q_j}\\
                                  &=\sum_{\qq\in\nset^{\nset} \atop |\qq|=2p} (2p)! \prod_{j=1}^\infty \frac{1}{q_j!}(\partial_{x_i}\psi_j(\xx)y_j)^{q_j}e^{2py_j\psi_j(\xx)}.
\end{align*}
Hence, %
\[
  \E{(\partial_{x_i}a(\xx,\yy))^{2p}} = \sum_{\qq\in\nset^{\nset} \atop |\qq|=2p} (2p)! \prod_{j=1}^\infty (\partial_{x_i}\psi_j(\xx))^{q_j}\E{\frac{1}{q_j!}y_j^{q_j}e^{2py_j\psi_j(\xx)}}.
\]
We now distinguish between even or odd $q_j$. For even $q_j$, we have
\[
  \E{\frac{1}{q_j!}y_j^{q_j}e^{2py_j\psi_j(\xx)}} \le \E{\frac{1}{q_j!}e^{2py_j\psi_j(\xx)}} = \frac{1}{q_j!} \frac{\sinh(2p\psi_j(\xx))}{2p\psi_j(\xx)}
\]
while for $q_j$ odd we have
\begin{align*}
  \E{\frac{1}{q_j!}y_j^{q_j}e^{2py_j\psi_j(\xx)}} &= \frac{1}{q_j!}\int_{-1}^1 \frac{1}{2}y^{q_j}e^{2py\psi_j(\xx)}dy = \frac{1}{q_j!}\int_{0}^1 y^{q_j}\sinh(2py\psi_j(\xx))dy \\
&=\frac{1}{q_j!}\sum_{n=0}^\infty \frac{(2p\psi_j(\xx))^{2n+1}}{(2n+1)!}\int_0^1y^{2n+1+q_j}dy = \frac{1}{q_j!}\sum_{n=0}^\infty \frac{(2p\psi_j(\xx))^{2n+1}}{(2n+1)!(2n+2+q_j)} \\
&\le \frac{1}{(q_j+1)!}\sinh(2p|\psi_j(\xx)|) \le \frac{2pb_{1,j}}{(q_j+1)!} \frac{\sinh(2p\psi_j(\xx))}{2p\psi_j(\xx)}.
\end{align*}
Hence, defining the function
\[
  f(q_j) = \begin{cases} \frac{1}{q_j!} & \text{for $q_j$
      even},\\[1mm] \frac{2pb_{1,j}}{(q_j+1)!} & \text{for $q_j$
      odd}, \end{cases}
\]
we have
\begin{align*}
   \E{(\partial_{x_i}a(\xx,\yy))^{2p}} &\le \sum_{\qq\in\nset^{\nset} \atop |\qq|=2p} (2p)! \prod_{j=1}^\infty b_{1,j}^{q_j} f(q_j) \frac{\sinh(2p\psi_j(\xx))}{2p\psi_j(\xx)}
= \theta_0(2p;\xx)\sum_{\qq\in\nset^{\nset} \atop |\qq|=2p} (2p)! \prod_{j=1}^\infty b_{1,j}^{q_j} f(q_j)\\
&\le \theta_0(2p;\xx)\sum_{\qq\in\nset^{\nset} \atop |\qq|=2p, \text{$\qq$ even}} (2p)! (1+2p)^{|\qq|_0} \prod_{j=1}^\infty \frac{b_{1,j}^{q_j}}{q_j!} \\
&\le (1+2p)^{p} \theta_0(2p;\xx)\sum_{\qq\in\nset^{\nset} \atop |\qq|=p} (2p)!  \prod_{j=1}^\infty \frac{b_{1,j}^{2q_j}}{(2q_j)!} \\
&\le (1+2p)^{p} (2p)^p \theta_0(2p;\xx)\sum_{\qq\in\nset^{\nset} \atop
  |\qq|=p} p!  \prod_{j=1}^\infty \frac{(b_{1,j}^2)^{q_j}}{q_j!} =
  (1+2p)^{p} (2p)^p \theta_0(2p;\xx) \left(\sum_{j\geq 1} b_{1,j}^2\right)
\end{align*}
from which we see that $\E{(\partial_{x_i}a(\xx,\yy))^{2p}}$ is bounded for any $0\le p<\infty$ and any $\xx\in\mathscr{B}$ if $\bb_1\in \ell^2$.
\end{proof}

\subsection{ $L^p(\Gamma)$ summability, uniform in space}

Assuming now that $\bb_s\in\ell^2$ so that the random field, $a$, is
$s$-times differentiable in an $L^p(\Gamma)$ sense according to
Lemma~\ref{lemma:CsLp}, we show that this implies some uniform
$L^p(\Gamma)$ summability as detailed in the next lemma.
\begin{lemma} \label{lem:Wa_inf_bound}
Let $s\in\nset_+$. If $\bb_s\in\ell^2$ then
$\E{\|a\|^p_{W^{\upsilon,\infty}(\mathscr{B})}}<\infty$
for all $1\le p<\infty$ and $\upsilon<s$.
\end{lemma}
\begin{proof}
  We exploit the Sobolev embedding,
  $W^{\upsilon+\frac{d}{2q},2q}(\mathscr{B}) \subseteq
  W^{\upsilon,\infty}(\mathscr{B})$, for all $\upsilon\ge 0$ and
  $q\ge 1$. %
  For $q\ge \max\{d/2(s-\upsilon),p/2\}$, we have
  \begin{align*}
    \E{\|a\|^{p}_{W^{\upsilon,\infty}(\mathscr{B})}} &\le \E{\|a\|^{2q}_{W^{s-\frac{d}{2q},\infty}(\mathscr{B})}}  \lesssim \E{\|a\|^{2q}_{W^{s,2q}(\mathscr{B})}} = \E{\sum_{|\ii|\le s} \int_{\mathscr{B}} (D^\ii a(\xx))^{2q}d\xx} \\
&= \sum_{|\ii|\le s} \int_{\mathscr{B}} \E{(D^\ii a(\xx))^{2q}}d\xx < \infty,
\end{align*}
where the last term is bounded from Lemma~\ref{lemma:CsLp}.
\end{proof}
Now, we directly observe by taking $\upsilon =0$ in the previous result that $a_{\max} = \|a\|_{L^\infty(\mathscr{B})}$ has bounded moments,
\[\E{a_{\max}^p}
<\infty,
\]
for all $1\le p<\infty$ and $0<s$. Finally, by observing that, due to
the construction \eqref{eq:kappa_and_a} in
Assumption~\ref{assump:b_and_p}, we have that
$a_{\min} = \frac{1}{\|a^{-1}\|_{L^\infty(\mathscr{B})}}$ has the same
distribution as $a_{\max}$. As a consequence, $a_{\min}$ has bounded
moments as well. This implies in turn that \eqref{eq:bounding_a} holds
and thus problem \eqref{eq:PDE1} is well posed in the Bochner space,
$L^p\left(\Gamma;H^1_0(\mathscr{B})\right)$. That is,

\begin{corollary}[Well posedness with log uniform
  coefficient]\label{cor:well_being}
We have for $0<\nu$ that the problem in Example~\ref{ex:logunif} is well posed in the Bochner space $L^p\left(\Gamma;H^1_0(\mathscr{B})\right)$.
The corresponding solution, $u$, satisfies
$$
\|u\|_{L^p(\Gamma;H^1_0(\mathscr{B}))} \le C \E{\frac{1}{a_{\min}^p}}^{1/p} \|f\|_{H^{-1}(\mathscr{B})}.
$$
\end{corollary}

We observe that higher regularity of the solution, $u$, can be
obtained by using larger values of $s$ in
Lemma~\ref{lem:Wa_inf_bound}.  This in turn yields control on moments
of $W^{\upsilon,\infty}(\mathscr{B})$ norms of the coefficient, $a$,
and following, for instance, estimates similar to (2.10) in
\cite[Theorem 2.4]{graham.scheichl.mixed_MLMC}, we can estimate
moments of the $H^{1+s}(\mathscr{B})$ norm of the solution, $u$. These
regularity estimates, once combined with pathwise error estimates for
the combination technique, can be further used to show the
corresponding $\nu$-dependent convergence rates of MIMC
\cite{abdullatif.etal:MultiIndexMC} for Example~\ref{ex:logunif},
similar to what was presented in Section~\ref{s:example_analysis} for
MISC in the current work.

\section{Shift theorem for problem \eqref{eq:PDE1}}\label{sec:app-regularity}
Here, we seek to establish a \emph{shift theorem} for the problem
\begin{equation}\label{eq:PDE2}
  \begin{cases}
    - \div (a(\xx)\,\nabla u(\xx)) = \forcing(\xx) & \mbox{in}\quad \mathscr{B}=[0,1]^D  \\
    u(\xx) \,=\, 0 & \mbox{on}\quad \partial \mathscr{B},
  \end{cases}
\end{equation}
under suitable assumptions on $a$ and $\forcing$.

With respect to problem \eqref{eq:PDE1}, for convenience, we drop the
dependence on the parameter vector, $\yy$.
We consider an odd periodic extension of $\forcing$, on $[-1,1]^D$, and
an even periodic extension of the coefficient $a$ on $[-1,1]^D$,
named, respectively, $\tilde{\forcing}$ and $\tilde{a}$. More precisely, for $\jj=\{0,1\}^D$, we denote by $\xx_\jj = ((-1)^{j_1}x_1,\ldots,(-1)^{j_D}x_D)$ and
\[
  \tilde{\forcing}(\xx_\jj+2\kk) = (-1)^{|\jj|}\forcing(\xx), \qquad  \tilde{a}(\xx_\jj+2\kk)=a(\xx), \qquad \forall \xx\in [0,1]^2, \;\; \jj\in\{0,1\}^D, \;\; \kk\in\nset^D.
\]
The following Shift theorem holds for problem \eqref{eq:PDE2}.
\begin{lemma}
If the coefficient $a$ is such that its periodic extension satisfies
$\tilde{a}\in W^{s,\infty}(\rset^D)$, $s\ge 0$ and $\forcing \in C^\infty_0(\mathscr{B})$ then $u\in H^{s+1}(\mathscr{B})$.
\end{lemma}
\begin{proof}
We define the extended problem
\[
  \begin{cases}
    - \div (\tilde{a}(\xx)\,\nabla \tilde{u}(\xx)) = \tilde{\forcing}(\xx) & \mbox{in}\quad \tilde{\mathscr{B}}=[-1,1]^D  \\
    \int_{\tilde{\mathscr{B}}} u(\xx) \,=\, 0 \\
  \text{periodic boundary conditions on $\partial\tilde{\mathscr{B}}$}.
  \end{cases}
\]
Since by assumption $\tilde{a}\in L^\infty(\rset^D)$ and
$\tilde{\forcing}\in L^2(\tilde{\mathscr{B}})$, this problem has a
unique solution,
$\tilde{u}\in H^1_{per}(\tilde{\mathscr{B}})\setminus\rset$, where we
denote with $H^s_{per}(\tilde{\mathscr{B}})$ the space of periodic
functions with (periodic) square integrable derivatives up to order
$s$. It is easy to check that the solution $\tilde{u}$ is odd, that is
$\tilde{u}(\xx_\jj)=(-1)^{\jj}\tilde{u}(\xx)$,
$\forall \xx\in[0,1]^D$, hence $\tilde{u}=0$ (in the sense of traces)
on $\partial\mathscr{B}$ and it coincides with the (unique) solution
of \eqref{eq:PDE2} on $\mathscr{B}$.  Moreover, standard elliptic
regularity arguments%
allow us to say that $\tilde{u}\in H^s_{per}(\tilde{\mathscr{B}})$, hence
$u\in H^s(\mathscr{B})$.
\end{proof}

\bibliographystyle{src/focm}

\end{document}